\documentclass[12pt]{amsart}
\usepackage[margin=3cm, marginpar=3cm]{geometry}

\usepackage{amssymb}
\usepackage{mathtools}
\usepackage{tikz}
\usepackage{tikz-cd}

\usepackage{braket}
\usepackage{mathrsfs}
\usepackage{ifthen}
\usepackage{graphicx}

\usepackage{comment} 
\usepackage{mleftright}
\usepackage{hyperref}
\usepackage{cleveref}
 \usepackage[all]{xy}
 \usepackage{amscd}
 \usepackage{color}
 \usepackage{enumitem}
\newlist{steps}{enumerate}{1}
\setlist[steps, 1]{label = Step \arabic*:}
\usepackage[fontsize=10]{scrextend}
\usepackage{bm}
\usepackage{afterpage}

\usepackage[
    backend=biber,
    style=alphabetic,
    sorting=nyt,
    doi=false,
    url=false,
    isbn=false,
    maxbibnames=10,
]{biblatex}

\usetikzlibrary{knots}

\DeclareFieldFormat[misc]{title}{``#1"}

\makeatletter
\DeclareRobustCommand\widecheck[1]{{\mathpalette\@widecheck{#1}}}
\def\@widecheck#1#2{%
   \setbox\z@\hbox{\m@th$#1#2$}%
   \setbox\tw@\hbox{\m@th$#1%
      \widehat{%
         \vrule\@width\z@\@height\ht\z@
         \vrule\@height\z@\@width\wd\z@}$}%
   \dp\tw@-\ht\z@
   \@tempdima\ht\z@ \advance\@tempdima2\ht\tw@ \divide\@tempdima\thr@@
   \setbox\tw@\hbox{%
      \raise\@tempdima\hbox{\scalebox{1}[-1]{\lower\@tempdima\box\tw@}}}%
   {\ooalign{\box\tw@ \cr \box\z@}}}
\makeatother

\theoremstyle{plain}
\newtheorem{thm}{Theorem}[section]
\crefname{thm}{Theorem}{Theorems}
\Crefname{thm}{Theorem}{Theorems}
\newtheorem{prop}[thm]{Proposition}
\crefname{prop}{Proposition}{Propositions}
\Crefname{prop}{Proposition}{Propositions}
\newtheorem{lem}[thm]{Lemma}
\crefname{lem}{Lemma}{Lemmas}
\Crefname{lem}{Lemma}{Lemmas}
\newtheorem{cor}[thm]{Corollary}
\crefname{cor}{Corollary}{Corollaries}
\Crefname{cor}{Corollary}{Corollaries}

\crefname{claim}{Claim}{Claims}
\Crefname{claim}{Claim}{Claims}

\crefname{property}{Property}{Properties}
\Crefname{property}{Property}{Properties}

\crefname{problem}{Problem}{Problems}
\Crefname{problem}{Problem}{Problems}
\newtheorem{ques}[thm]{Question}
\crefname{ques}{Question}{Questions}
\Crefname{ques}{Question}{Questions}

\theoremstyle{definition}
\newtheorem{defn}[thm]{Definition}
\crefname{defn}{Definition}{Definitions}
\Crefname{defn}{Definition}{Definitions}

\crefname{notation}{Notation}{Notations}
\Crefname{notation}{Notation}{Notations}

\crefname{convention}{Convention}{Conventions}
\Crefname{convention}{Convention}{Conventions}

\crefname{cond}{Condition}{Conditions}
\Crefname{cond}{Condition}{Conditions}

\crefname{assum}{Assumption}{Assumptions}
\Crefname{assum}{Assumption}{Assumptions}

\crefname{conj}{Conjecture}{Conjectures}
\Crefname{conj}{Conjecture}{Conjectures}

\theoremstyle{remark}
\newtheorem{rem}[thm]{Remark}
\crefname{rem}{Remark}{Remarks}
\Crefname{rem}{Remark}{Remarks}

\crefname{ex}{Example}{Examples}
\Crefname{ex}{Example}{Examples}

\crefname{section}{Section}{Sections}
\Crefname{section}{Section}{Sections}
\crefname{subsection}{Subsection}{Subsections}
\Crefname{subsection}{Subsection}{Subsections}
\crefname{figure}{Figure}{Figures}
\Crefname{figure}{Figure}{Figures}

\newtheoremstyle{restate-thm}
    {\topsep}{\topsep} 
    {\itshape}         
    {}                 
    {\bfseries}        
    {.}                
    { }                
    {\thmname{#1} \ref{#3} {\normalfont(Restated)}}
    
\theoremstyle{restate-thm}
    \newtheorem{restate-thm}{Theorem}
    \newtheorem{restate-prop}{Proposition}

\AddToHook{env/prop/begin}{\crefalias{thm}{prop}} 
\AddToHook{env/lem/begin}{\crefalias{thm}{lem}}
\AddToHook{env/cor/begin}{\crefalias{thm}{cor}}
\AddToHook{env/claim/begin}{\crefalias{thm}{claim}}
\AddToHook{env/ques/begin}{\crefalias{thm}{ques}}
\AddToHook{env/defn/begin}{\crefalias{thm}{defn}}

\newcommand{\Z}{\mathbb{Z}}

\newcommand{\Q}{\mathbb{Q}}

\newcommand{\calG}{{\mathcal G}}

\newcommand{\Frac}{\mathrm{Frac}}

\newcommand{\Ker}{\mathop{\mathrm{Ker}}\nolimits}

\newcommand{\im}{\operatorname{Im}}
\newcommand{\rank}{\mathop{\mathrm{rank}}\nolimits}
\newcommand{\Hom}{\mathop{\mathrm{Hom}}\nolimits}

\newcommand{\id}{\mathrm{id}}

\newcommand{\C}{\mathbb{C}}

\newcommand{\R}{\mathbb R}

\newcommand{\ctext}[1]{\raise0.2ex\hbox{\textcircled{\scriptsize{#1}}}}

\def\ker{\operatorname{Ker}}

\def\dim{\operatorname{dim}}
\def\rank{\operatorname{rank}}

\def\Hom{\operatorname{Hom}}

\def\id{\operatorname{Id}}

\newcommand{\mbar}[1]{{\ooalign{\hfil#1\hfil\crcr\raise.167ex\hbox{--}}}}

\def\bd{\mathbf{d}}

\renewcommand{\a}{\mathfrak{a}}
\renewcommand{\b}{\mathfrak{b}}

\def\gr{\operatorname{gr}}

\def\bu{\mathbf{u}}

\def\wt{\widetilde}

\def\bbF{\mathbb{F}}
\DeclareMathOperator{\Tor}{Tor}

\newcommand{\CKh}{\mathit{CKh}}
\newcommand{\Kh}{\mathit{Kh}}
\newcommand{\rCKh}{\widetilde{\CKh}{}}

\newcommand{\Cob}{\mathrm{Cob}}

\DeclareMathOperator{\Mdeg}{Mdeg}

\newcommand{\low}{\mathrm{low}}
\newcommand{\bal}{\mathrm{bal}}
\newcommand{\hi}{\mathrm{hi}}

\def\bv{\mathbf{v}}

\renewcommand{\emptyset}{\varnothing}

\title{Cobordism maps in Khovanov homology and singular instanton homology II}
\author{Hayato Imori, Taketo Sano, Kouki Sato, and Masaki Taniguchi}

\newcommand{\printaddresses}{{
  \setlength{\parindent}{0pt}
  \bigskip
  Hayato Imori, \textsc{Department of Mathematical Sciences, KAIST, Daejeon, Republic of Korea}\\
  \textit{E-mail address}: \url{himori@kaist.ac.kr}\par
  \medskip
  Taketo Sano, \textsc{RIKEN iTHEMS, Wako, Saitama 351-0198, Japan }\\
  \textit{E-mail address}: \url{taketo.sano@riken.jp}\par
  \medskip
  Kouki Sato, \textsc{Meijo University, Tempaku, Nagoya 468-8502, Japan}\\
  \textit{E-mail address}: \url{satokou@meijo-u.ac.jp}\par
  \medskip
  Masaki Taniguchi, \textsc{Department of Mathematics, Graduate School of Science, Kyoto University, Kitashirakawa Oiwake-cho,
Sakyo-ku, Kyoto 606-8502, Japan}\par
  \textit{E-mail address}: \url{taniguchi.masaki.7m@kyoto-u.ac.jp}
}}

\begin{document}

\begin{abstract}
This paper is a continuation of our previous work \cite{ISST1}, where we showed that the embedded cobordism maps on Khovanov homology and singular instanton homology are compatible under Kronheimer and Mrowka's spectral sequence. 
In this paper, we extend this construction to immersed cobordisms, where we define an immersed cobordism map on Khovanov homology and prove that it is compatible with the immersed cobordism map on singular instanton homology. We give two applications: (i) For any smooth, oriented concordance $C$ from a two-bridge torus knot, the induced map $\mathit{Kh}(C)$ on Khovanov homology is injective, and its left inverse is given by the reversal of $C$. (ii) Any pair of relatively exotic surfaces in $D^4$ that are detected by the embedded cobordism map in $\mathit{Kh}$ remain exotic even after applying any number of positive twist moves. 
\end{abstract}
\maketitle

\setcounter{tocdepth}{2}

\section{Introduction}
\label{sec:intro}

\subsection{A Khovanov--Floer type statement for immersed cobordisms}

In \cite{KM11u}, Kronheimer and Mrowka introduced the \textit{singular instanton knot Floer homology} for a link $L$, and constructed a spectral sequence having \textit{Khovanov homology} (over $\Z$) as its $E^2$ term and abutting to singular instanton knot Floer homology\footnote{
    $L^*$ denotes the mirror of $L$. This is necessary from conventional reasons.
}
\[
    \Kh (L^*)  \Rightarrow I^\sharp (L),
\]
which led to the proof that \textit{Khovanov homology detects the unknot}. The spectral sequence is obtained from the \textit{instanton cube complex} $\CKh^\sharp (L)$ whose homology gives singular instanton knot Floer homology $I^\sharp(L)$, together with the \textit{instanton homological filtration} such that the Khovanov complex $\CKh(L^*)$ naturally arise in the $E^1$ term of the induced spectral sequence. Lately in \cite{KM14}, Kronheimer and Mrowka also introduced the \textit{instanton quantum filtration} on $\CKh^\sharp(L)$ and proved that Khovanov homology $\Kh(L^*)$ also arise in the $E^1$ term of the induced spectral sequence. 

In \cite{ISST1}, the authors constructed an embedded cobordism map for the instanton cube complex that recovers the cobordism maps both in Khovanov homology and singular instanton theory. 
Namely, let $L, L'$ be links with diagram $D, D'$, and $S$ a link cobordism in $[0,1]\times \R^3$ between $L$ and $L'$, represented as a sequence of elementary moves. Then there is a doubly filtered chain map between the instanton cube complexes: 
\[
    \phi^\sharp_S : \CKh^\sharp (D) \to \CKh^\sharp (D')
\]
of order
\[
    \geq \left(\frac{1}{2}(S \cdot S),\ \chi(S)+\frac{3}{2}(S \cdot S)\right)
\]
whose induced map on the $E^2$ term with respect to the homological filtration (resp.\ the $E^1$ term with respect to the quantum filtration) coincides up to sign with the cobordism map of Khovanov homology\footnote{
    $S^*$ denotes the image of $S$ under $\id \times r$, where $r$ is the reflection on $\R^3$ that gives the mirroring of links. 
}
\[
    \Kh(S^*): \Kh(L^*) \to \Kh(L'^*)
\]
and whose induced map on homology coincides with the cobordism map of the singular instanton knot Floer homology
\[
    I^\sharp(S) : I^\sharp (L) \to I^\sharp (L').
\]

In this paper, we extend this construction to \textit{normally immersed cobordisms}. Here, a \textit{normally immersed cobordism} $S$ between (oriented) links $L, L'$ is a smoothly immersed (possibly non-orientable) surface $S$ in $[0,1]\times \R^3$ with boundary $\{0\} \times L \sqcup \{1\} \times L'$, that has only transverse double points. If $S$ is oriented, then we also assume that it respects the orientations of the boundary links $L, L'$. 
Suppose we are given such normally immersed link cobordism $S\colon L \to L'$. Starting from Kronheimer's work \cite{kronheimer1997obstruction} on singular Donaldson theory, the blowing-up construction has been used to define immersed cobordism maps. 
In particular, we focus on Kronheimer--Mrowka's immersed cobordism map \cite{KM11}, 
\[
I^\sharp(S) : I^\sharp(L) \to I^\sharp(L').
\]
On the Khovanov side, we first give a combinatorial description of the \textit{immersed cobordism map of lowest homological degree}
\[
    \Kh^{\low}(S) : \Kh(L) \to \Kh(L'),
\]
and then prove that the above two maps are compatible under Kronheimer and Mrowka's spectral sequence. Namely,

\begin{thm}
\label{mainKF}
\label{thm:induced-map}
For any links $L, L'$ with diagrams $D, D'$, and normally immersed cobordism $S$ between $L$ and $L'$, there exists a doubly filtered chain map between the instanton cube complexes
\[
    \phi^\sharp_S:   \CKh^\sharp(D) \to \CKh^\sharp(D')
\]
of order
\[
    \geq \left(-2s_+ + \frac{1}{2} (S\cdot S),\ 
    \chi(S) + \frac{3}{2} (S\cdot S) - 6s_+ \right)
\] 
whose induced map on $E^2$-term with respect to the $h$-filtration coincides up to sign with 
\[
    \Kh^{\low}(S^*) : \Kh (L^*) \to \Kh(L'{}^*),
\] 
and whose induced map on homology coincides with the cobordism map of singular instanton knot Floer homology
\[
    I^\sharp(S) : I^\sharp (L) \to I^\sharp (L').
\]
Here, $\chi(S)$ denotes the Euler characteristic, $S \cdot S$ the normal Euler number (precisely defined in \cref{normal_Euler}), and $s_+$ the number of positive double points of $S$.  
\end{thm}

\begin{rem}
    In the first version of this paper, there was an analogous statement of \Cref{thm:induced-map} for the reduced theory. However, there was a gap in the construction of the map, hence has been removed. The later stated \Cref{thm:left-inverse,thm:T2q} have been modified accordingly for the unreduced theory.
\end{rem}

\subsection{Immersed cobordism maps on Khovanov homology}

The immersed cobordism map $\Kh^{\low}(S)$ introduced above is a special case of a more general construction we give in this paper. We define various crossing change maps and immersed cobordism maps for Khovanov homology and its deformed versions over an arbitrary ring $R$. Here, we let $\Kh_{h, t}$ denote the generalized Khovanov homology over $R$, where $h, t$ are fixed elements in $R$. The special case $h = t = 0$ gives the original theory. 

In previous works, various crossing change maps in Khovanov homology and its variants has been defined; by Alishahi~\cite{Alishahi:2017}, Alishahi--Dowlin~\cite{Alishahi:2018} and Ito--Yoshida~\cite{Ito-Yoshida:2021}. We shall see in \Cref{sec:kh-x-ch} that each of these maps corresponds to an element of $\Kh_{h, t}(H)$, i.e.\ the Khovanov homology of the Hopf link $H$, and conversely, any element of $\Kh_{h, t}(H)$ give rise to a crossing change map.\footnote{
    To be precise, we additionally need to specify a \textit{direction} to determine the crossing change map. However, if $z$ is homogeneous, directions will only contribute to signs. This is also true for immersed cobordism maps. 
}

Crossing change maps can be composed with the standard embedded cobordism map to give immersed cobordism maps. To be precise, suppose $S$ is a normally immersed cobordism with $s_+$ positive, $s_-$ negative double points. Further suppose that $S$ is decomposed into elementary pieces, each of which contains at most one double point. Then, together with an $(s_+ + s_-)$-tuple of elements in $\Kh_{h, t}(H)$
\[
    \mathbf{z} = (z^+_1, \ldots, z^+_{s_+}; z^-_1, \ldots, z^-_{s_-}),
\]
we obtain an immersed cobordism map 
\[
    \Kh_{h, t}(S; \mathbf{z})\colon \Kh_{h, t}(L) \to \Kh_{h, t}(L')
\] 
by the composition of elementary cobordism maps and the corresponding crossing change maps. In particular, let $\zeta_0, \zeta_1$ be two generators of $\Kh_{h, t}(H)$ whose (relative) bigradings are given by $(0, 2)$ and $(2, 6)$ respectively (see \Cref{fig:zeta01} in \Cref{subsect: Geometric description}). The immersed cobordism map of \textit{lowest homological degree} $\Kh^{\low}_{h, t}(S)$ is given by choosing $\zeta_0$ for all double points,
\[
    \Kh^{\low}_{h, t}(S) = \Kh_{h, t}(S; \zeta_0, \ldots, \zeta_0; \zeta_0, \ldots, \zeta_0).
\]
Another natural choice would be to choose $\zeta_0$ for all positive double points, and $\zeta_1$ for all negative double points, 
\[
    \Kh^{\bal}_{h, t}(S) = \Kh_{h, t}(S; \zeta_0, \ldots, \zeta_0; \zeta_1, \ldots, \zeta_1)
\]
which is called the immersed cobordism map of \textit{balanced homological degree}. The name comes from the fact that, when $S$ is oriented, the map has homological degree $0$. In general, we say $\Kh_{h, t}(S; \mathbf{z})$ is \textit{homogeneous} if each element $z^\pm_i$ is homogeneous in $\Kh_{h, t}(H)$. Although there are various choices of immersed cobordism maps, that $\Kh^{\low} = \Kh^{\low}_{0, 0}$ particularly appears in \Cref{thm:induced-map} is due to the fact that the $E^1$-term with respect to the $h$-filtration only captures the lowest homological degree part of a possibly non-homogeneous map. 

S.\ Carter, B.\ Cooper, M.\ Khovanov, and V.\ Krushkal have been independently studying immersed cobordism maps in Khovanov homology, and their paper \cite{CCKK:2025-immersed} appeared soon after we posted the first version of this paper. They give another homogeneous immersed cobordism map on Khovanov homology, which can be described using our notation as
\[
    \Kh_{0, 0}(S; \zeta_1, \ldots, \zeta_1; \zeta_0, \ldots, \zeta_0).
\]
Furthermore, they prove that the map is invariant up to sign under smooth isotopies of $S$ rel boundary, by extending Carter and Saito's \textit{movie moves} to immersed surfaces \cite[Theorem 4.1]{CCKK:2025-immersed}, and proving that the above map is indeed invariant under the additional moves. 

Since \cite[Theorem 4.1]{CCKK:2025-immersed} is a general result, it can be applied to our construction to prove that any homogeneous immersed cobordism map in Khovanov homology is invariant up to sign under smooth isotopies, thus extending the main theorem of \cite{CCKK:2025-immersed}. 

\begin{thm}
\label{thm:isotopy-invariance}
    Suppose $S$ is a normally immersed cobordism between links $L$, $L'$ with $s_+$ positive, $s_-$ negative double points. Given any $(s_+ + s_-)$-tuple $\mathbf{z}$ of homogeneous elements of $\Kh_{h, t}(H)$, the immersed cobordism map
    \[
        \Kh_{h, t}(S; \mathbf{z})\colon \Kh_{h, t}(L) \to \Kh_{h, t}(L')
    \]
    is invariant up to sign under smooth isotopies of $S$ rel boundary, and is projectively functorial with respect to composition of cobordisms. Furthermore, if we restrict to oriented immersed cobordisms, then the sign of the map $\Kh^{\bal}_{h, t}(S)$ can be fixed, so that it is strictly invariant under isotopies of $S$ and hence is strictly functorial. 
\end{thm}

\subsection{Applications}

Hereafter, we only consider oriented link cobordisms. 

\subsubsection{Injectivity from concordance}
Gordon \cite{Gor81} proposed a conjecture stating that the existence of a {\it ribbon concordance} defines a partial order on the set of isotopy classes of knots. Here, a ribbon concordance is a smooth concordance 
\(
    S\colon K \to K'
\)
in \([0,1]\times \R^3\) without local maxima with respect to the projection 
\([0,1]\times \R^3 \to [0,1]\).
Recently, Agol \cite{agol2022ribbon} proved this conjecture using \(SO(n)\)-character varieties of knots. 
Here, we write \(K \leq K'\) if there exists a ribbon concordance from \(K\) to \(K'\). The following question is still open.

\begin{ques}[\text{\cite[Question~6.1]{Gor81}}]\label{ques:ribbon}
    Let \(K_0\) be a minimal element with respect to \(\leq\).
    If a knot \(K\) is smoothly concordant to \(K_0\), does it follow that
    $
    K_0 \leq K$ ? 
\end{ques}

In particular, when \(K_0\) is the unknot, \cref{ques:ribbon} is precisely the {\it slice--ribbon conjecture}. 
Note that any torus knot is known to be minimal \cite{Gor81}. 
When \(K_0\) is a torus knot, an analogous question has been raised in \cite{DS20}, and affirmative evidence has been provided in terms of $SU(2)$-character varieties in \cite{DS20, imori2024instanton}. Also, see \cite{AT24} for a partial answer to \cref{ques:ribbon} for torus knots.  

From the perspective of knot homology theories, it has been shown that a ribbon concordance induces injections on various knot homologies, whose left inverses are given by the reversed cobordism. Examples of such homology theories include Heegaard--Floer knot homology, Khovanov homology, and singular instanton knot homology 
\cite{zemke2019knot, levine2019khovanov, daemi2022ribbon, kang2022link}. 
On the other hand, in (equivariant) knot instanton Floer theory, Daemi and Scauto proved that such \textit{injectivity property}---a smooth concordance from a specific knot induces an injective map, together with a specific left inverse---holds in some case even when the concordance is not ribbon \cite[Theorem 4.45]{DS20}. 

Here, we combine \cref{thm:induced-map} with the results of \cite{DS20,daemi2022instantons} to prove that such an injectivity property holds in Khovanov homology, for arbitrary concordances starting from negative two-bridge torus knots, which can be seen as an algebraic affirmative evidence of \cref{ques:ribbon}. 

\begin{thm}\label{thm:left-inverse}
For any smooth knot concordance $C$ from a negative two-bridge torus knot, the induced map on $\Kh$ is injective, with a left inverse given by the reversal of \(C\).
\end{thm}

\begin{rem}
    \Cref{thm:left-inverse} can be extended to any generalized Khovanov homology $\Kh_{h,t}$ over a PID $R$, including Bar-Natan homology and Lee homology (\Cref{thm:left-inverse-general}). Moreover, an analogous statement on $I^\sharp$ also holds (\Cref{thm:left-inverse-sharp}).
\end{rem}

The proof of \cref{thm:left-inverse} is based on the following structure theorem on the Khovanov homology for immersed cobordism maps from the unknot $U_1$ and the negative trefoil $T^*_{2, 3}$ to a negative two-bridge torus knot. 

\begin{thm}
\label{thm:T2q}
    Let $k$ be a positive integer and $s_-$ an integer such that $0 \leq s_- \leq k$. 
    \begin{enumerate}
        \item Let $U_1$ denote the unknot, and $S$ be an immersed cobordism from $U_1$ to $T^*_{2, 2k + 1}$ with genus $g = k - s_-$ and $s_-$ negative double points. Then the immersed cobordism map on Khovanov homology
        \[
            \Kh^{\low}(S)\colon \Kh(U_1) \to \Kh(T^*_{2, 2k + 1})
        \]
        restricts to a surjection
        \[
            \Kh^{\low}(S)\colon \Kh^0(U_1) \cong \Z^2 \to \Kh^{-2s_-}(T^*_{2, 2k + 1}). 
        \]

        \item Let $S$ be an immersed cobordism from $T^*_{2, 3}$ to $T^*_{2, 2k + 1}$ with genus $g = k - s_- - 1$ and $s_-$ negative double points. Then the immersed cobordism map on Khovanov homology
        \[
            \Kh^{\low}(S)\colon \Kh(T^*_{2, 3}) \to \Kh(T^*_{2, 2k + 1})
        \]
        restricts to a bijection
        \[
            \Kh^{\low}(S)\colon \Kh^{-3}(T^*_{2, 3}) \cong \Z \to \Kh^{-2s_- - 3}(T^*_{2, 2k + 1}).
        \]
    \end{enumerate}
\end{thm}

Note that the homological grading of $\Kh(T^*_{2, 2k + 1})$ ranges from $-2k - 1$ to $0$. Thus \Cref{thm:T2q} implies that $\Kh(T^*_{2, 2k + 1})$ is covered by the images of immersed cobordism maps from $U_1$ and $T^*_{2, 3}$, by varying the genus and the number of negative double points of $S$. 


It is interesting to ask whether analogous statements of \Cref{thm:left-inverse} hold for other knot homology theories. 

\begin{ques}
    Given a knot homology theory, for which classes of knots do analogous statements of \Cref{thm:left-inverse} hold?  
\end{ques}
For singular instanton theory, an injectivity holds for any concordance starting from any torus knot, which is again proven in \cite{DS20}. Also see \cite{imori2024instanton} for the corresponding statements for singular instanton theory with general holonomy parameters. \footnote{We do not know whether a left inverse can be obtained as its reversal cobordism. }
In instanton theory, there are two important classes of knots that contain the torus knots: {\it instanton L-space knots} and {\it I-basic knots}. 
\textit{Instanton L-space knots} were introduced in \cite[Definition 1.13]{baldwin2022instantons} as an instanton counterpart of the L-space knots in Heegaard--Floer knot homology, and is conjectured that the two are equal. The \textit{I-basic knots} were introduced in \cite[Section 1.4]{daemi2024unoriented} as a natural generalization of torus knots in terms of the behavior of equivariant singular instanton homology. Thus, it is interesting to ask: 
\begin{ques}
    Do analogous statements of \Cref{thm:left-inverse} hold for instanton L-space knots and I-basic knots with framed singular instanton Floer homology \cite{KM11}, sutured instanton Floer homology \cite{kronheimer2010knots}, and equivariant singular instanton Floer homology \cite{DS19}? 
\end{ques}

After posting the first version of this paper, Zhenkun Li personally informed us that the analogous statement of \Cref{thm:left-inverse} holds for L-space knots in Heegaard--Floer knot homology, see \cite[Theorem 1.4 and 1.7]{Z19}.

Another possible way to prove such a statement would be to combine \Cref{thm:left-inverse} with spectral sequences from Khovanov homology to various homology theories, including: the (involutive) Heegaard Floer homology of branched covers \cite{OS05, alishahi2023khovanov}, the plane Floer homology \cite{Da15}, the (involutive) monopole Floer homology of branched covers \cite{B11, lin2019bar}, the framed instanton homology of branched covers \cite{Sca15}, the Heegaard--Floer knot homology \cite{Do24, nahm2025spectral} and the real monopole Floer homology \cite{Li24}. For a formal treatment of cobordism maps of such spectral sequences, see \cite{BHL19}. 

\subsubsection{Detection of exotic slice surfaces}

It was shown in \cite[Section~1.6, Proposition]{freedman2014topology} that any two smoothly and normally immersed surfaces in $[0,1]\times \R^3$ that are homotopic rel boundary can be related by a finite sequence of smooth ambient isotopies rel boundary, together with the following three moves and their inverses:
\begin{itemize}
    \setlength{\itemsep}{.25em}
    \item \emph{Positive twist move}, which locally introduces a positive double point;
    \item \emph{Negative twist move}, which locally introduces a negative double point;
    \item \emph{Finger move}, which locally introduces a canceling pair of negative and positive double points.
\end{itemize}
Here, we consider Khovanov homology $\Kh_{h, t}$ in its generalized form, and describe the behaviors of the two specific immersed cobordism maps $\Kh^{\low}_{h, t}$ and $\Kh^{\bal}_{h, t}$ under these moves. 
\begin{prop}
\label{formula_of_moves}
    Let $S$ be a normally immersed link cobordism in $[0,1]\times \R^3$ from $L$ to $L'$, and $S'$ be another cobordism obtained from $S$ by applying one of the above three moves. 
    \begin{itemize}
        \item If $S'$ is obtained from $S$ by adding a positive twist, then
        \begin{align*}
            \Kh^{\low}_{h, t}(S) &= \Kh^{\low}_{h, t}(S'),\\
            \Kh^{\bal}_{h, t}(S) &= \Kh^{\bal}_{h, t}(S').            
        \end{align*}
        \item If $S'$ is obtained from $S$ by adding a negative twist or by applying a finger move, then 
        \[
            \Kh^{\low}_{h, t}(S') = 0.
        \]
        Moreover, if $\Kh^{\bal}_{h, t}(S)$ is nonzero, then $\Kh^{\bal}_{h, t}(S')$ is also generally nonzero, with its degree shifted by $(0, -2)$. (Detailed descriptions are given in \Cref{prop:twist-move-map,prop:finger-move-map}.)
    \end{itemize}
\end{prop}
The above proposition is proved within Khovanov homology and independent from instanton theory. Analogous statements have been proven for singular Donaldson polynomial invariants \cite{kronheimer1997obstruction}, singular instanton Floer theory \cite{KM13, DS20, imori2024instanton} and $\Z_2$-equivariant Seiberg--Witten Floer theory \cite{iida2024monopoles}. 

Recently, there have been various studies of relatively exotic surfaces using link cobordism maps in Khovanov theory, see \cite{Sundberg-Swan:2022,hayden2024khovanov, banerjee2025spanning, hayden2025seifert}.  
Here, a pair of proper and smooth surface embeddings $S$ and $S'$ in $D^4$ with $\partial S= \partial S' \subset S^3 = \partial D^4$ is {\it relatively exotic} if there is no smooth isotopy between $S$ and $S'$ rel boundary but there is such a topological isotopy. 
We note that any pair of properly immersed surfaces in $D^4$ with the same boundary knot and with the same genera are homotopic, thus, in particular, any relatively exotic surfaces in $D^4$ are related by the above stated moves. 
The following is immediate from \Cref{formula_of_moves} combined with \Cref{thm:isotopy-invariance}. 
\begin{thm}\label{thm:exotic}
    Suppose $S$ and $S'$ are relatively exotic surfaces in $D^4$ that induce (projectively) distinct maps on Khovanov homology (or any of its deformed versions). Then the immersed cobordism maps, either using $\Kh^\low$ or $\Kh^\bal$, remain distinct under any number of positive twist moves on both surfaces, and hence the surfaces remain relatively exotic.
\end{thm}

We can always replace neighborhoods of immersed points by twice-punctured tori to get an embedded surface, but this process loses the information about their signs of the immersed points. 
Summing trivially embedded tori to an embedded surface is called the {\it inner stabilization} and it is also known as an unknotting operation for surface knots in 4-manifolds \cite{BS16}. For inner stabilizations of surfaces, there are various studies, for example, see \cite{guth2024exotic, hayden2023atomic, guth2023doubled} for exotic surfaces in $D^4$.

\medskip

\noindent
\textbf{Organization}.
The paper is organized as follows. \Cref{sec:kh-x-ch,sec:Kh-structure} are entirely within Khovanov theory. In \Cref{sec:kh-x-ch}, we define crossing change maps and assemble them to define immersed cobordism maps on Khovanov homology. \Cref{thm:isotopy-invariance,thm:exotic} are proved therein. 
In \Cref{sec:Kh-structure}, we describe the Khovanov homology of the negative $(2, 2k+1)$-torus knot $T^*_{2,2k+1}$ and prove the combinatorial part of \Cref{thm:T2q}. 
In Section~4, we construct immersed cobordism maps on instanton cube complexes and prove the Khovanov–Floer type compatibility: the induced map on the $E^2$–term (with respect to the h–filtration) agrees with the Khovanov cobordism map.
In Section~5, using equivariant instanton theory, we compute and constrain the maps for two–bridge torus knots, leading to injectivity with an explicit left inverse for concordances starting at $T^*_{2,q}$.
An appendix collects homological–algebra background on filtered complexes, spectral sequences, filtered maps, and tensor products.

\medskip

\noindent
\textbf{Acknowledgement}. 
We would like to thank JungHwan Park and Zenkun Li for helpful discussions. 
HI was partially supported by the Samsung Science and Technology Foundation (SSTF-BA2102-02) and the Jang Young Sil Fellowship from KAIST.
TS was supported by JSPS KAKENHI Grant Number 23K12982 and academist crowdfunding Project No.\ 121. 
MT was partially supported by JSPS KAKENHI Grant Number 22K13921, and RIKEN iTHEMS Program. 
\section{Immersed cobordism maps on Khovanov homology}
\label{sec:kh-x-ch}

Throughout this paper, we work in the smooth category and assume all objects and maps to be smooth. Knots and links are assumed to be oriented, whereas link cobordisms are not necessarily oriented or orientable. We assume that the reader is familiar with the construction of Khovanov homology and its equivariant versions \cite{Khovanov:2000,BarNatan:2004,Khovanov:2004}.

Let $R$ be a commutative ring with unity and $A_{h, t}$ the Frobenius algebra given by $A_{h, t} = R[X]/(X^2 - hX - t)$ with $h, t \in R$ and $\epsilon(1) = 0,\ \epsilon(X) = 1$. 
For a link diagram $D$, let $\CKh_{h, t}(D)$ denote the Khovanov chain complex of $D$ obtained from the Frobenius algebra $A_{h, t}$, and $\Kh_{h, t}(D)$ its homology. This includes the \textit{universal Khovanov homology} (or the $U(2)$-equivariant Khovanov homology), given by $R = \Z[h, t]$ and $A_{h, t} = R[X]/(X^2 - hX - t))$. Other variants are obtained by specializations of this theory; for example, the original construction of the Khovanov complex \cite{Khovanov:2000} is given by setting $(h, t) = (0, 0)$.

If $R$ is graded and $\deg h = -2,\ \deg t = -4$ (including the case $h = 0$ or $t = 0$), then $\CKh_{h, t}$ admits a secondary grading, called the \textit{quantum grading}, which is preserved by the differential $d$. Otherwise, if $R$ is non-graded and $\deg h = \deg t = 0$, then $d$ is quantum grading non-decreasing and thus $\CKh_{h, t}$ admits a filtration, called the \textit{quantum filtration}. In this section, we assume that the first assumption holds (later in \Cref{sec:computation}, we consider the filtered case). Hereafter, we make the ground ring $R$ and $(h, t)$ implicit and omit them from the notations, unless stating results specific to the Frobenius extension.

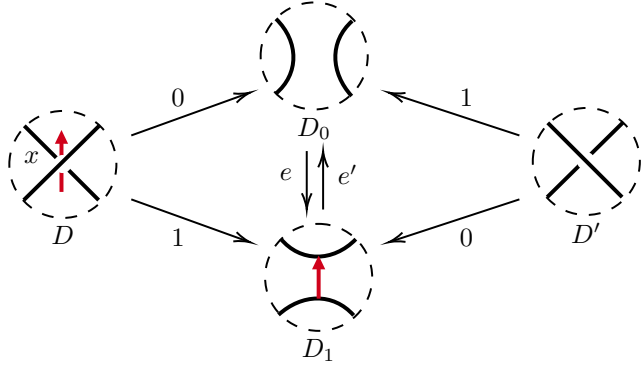
\begin{figure}[t]
    \centering
    \tikzset{every picture/.style={line width=0.75pt}} 

\begin{tikzpicture}[x=0.75pt,y=0.75pt,yscale=-.75,xscale=.75]

\draw [color={rgb, 255:red, 208; green, 2; blue, 27 }  ,draw opacity=1 ][line width=1.5]    (53.07,92.81) -- (53.07,130.49) ;
\draw [shift={(53.07,88.81)}, rotate = 90] [fill={rgb, 255:red, 208; green, 2; blue, 27 }  ,fill opacity=1 ][line width=0.08]  [draw opacity=0] (9.29,-4.46) -- (0,0) -- (9.29,4.46) -- cycle    ;
\draw  [dash pattern={on 4.5pt off 4.5pt}] (18,112) .. controls (18,92.12) and (34.12,76) .. (54,76) .. controls (73.88,76) and (90,92.12) .. (90,112) .. controls (90,131.88) and (73.88,148) .. (54,148) .. controls (34.12,148) and (18,131.88) .. (18,112) -- cycle ;
\draw [color={rgb, 255:red, 0; green, 0; blue, 0 }  ,draw opacity=1 ][line width=1.5]    (28.07,86) -- (78.07,136) ;
\draw  [draw opacity=0][fill={rgb, 255:red, 255; green, 255; blue, 255 }  ,fill opacity=1 ] (47.71,111.71) .. controls (47.71,108.36) and (50.43,105.64) .. (53.79,105.64) .. controls (57.14,105.64) and (59.86,108.36) .. (59.86,111.71) .. controls (59.86,115.07) and (57.14,117.79) .. (53.79,117.79) .. controls (50.43,117.79) and (47.71,115.07) .. (47.71,111.71) -- cycle ;
\draw [color={rgb, 255:red, 0; green, 0; blue, 0 }  ,draw opacity=1 ][line width=1.5]    (196.93,64.56) .. controls (213,51.34) and (213,29.34) .. (197.29,14) ;
\draw [color={rgb, 255:red, 0; green, 0; blue, 0 }  ,draw opacity=1 ][line width=1.5]    (248.21,66) .. controls (234,50.34) and (233,29.34) .. (248.57,15.44) ;
\draw  [dash pattern={on 4.5pt off 4.5pt}] (187,39) .. controls (187,19.12) and (203.12,3) .. (223,3) .. controls (242.88,3) and (259,19.12) .. (259,39) .. controls (259,58.88) and (242.88,75) .. (223,75) .. controls (203.12,75) and (187,58.88) .. (187,39) -- cycle ;

\draw  [dash pattern={on 4.5pt off 4.5pt}] (370,110) .. controls (370,90.12) and (386.12,74) .. (406,74) .. controls (425.88,74) and (442,90.12) .. (442,110) .. controls (442,129.88) and (425.88,146) .. (406,146) .. controls (386.12,146) and (370,129.88) .. (370,110) -- cycle ;
\draw [color={rgb, 255:red, 0; green, 0; blue, 0 }  ,draw opacity=1 ][line width=1.5]    (430.43,84) -- (379.71,134) ;
\draw  [draw opacity=0][fill={rgb, 255:red, 255; green, 255; blue, 255 }  ,fill opacity=1 ] (399.71,109.71) .. controls (399.71,106.36) and (402.43,103.64) .. (405.79,103.64) .. controls (409.14,103.64) and (411.86,106.36) .. (411.86,109.71) .. controls (411.86,113.07) and (409.14,115.79) .. (405.79,115.79) .. controls (402.43,115.79) and (399.71,113.07) .. (399.71,109.71) -- cycle ;
\draw [color={rgb, 255:red, 0; green, 0; blue, 0 }  ,draw opacity=1 ][line width=1.5]    (380.07,84) -- (420.04,123.96) -- (430.07,134) ;
\draw    (100,92.49) -- (178.12,64.28) ;
\draw [shift={(180,63.6)}, rotate = 160.14] [color={rgb, 255:red, 0; green, 0; blue, 0 }  ][line width=0.75]    (10.93,-3.29) .. controls (6.95,-1.4) and (3.31,-0.3) .. (0,0) .. controls (3.31,0.3) and (6.95,1.4) .. (10.93,3.29)   ;
\draw    (274.9,163.78) -- (361,135.51) ;
\draw [shift={(273,164.4)}, rotate = 341.82] [color={rgb, 255:red, 0; green, 0; blue, 0 }  ][line width=0.75]    (10.93,-3.29) .. controls (6.95,-1.4) and (3.31,-0.3) .. (0,0) .. controls (3.31,0.3) and (6.95,1.4) .. (10.93,3.29)   ;
\draw    (100,135.51) -- (178.12,163.72) ;
\draw [shift={(180,164.4)}, rotate = 199.86] [color={rgb, 255:red, 0; green, 0; blue, 0 }  ][line width=0.75]    (10.93,-3.29) .. controls (6.95,-1.4) and (3.31,-0.3) .. (0,0) .. controls (3.31,0.3) and (6.95,1.4) .. (10.93,3.29)   ;
\draw    (274.9,64.22) -- (361,92.49) ;
\draw [shift={(273,63.6)}, rotate = 18.18] [color={rgb, 255:red, 0; green, 0; blue, 0 }  ][line width=0.75]    (10.93,-3.29) .. controls (6.95,-1.4) and (3.31,-0.3) .. (0,0) .. controls (3.31,0.3) and (6.95,1.4) .. (10.93,3.29)   ;
\draw [color={rgb, 255:red, 0; green, 0; blue, 0 }  ,draw opacity=1 ][line width=1.5]    (200.44,161.93) .. controls (213.66,178) and (235.66,178) .. (251,162.29) ;
\draw [color={rgb, 255:red, 0; green, 0; blue, 0 }  ,draw opacity=1 ][line width=1.5]    (199,213.21) .. controls (214.66,199) and (235.66,198) .. (249.56,213.57) ;
\draw  [dash pattern={on 4.5pt off 4.5pt}] (226,152) .. controls (245.88,152) and (262,168.12) .. (262,188) .. controls (262,207.88) and (245.88,224) .. (226,224) .. controls (206.12,224) and (190,207.88) .. (190,188) .. controls (190,168.12) and (206.12,152) .. (226,152) -- cycle ;

\draw [color={rgb, 255:red, 0; green, 0; blue, 0 }  ,draw opacity=1 ][line width=0.75]    (218,103) -- (218,140.34) ;
\draw [shift={(218,142.34)}, rotate = 270] [color={rgb, 255:red, 0; green, 0; blue, 0 }  ,draw opacity=1 ][line width=0.75]    (10.93,-3.29) .. controls (6.95,-1.4) and (3.31,-0.3) .. (0,0) .. controls (3.31,0.3) and (6.95,1.4) .. (10.93,3.29)   ;
\draw [color={rgb, 255:red, 0; green, 0; blue, 0 }  ,draw opacity=1 ][line width=0.75]    (229,105) -- (229,142.34) ;
\draw [shift={(229,103)}, rotate = 90] [color={rgb, 255:red, 0; green, 0; blue, 0 }  ,draw opacity=1 ][line width=0.75]    (10.93,-3.29) .. controls (6.95,-1.4) and (3.31,-0.3) .. (0,0) .. controls (3.31,0.3) and (6.95,1.4) .. (10.93,3.29)   ;
\draw [color={rgb, 255:red, 0; green, 0; blue, 0 }  ,draw opacity=1 ][line width=1.5]    (78.43,86) -- (27.71,136) ;
\draw [color={rgb, 255:red, 208; green, 2; blue, 27 }  ,draw opacity=1 ][line width=1.5]    (226,177.31) -- (226,202.31) ;
\draw [shift={(226,173.31)}, rotate = 90] [fill={rgb, 255:red, 208; green, 2; blue, 27 }  ,fill opacity=1 ][line width=0.08]  [draw opacity=0] (9.29,-4.46) -- (0,0) -- (9.29,4.46) -- cycle    ;

\draw (54,151.4) node [anchor=north] [inner sep=0.75pt]    {$D$};
\draw (406,149.4) node [anchor=north] [inner sep=0.75pt]    {$D'$};
\draw (223,78.4) node [anchor=north] [inner sep=0.75pt]    {$D_{0}$};
\draw (226,227.4) node [anchor=north] [inner sep=0.75pt]    {$D_{1}$};
\draw (138,153.35) node [anchor=north east] [inner sep=0.75pt]    {$1$};
\draw (138,74.65) node [anchor=south east] [inner sep=0.75pt]    {$0$};
\draw (319,153.35) node [anchor=north west][inner sep=0.75pt]    {$0$};
\draw (319,74.65) node [anchor=south west] [inner sep=0.75pt]    {$1$};
\draw (216,118.67) node [anchor=east] [inner sep=0.75pt]  [color={rgb, 255:red, 0; green, 0; blue, 0 }  ,opacity=1 ]  {$e\ $};
\draw (231,118.67) node [anchor=west] [inner sep=0.75pt]  [color={rgb, 255:red, 0; green, 0; blue, 0 }  ,opacity=1 ]  {$\textcolor[rgb]{0.82,0.01,0.11}{\ } e'$};
\draw (26,101.4) node [anchor=north west][inner sep=0.75pt]    {$x$};

\end{tikzpicture}
    \caption{Diagrams $D$, $D'$ and their resolutions $D_0, D_1$ at $x$. }
    \label{fig:D01}
\end{figure}

\subsection{Crossing change maps}
\label{subsec:x-ch-map}

First, we define the crossing change maps combinatorially. Suppose $D$, $D'$ are diagrams related by a crossing change at a crossing $x$ of $D$, either from positive to negative or the other way, as depicted in \Cref{fig:D01}. Let $D_0, D_1$ be the $0$-, $1$-resolved diagram of $D$ at $x$ respectively. Furthermore, suppose that the crossing $x$ is associated a \textit{direction}, as indicated by the red arrow, which determines the order of the two arcs appearing the $1$-resolved diagram $D_1$ of $D$. Two chain maps $f_0, f_1$ are defined as follows: 
\[
    \begin{tikzcd}[row sep=3em]
    \CKh(D) \arrow[d, "f_0"', dashed, shift right] \arrow[d, "f_1", shift left] & = & \{0 \arrow[r] & \CKh(D_0) \arrow[rd, "I", pos=.75] \arrow[r, "e"] & \CKh(D_1) \arrow[r] \arrow[ld, "\Phi"', dashed, pos=.75] & 0\} \\
    \CKh(D')                                                            & = & \{0 \arrow[r]   & \CKh(D_1) \arrow[r, "e'"]                   & \CKh(D_0) \arrow[r]                                & 0\}.  
    \end{tikzcd}
\]
The the dashed arrows indicate $f_0$ and solid arrows indicate $f_1$. On the right side, $\CKh(D)$ is regarded as the mapping cone of the saddle map $e$, and $\CKh(D')$ as the cone of the saddle map $e'$, as indicated by the two vertical arrows in \Cref{fig:D01}. The diagonal $I$ indicates the identity map on $\CKh(D_0)$, and $\Phi$ is the endomorphism on $\CKh(D_1)$ defined as
\begin{center}
    \resizebox{!}{4em}{
    \tikzset{every picture/.style={line width=0.75pt}} 

\begin{tikzpicture}[x=0.75pt,y=0.75pt,yscale=-1,xscale=1]

\draw [color={rgb, 255:red, 208; green, 2; blue, 27 }  ,draw opacity=1 ][line width=2.25]    (163,38.32) -- (163,60.68) ;
\draw [shift={(163,33.32)}, rotate = 90] [fill={rgb, 255:red, 208; green, 2; blue, 27 }  ,fill opacity=1 ][line width=0.08]  [draw opacity=0] (10,-4.8) -- (0,0) -- (10,4.8) -- cycle    ;
\draw [color={rgb, 255:red, 208; green, 2; blue, 27 }  ,draw opacity=1 ][line width=2.25]    (46,38.32) -- (46,60.68) ;
\draw [shift={(46,33.32)}, rotate = 90] [fill={rgb, 255:red, 208; green, 2; blue, 27 }  ,fill opacity=1 ][line width=0.08]  [draw opacity=0] (10,-4.8) -- (0,0) -- (10,4.8) -- cycle    ;
\draw  [fill={rgb, 255:red, 0; green, 0; blue, 0 }  ,fill opacity=1 ] (162.95,59.25) .. controls (165.44,59.25) and (167.45,61.26) .. (167.45,63.75) .. controls (167.45,66.24) and (165.44,68.25) .. (162.95,68.25) .. controls (160.46,68.25) and (158.45,66.24) .. (158.45,63.75) .. controls (158.45,61.26) and (160.46,59.25) .. (162.95,59.25) -- cycle ;
\draw  [fill={rgb, 255:red, 0; green, 0; blue, 0 }  ,fill opacity=1 ] (46,28) .. controls (48.49,28) and (50.5,30.01) .. (50.5,32.5) .. controls (50.5,34.99) and (48.49,37) .. (46,37) .. controls (43.51,37) and (41.5,34.99) .. (41.5,32.5) .. controls (41.5,30.01) and (43.51,28) .. (46,28) -- cycle ;
\draw [color={rgb, 255:red, 0; green, 0; blue, 0 }  ,draw opacity=1 ][line width=1.5]    (21.19,21.18) .. controls (41,36.81) and (52,36.81) .. (71.75,21.54) ;
\draw [color={rgb, 255:red, 0; green, 0; blue, 0 }  ,draw opacity=1 ][line width=1.5]    (19.75,72.46) .. controls (37,59.81) and (55,57.81) .. (70.31,72.82) ;
\draw  [dash pattern={on 4.5pt off 4.5pt}] (10,46) .. controls (10,26.12) and (26.12,10) .. (46,10) .. controls (65.88,10) and (82,26.12) .. (82,46) .. controls (82,65.88) and (65.88,82) .. (46,82) .. controls (26.12,82) and (10,65.88) .. (10,46) -- cycle ;
\draw [color={rgb, 255:red, 0; green, 0; blue, 0 }  ,draw opacity=1 ][line width=1.5]    (138.19,21.18) .. controls (158,36.81) and (169,36.81) .. (188.75,21.54) ;
\draw [color={rgb, 255:red, 0; green, 0; blue, 0 }  ,draw opacity=1 ][line width=1.5]    (136.75,72.46) .. controls (154,59.81) and (172,57.81) .. (187.31,72.82) ;
\draw  [dash pattern={on 4.5pt off 4.5pt}] (127,46) .. controls (127,26.12) and (143.12,10) .. (163,10) .. controls (182.88,10) and (199,26.12) .. (199,46) .. controls (199,65.88) and (182.88,82) .. (163,82) .. controls (143.12,82) and (127,65.88) .. (127,46) -- cycle ;

\draw (100,35.4) node [anchor=north west][inner sep=0.75pt]    {$-$};

\end{tikzpicture}
    }
\end{center}
Here, a black dot indicates multiplying $X$ on the circle that it lies on, and the direction is used to fix the sign of $\Phi$. Alternatively, the map $\Phi$ can be described in the form of cobordisms as: 
\begin{center}
    \resizebox{!}{6em}{
    \tikzset{every picture/.style={line width=0.75pt}} 

\begin{tikzpicture}[x=0.75pt,y=0.75pt,yscale=-.8,xscale=.8]

\draw [color={rgb, 255:red, 0; green, 0; blue, 0 }  ,draw opacity=1 ][line width=1.5]    (104.96,30.56) .. controls (105,29.25) and (109,32.75) .. (114.14,24.74) ;
\draw [color={rgb, 255:red, 0; green, 0; blue, 0 }  ,draw opacity=1 ][line width=1.5]    (96.5,60.56) .. controls (100.5,60.75) and (103.23,59.24) .. (109.32,74.52) ;
\draw  [dash pattern={on 4.5pt off 4.5pt}] (105.5,9.33) .. controls (113.05,13.59) and (117.83,32.78) .. (116.18,52.2) .. controls (114.52,71.62) and (107.06,83.92) .. (99.5,79.66) .. controls (91.95,75.41) and (87.17,56.21) .. (88.82,36.79) .. controls (90.48,17.37) and (97.94,5.07) .. (105.5,9.33) -- cycle ;
\draw [color={rgb, 255:red, 0; green, 0; blue, 0 }  ,draw opacity=1 ][line width=1.5]    (19.96,13.56) .. controls (25.84,31.28) and (30.32,34.82) .. (39.14,24.74) ;
\draw [color={rgb, 255:red, 0; green, 0; blue, 0 }  ,draw opacity=1 ][line width=1.5]    (15.14,63.34) .. controls (23.59,57.65) and (28.23,59.24) .. (34.32,74.52) ;
\draw  [dash pattern={on 4.5pt off 4.5pt}] (30.5,9.33) .. controls (38.05,13.59) and (42.83,32.78) .. (41.18,52.2) .. controls (39.52,71.62) and (32.06,83.92) .. (24.5,79.66) .. controls (16.95,75.41) and (12.17,56.21) .. (13.82,36.79) .. controls (15.48,17.37) and (22.94,5.07) .. (30.5,9.33) -- cycle ;
\draw [line width=0.75]    (114.14,24.74) -- (39.14,24.74) ;
\draw [line width=0.75]    (25.5,63.34) -- (15.14,63.34) ;
\draw [line width=0.75]    (96.5,60.56) -- (24.96,60.56) ;
\draw [line width=0.75]    (109.32,74.52) -- (34.32,74.52) ;
\draw [line width=0.75]    (104.96,30.56) -- (29.96,30.56) ;
\draw [line width=0.75]    (94.96,13.56) -- (19.96,13.56) ;
\draw [line width=0.75]    (61.46,30.75) -- (61.46,60.31) ;
\draw [line width=0.75]    (69.96,30.75) -- (69.96,60.31) ;
\draw [color={rgb, 255:red, 0; green, 0; blue, 0 }  ,draw opacity=1 ][line width=1.5]    (94.96,13.56) .. controls (96.15,17.14) and (97.28,20.14) .. (98.4,22.57) ;
\draw [line width=0.75]    (191.17,59.79) .. controls (193.5,46.81) and (210.22,45.83) .. (212.56,60.11) ;
\draw [line width=0.75]    (185.33,59.79) .. controls (186.89,42.26) and (213.33,36.75) .. (218,60.11) ;

\draw [color={rgb, 255:red, 0; green, 0; blue, 0 }  ,draw opacity=1 ][line width=1.5]    (241.96,30.56) .. controls (242,29.25) and (246,32.75) .. (251.14,24.74) ;
\draw [color={rgb, 255:red, 0; green, 0; blue, 0 }  ,draw opacity=1 ][line width=1.5]    (233.5,60.56) .. controls (237.5,60.75) and (240.23,59.24) .. (246.32,74.52) ;
\draw  [dash pattern={on 4.5pt off 4.5pt}] (242.5,9.33) .. controls (250.05,13.59) and (254.83,32.78) .. (253.18,52.2) .. controls (251.52,71.62) and (244.06,83.92) .. (236.5,79.66) .. controls (228.95,75.41) and (224.17,56.21) .. (225.82,36.79) .. controls (227.48,17.37) and (234.94,5.07) .. (242.5,9.33) -- cycle ;
\draw [color={rgb, 255:red, 0; green, 0; blue, 0 }  ,draw opacity=1 ][line width=1.5]    (156.96,13.56) .. controls (162.84,31.28) and (167.32,34.82) .. (176.14,24.74) ;
\draw [color={rgb, 255:red, 0; green, 0; blue, 0 }  ,draw opacity=1 ][line width=1.5]    (152.14,63.34) .. controls (160.59,57.65) and (165.23,59.24) .. (171.32,74.52) ;
\draw  [dash pattern={on 4.5pt off 4.5pt}] (167.5,9.33) .. controls (175.05,13.59) and (179.83,32.78) .. (178.18,52.2) .. controls (176.52,71.62) and (169.06,83.92) .. (161.5,79.66) .. controls (153.95,75.41) and (149.17,56.21) .. (150.82,36.79) .. controls (152.48,17.37) and (159.94,5.07) .. (167.5,9.33) -- cycle ;
\draw [line width=0.75]    (251.14,24.74) -- (176.14,24.74) ;
\draw [line width=0.75]    (162.5,63.34) -- (152.14,63.34) ;
\draw [line width=0.75]    (233.5,60.56) -- (161.96,60.56) ;
\draw [line width=0.75]    (246.32,74.52) -- (171.32,74.52) ;
\draw [line width=0.75]    (241.96,30.56) -- (166.96,30.56) ;
\draw [line width=0.75]    (231.96,13.56) -- (156.96,13.56) ;
\draw [color={rgb, 255:red, 0; green, 0; blue, 0 }  ,draw opacity=1 ][line width=1.5]    (231.96,13.56) .. controls (233.15,17.14) and (234.28,20.14) .. (235.4,22.57) ;
\draw [color={rgb, 255:red, 208; green, 2; blue, 27 }  ,draw opacity=1 ][line width=1.5]    (26.5,34) -- (26.5,59) ;
\draw [shift={(26.5,30)}, rotate = 90] [fill={rgb, 255:red, 208; green, 2; blue, 27 }  ,fill opacity=1 ][line width=0.08]  [draw opacity=0] (9.29,-4.46) -- (0,0) -- (9.29,4.46) -- cycle    ;
\draw [color={rgb, 255:red, 208; green, 2; blue, 27 }  ,draw opacity=1 ][line width=1.5]    (163.5,34) -- (163.5,59) ;
\draw [shift={(163.5,30)}, rotate = 90] [fill={rgb, 255:red, 208; green, 2; blue, 27 }  ,fill opacity=1 ][line width=0.08]  [draw opacity=0] (9.29,-4.46) -- (0,0) -- (9.29,4.46) -- cycle    ;

\draw (126.5,36.4) node [anchor=north west][inner sep=0.75pt]    {$-$};

\end{tikzpicture}
    }
\end{center}
Using the notation of \cite{Khovanov:2022}, $\Phi$ can also be expressed as a tube with a \textit{defect circle} along its meridian,
\begin{center}
    \tikzset{every picture/.style={line width=0.75pt}} 

\begin{tikzpicture}[x=0.75pt,y=0.75pt,yscale=-1,xscale=1]

\draw [color={rgb, 255:red, 0; green, 0; blue, 0 }  ,draw opacity=1 ][line width=1.5]    (161.96,34.56) .. controls (162,33.25) and (166,36.75) .. (171.14,28.74) ;
\draw [color={rgb, 255:red, 0; green, 0; blue, 0 }  ,draw opacity=1 ][line width=1.5]    (153.5,64.56) .. controls (157.5,64.75) and (160.23,63.24) .. (166.32,78.52) ;
\draw  [dash pattern={on 4.5pt off 4.5pt}] (162.5,13.33) .. controls (170.05,17.59) and (174.83,36.78) .. (173.18,56.2) .. controls (171.52,75.62) and (164.06,87.92) .. (156.5,83.66) .. controls (148.95,79.41) and (144.17,60.21) .. (145.82,40.79) .. controls (147.48,21.37) and (154.94,9.07) .. (162.5,13.33) -- cycle ;
\draw [color={rgb, 255:red, 0; green, 0; blue, 0 }  ,draw opacity=1 ][line width=1.5]    (76.96,17.56) .. controls (82.84,35.28) and (87.32,38.82) .. (96.14,28.74) ;
\draw [color={rgb, 255:red, 0; green, 0; blue, 0 }  ,draw opacity=1 ][line width=1.5]    (72.14,67.34) .. controls (80.59,61.65) and (85.23,63.24) .. (91.32,78.52) ;
\draw  [dash pattern={on 4.5pt off 4.5pt}] (87.5,13.33) .. controls (95.05,17.59) and (99.83,36.78) .. (98.18,56.2) .. controls (96.52,75.62) and (89.06,87.92) .. (81.5,83.66) .. controls (73.95,79.41) and (69.17,60.21) .. (70.82,40.79) .. controls (72.48,21.37) and (79.94,9.07) .. (87.5,13.33) -- cycle ;
\draw [line width=0.75]    (171.14,28.74) -- (96.14,28.74) ;
\draw [line width=0.75]    (82.5,67.34) -- (72.14,67.34) ;
\draw [line width=0.75]    (153.5,64.56) -- (81.96,64.56) ;
\draw [line width=0.75]    (166.32,78.52) -- (91.32,78.52) ;
\draw [line width=0.75]    (161.96,34.56) -- (86.96,34.56) ;
\draw [line width=0.75]    (151.96,17.56) -- (76.96,17.56) ;
\draw [line width=0.75]    (116.46,34.75) -- (116.46,64.31) ;
\draw [line width=0.75]    (132.96,34.75) -- (132.96,64.31) ;
\draw [color={rgb, 255:red, 0; green, 0; blue, 0 }  ,draw opacity=1 ][line width=1.5]    (151.96,17.56) .. controls (153.15,21.14) and (154.28,24.14) .. (155.4,26.57) ;
\draw  [draw opacity=0][line width=1.5]  (132.96,45.3) .. controls (132.96,45.34) and (132.96,45.38) .. (132.96,45.42) .. controls (132.96,47.9) and (129.31,49.9) .. (124.8,49.9) .. controls (120.34,49.9) and (116.71,47.93) .. (116.65,45.48) -- (124.8,45.42) -- cycle ; \draw  [color={rgb, 255:red, 208; green, 2; blue, 27 }  ,draw opacity=1 ][line width=1.5]  (132.96,45.3) .. controls (132.96,45.34) and (132.96,45.38) .. (132.96,45.42) .. controls (132.96,47.9) and (129.31,49.9) .. (124.8,49.9) .. controls (120.34,49.9) and (116.71,47.93) .. (116.65,45.48) ;  
\draw [color={rgb, 255:red, 208; green, 2; blue, 27 }  ,draw opacity=1 ][line width=1.5]    (125.29,49.17) -- (125.29,54.57) ;


\end{tikzpicture}
\end{center}
That the two definitions are equal can be checked by the \textit{neck-cutting relation} (see \Cref{sec:Kh-structure}). One can easily verify that $\Phi e = 0$ and $e' \Phi = 0$, hence the two maps $f_0, f_1$ are indeed chain maps. Hereafter, any linear combination of the maps $f_0, f_1$ is called a \textit{crossing change map}. We make the direction implicit whenever the sign of $f_0$ is irrelevant. 

The bidegrees of $f_0, f_1$ depend on the sign of the crossing $x$. We write $f_1^-$ (resp.\ $f_1^+$) to indicate that $x$ is positive (resp.\ negative) and $D \to D'$ is a positive-to-negative (resp.\ negative-to-positive) crossing change. Then we have 
\begin{align*}
    \deg f_0^- &= (-2, -6), &\deg f_0^+ &= (0, 0), \\
    \deg f_1^-    &= (0, -2),  &\deg f_1^+    &= (2, 4),
\end{align*}
In either case, we have 
\[
    \deg f_1 - \deg f_0 = (2, 4). 
\]
Any linear combination of the maps $f^-_0, f^-_1$ (resp.\ $f^+_0, f^+_1$) is called a \textit{positive-to-negative} (resp.\ \textit{negative-to-positive}) \textit{crossing change map}. 

\begin{rem}
    The map $\Phi$ appears in \cite{Ito-Yoshida:2021}, as the negative-to-positive crossing change map. In \cite{Alishahi:2017} and \cite{Alishahi:2018}, crossing change maps
    \[
    \begin{tikzcd}
        \CKh(D^+) \arrow[r, "f^-", shift left] & \CKh(D^-) \arrow[l, "f^+", shift left]
    \end{tikzcd}
    \]
    for the (bigraded) Bar-Natan complex $(h, t) = (h, 0)$ over $R = \bbF_2[h]$ and the (bigraded) Lee complex $(h, t) = (0, t)$ over $R = \Q[t]$ are given, both of which can be described using our crossing change maps as $f^- = f_0^- + f_1^-$ and $f^+ = f_0^+ + f_1^+$. Note that the two crossing change maps $f^\pm$ are not homogeneous, but the compositions $f^+ f^-$ and $f^- f^+$ are both homogeneous with bigrading $(0, -2)$. 
\end{rem}

The following proposition relates the above defined maps with cobordism maps obtained from \textit{embedded} cobordisms that realize the crossing changes in both directions. 

\begin{prop}
\label{prop:f0-explicit}
    Suppose $D^+$, $D^-$ are diagrams related by a positive-to-negative crossing change at a crossing $x$. Consider the following sequence of elementary moves:
    \vspace{.5em}
    \begin{center}
        \resizebox{0.8\textwidth}{!}{
            \input{tikzpictures/emb-xch}
        }
    \end{center}
    \vspace{.5em}
    Then the chain map obtained from the sequence of moves from $D^+$ to $D^-$ coincides with 
    \[
        f_1^-: \CKh(D^+) \rightarrow \CKh(D^-).
    \]
    The chain map corresponding to the reversed sequence of moves from $D^-$ to $D^+$ coincides with 
    \[
        f_0^+ f_1^- f_0^+: \CKh(D^-) \rightarrow \CKh(D^+).
    \]
    Both of these maps have bidegree $(0, -2)$. 
\end{prop}

\begin{proof}
    Immediate from the explicit descriptions of the maps given in \cite{BarNatan:2004}. 
\end{proof}

\Cref{prop:f0-explicit} shows that the positive-to-negative crossing change map $f_1^-$ can be realized by an embedded cobordism, but the negative-to-positive $f_0^+$ cannot. This asymmetry will be essential throughout the paper. 

\subsection{Geometric description} 
\label{subsect: Geometric description}

\begin{figure}[t]
    \centering
    \input{tikzpictures/ckh-hopf}
    \caption{}
    \label{fig:ckh-hopf}
    \vspace{2.5em}
    \input{tikzpictures/zeta01}
    \caption{}
    \label{fig:zeta01}
\end{figure}

Next, we give geometric descriptions for the above defined crossing change maps. Consider the diagram $H$ of the Hopf link depicted in \Cref{fig:ckh-hopf}, which we call the \textit{standard Hopf link diagram}. First, we ignore the orientations on $H$, and consider relative bigradings on $\Kh(H)$. The cube of resolutions for $H$ is described in the right of \Cref{fig:ckh-hopf}. It can be computed directly that $\Kh(H)$ is free of rank $2$, with the relative bigrading given by 
\[
    \Kh(H) \cong R\{0, 0\} \oplus R\{0, 2\} \oplus R\{2, 4\} \oplus R\{2, 6\}.
\]
Define elements $\zeta_0, \zeta_1$ in $\CKh(H)$ as depicted in \Cref{fig:zeta01}. Obviously, these two elements are cycles, and it can be shown that the four cycles
\begin{align*}
    \zeta_0  &\in \CKh^{0, 2}(H), &\zeta_1  &\in \CKh^{2, 6}(H),\\
    \bar{X}\zeta_0 &\in \CKh^{0, 0}(H), &\bar{X}\zeta_1 &\in \CKh^{2, 4}(H)
\end{align*}
generate $\Kh(H)$. Here, $\bar{X} \zeta_i$ indicates multiplying $X \in A$ on one of the two components of $\zeta_i$. 

Let $H^\pm$ denote the standard Hopf link diagram equipped with an orientation indicated by the superscript. For the positive diagram $H^+$, the above relative bigrading is exactly the absolute bigrading on $\CKh(H^+)$. For the negative diagram $H^-$, there is an identification
\[
    \CKh(H^-) = \CKh(H^+)\{-2, -6\}
\]
and we have 
\[
    \zeta_0 \in \CKh^{-2, -4}(H^-), \quad \zeta_1 \in \CKh^{0, 0}(H^-).
\]

\begin{figure}[t]
    \centering
    \tikzset{every picture/.style={line width=0.75pt}} 

\begin{tikzpicture}[x=0.75pt,y=0.75pt,yscale=-.75,xscale=.75]

\draw [color={rgb, 255:red, 0; green, 0; blue, 0 }  ,draw opacity=1 ][line width=1.5]    (14.74,19.33) -- (111.64,116.23) ;
\draw  [draw opacity=0][fill={rgb, 255:red, 255; green, 255; blue, 255 }  ,fill opacity=1 ] (52.81,69.17) .. controls (52.81,62.67) and (58.08,57.4) .. (64.57,57.4) .. controls (71.07,57.4) and (76.34,62.67) .. (76.34,69.17) .. controls (76.34,75.67) and (71.07,80.93) .. (64.57,80.93) .. controls (58.08,80.93) and (52.81,75.67) .. (52.81,69.17) -- cycle ;
\draw [color={rgb, 255:red, 0; green, 0; blue, 0 }  ,draw opacity=1 ][line width=1.5]    (112.33,19.33) -- (14.05,116.23) ;
\draw  [draw opacity=0][line width=1.5]  (118.16,67.79) .. controls (120.08,65.57) and (121.25,62.67) .. (121.25,59.5) .. controls (121.25,52.49) and (115.57,46.81) .. (108.56,46.81) .. controls (101.55,46.81) and (95.87,52.49) .. (95.87,59.5) .. controls (95.87,66.5) and (101.55,72.18) .. (108.56,72.18) .. controls (110.04,72.18) and (111.46,71.93) .. (112.78,71.46) -- (108.56,59.5) -- cycle ; \draw  [line width=1.5]  (118.16,67.79) .. controls (120.08,65.57) and (121.25,62.67) .. (121.25,59.5) .. controls (121.25,52.49) and (115.57,46.81) .. (108.56,46.81) .. controls (101.55,46.81) and (95.87,52.49) .. (95.87,59.5) .. controls (95.87,66.5) and (101.55,72.18) .. (108.56,72.18) .. controls (110.04,72.18) and (111.46,71.93) .. (112.78,71.46) ;  
\draw  [draw opacity=0][line width=1.5]  (98.96,72.14) .. controls (97.04,74.37) and (95.87,77.27) .. (95.87,80.45) .. controls (95.87,87.46) and (101.56,93.15) .. (108.57,93.15) .. controls (115.58,93.15) and (121.27,87.46) .. (121.27,80.45) .. controls (121.27,73.43) and (115.58,67.75) .. (108.57,67.75) .. controls (107.09,67.75) and (105.67,68) .. (104.35,68.46) -- (108.57,80.45) -- cycle ; \draw  [line width=1.5]  (98.96,72.14) .. controls (97.04,74.37) and (95.87,77.27) .. (95.87,80.45) .. controls (95.87,87.46) and (101.56,93.15) .. (108.57,93.15) .. controls (115.58,93.15) and (121.27,87.46) .. (121.27,80.45) .. controls (121.27,73.43) and (115.58,67.75) .. (108.57,67.75) .. controls (107.09,67.75) and (105.67,68) .. (104.35,68.46) ;  
\draw [color={rgb, 255:red, 208; green, 2; blue, 27 }  ,draw opacity=1 ][line width=1.5]    (100,31.67) -- (108.33,47) ;
\draw [color={rgb, 255:red, 208; green, 2; blue, 27 }  ,draw opacity=1 ][line width=1.5]    (108,92.33) -- (101.67,106.33) ;
\draw [color={rgb, 255:red, 0; green, 0; blue, 0 }  ,draw opacity=1 ][line width=4.5]    (278.2,92.98) -- (272.2,104.18) ;
\draw [color={rgb, 255:red, 0; green, 0; blue, 0 }  ,draw opacity=1 ][line width=4.5]    (270.4,32.78) -- (278.6,47.38) ;
\draw [color={rgb, 255:red, 0; green, 0; blue, 0 }  ,draw opacity=1 ][line width=1.5]    (186.74,19.77) -- (283.64,116.67) ;
\draw  [draw opacity=0][fill={rgb, 255:red, 255; green, 255; blue, 255 }  ,fill opacity=1 ] (224.81,69.6) .. controls (224.81,63.1) and (230.08,57.83) .. (236.57,57.83) .. controls (243.07,57.83) and (248.34,63.1) .. (248.34,69.6) .. controls (248.34,76.1) and (243.07,81.37) .. (236.57,81.37) .. controls (230.08,81.37) and (224.81,76.1) .. (224.81,69.6) -- cycle ;
\draw [color={rgb, 255:red, 0; green, 0; blue, 0 }  ,draw opacity=1 ][line width=1.5]    (284.33,19.77) -- (186.05,116.67) ;
\draw  [draw opacity=0][line width=1.5]  (290.16,68.23) .. controls (292.08,66) and (293.25,63.1) .. (293.25,59.93) .. controls (293.25,52.92) and (287.57,47.24) .. (280.56,47.24) .. controls (273.55,47.24) and (267.87,52.92) .. (267.87,59.93) .. controls (267.87,66.93) and (273.55,72.62) .. (280.56,72.62) .. controls (282.04,72.62) and (283.46,72.36) .. (284.78,71.9) -- (280.56,59.93) -- cycle ; \draw  [line width=1.5]  (290.16,68.23) .. controls (292.08,66) and (293.25,63.1) .. (293.25,59.93) .. controls (293.25,52.92) and (287.57,47.24) .. (280.56,47.24) .. controls (273.55,47.24) and (267.87,52.92) .. (267.87,59.93) .. controls (267.87,66.93) and (273.55,72.62) .. (280.56,72.62) .. controls (282.04,72.62) and (283.46,72.36) .. (284.78,71.9) ;  
\draw  [draw opacity=0][line width=1.5]  (270.96,72.57) .. controls (269.04,74.8) and (267.87,77.7) .. (267.87,80.88) .. controls (267.87,87.89) and (273.56,93.58) .. (280.57,93.58) .. controls (287.58,93.58) and (293.27,87.89) .. (293.27,80.88) .. controls (293.27,73.86) and (287.58,68.18) .. (280.57,68.18) .. controls (279.09,68.18) and (277.67,68.43) .. (276.35,68.9) -- (280.57,80.88) -- cycle ; \draw  [line width=1.5]  (270.96,72.57) .. controls (269.04,74.8) and (267.87,77.7) .. (267.87,80.88) .. controls (267.87,87.89) and (273.56,93.58) .. (280.57,93.58) .. controls (287.58,93.58) and (293.27,87.89) .. (293.27,80.88) .. controls (293.27,73.86) and (287.58,68.18) .. (280.57,68.18) .. controls (279.09,68.18) and (277.67,68.43) .. (276.35,68.9) ;  
\draw [color={rgb, 255:red, 0; green, 0; blue, 0 }  ,draw opacity=1 ][line width=1.5]    (372.74,19.77) -- (417,68.59) ;
\draw [color={rgb, 255:red, 0; green, 0; blue, 0 }  ,draw opacity=1 ][line width=1.5]    (470.33,19.77) -- (456.72,33.38) ;
\draw [color={rgb, 255:red, 255; green, 255; blue, 255 }  ,draw opacity=1 ][line width=1.5]    (279.6,90.37) -- (269.4,109.78) ;
\draw [color={rgb, 255:red, 255; green, 255; blue, 255 }  ,draw opacity=1 ][line width=1.5]    (268.2,28.98) -- (282,53.58) ;
\draw  [draw opacity=0][fill={rgb, 255:red, 255; green, 255; blue, 255 }  ,fill opacity=1 ] (443.89,38.89) -- (455.39,32.41) -- (470.25,58.76) -- (458.76,65.25) -- cycle ;
\draw [color={rgb, 255:red, 0; green, 0; blue, 0 }  ,draw opacity=1 ][line width=1.5]    (456.8,32.3) -- (465.13,47.63) ;
\draw [color={rgb, 255:red, 0; green, 0; blue, 0 }  ,draw opacity=1 ][line width=1.5]    (464.8,92.97) -- (458.47,106.97) ;
\draw [color={rgb, 255:red, 0; green, 0; blue, 0 }  ,draw opacity=1 ][line width=1.5]    (415.4,73.31) -- (372.05,116.67) ;
\draw [color={rgb, 255:red, 0; green, 0; blue, 0 }  ,draw opacity=1 ][line width=1.5]    (458.47,106.97) -- (469.64,116.67) ;
\draw  [draw opacity=0][fill={rgb, 255:red, 255; green, 255; blue, 255 }  ,fill opacity=1 ] (443.97,59.5) -- (466.02,59.55) -- (465.98,77.36) -- (443.93,77.31) -- cycle ;
\draw [color={rgb, 255:red, 0; green, 0; blue, 0 }  ,draw opacity=1 ][line width=1.5]    (415.4,73.31) -- (473.5,73.31) ;
\draw [color={rgb, 255:red, 0; green, 0; blue, 0 }  ,draw opacity=1 ][line width=1.5]    (417,68.59) -- (466.5,68.59) ;
\draw  [draw opacity=0][line width=1.5]  (478.8,66.22) .. controls (479.88,64.31) and (480.5,62.11) .. (480.5,59.77) .. controls (480.5,52.46) and (474.47,46.52) .. (467.04,46.52) .. controls (466.05,46.52) and (465.1,46.63) .. (464.17,46.82) -- (467.04,59.77) -- cycle ; \draw  [line width=1.5]  (478.8,66.22) .. controls (479.88,64.31) and (480.5,62.11) .. (480.5,59.77) .. controls (480.5,52.46) and (474.47,46.52) .. (467.04,46.52) .. controls (466.05,46.52) and (465.1,46.63) .. (464.17,46.82) ;  
\draw  [draw opacity=0][line width=1.5]  (463.71,93.98) .. controls (465.16,94.5) and (466.73,94.79) .. (468.36,94.79) .. controls (475.8,94.79) and (481.82,88.86) .. (481.82,81.54) .. controls (481.82,74.22) and (475.8,68.29) .. (468.36,68.29) .. controls (467.38,68.29) and (466.42,68.39) .. (465.5,68.59) -- (468.36,81.54) -- cycle ; \draw  [line width=1.5]  (463.71,93.98) .. controls (465.16,94.5) and (466.73,94.79) .. (468.36,94.79) .. controls (475.8,94.79) and (481.82,88.86) .. (481.82,81.54) .. controls (481.82,74.22) and (475.8,68.29) .. (468.36,68.29) .. controls (467.38,68.29) and (466.42,68.39) .. (465.5,68.59) ;  
\draw    (143.5,73.43) -- (178,73.43) ;
\draw [shift={(180,73.43)}, rotate = 180] [color={rgb, 255:red, 0; green, 0; blue, 0 }  ][line width=0.75]    (10.93,-3.29) .. controls (6.95,-1.4) and (3.31,-0.3) .. (0,0) .. controls (3.31,0.3) and (6.95,1.4) .. (10.93,3.29)   ;
\draw    (332.5,72.93) -- (367,72.93) ;
\draw [shift={(369,72.93)}, rotate = 180] [color={rgb, 255:red, 0; green, 0; blue, 0 }  ][line width=0.75]    (10.93,-3.29) .. controls (6.95,-1.4) and (3.31,-0.3) .. (0,0) .. controls (3.31,0.3) and (6.95,1.4) .. (10.93,3.29)   ;

\end{tikzpicture}
    \caption{}
    \label{fig:x-ch-hopf}
    \vspace{2.5em}
    \input{tikzpictures/x-ch-hopf-proof}
    \caption{}
    \label{fig:x-ch-hopf-proof}
\end{figure}

Now, suppose $D$, $D'$ are diagrams related by a crossing change at $x$ as in \Cref{fig:D01}. Note that the transformation from $D$ to $D'$ can be realized by placing a standard Hopf link diagram $H$ near $x$, and surguring it into $D$, as in \Cref{fig:x-ch-hopf}. With this picture in mind, given any cycle $z \in \CKh(H)$, we define a chain map by the following composition
\[
    F(z) \colon
    \CKh(D) 
        \xrightarrow{1 \otimes z}
    \CKh(D) \otimes \CKh(H)
        \xrightarrow{\text{saddle}}
    \CKh(D'')
        \xrightarrow{\text{R2}^{-1}}
    \CKh(D').
\]
The mapping $z \mapsto F(z)$ gives an $R$-homomorphism
\begin{align*}
    F &\colon
    Z(\CKh(H))
        \rightarrow
    \Hom(\CKh(D), \CKh(D'))
\end{align*}
where the domain is the cycle module of $\CKh(H)$ and the codomain is the $R$-module of chain maps from $\CKh(D)$ to $\CKh(D')$. 

\begin{prop}
    If cycles $z, z' \in \CKh(H)$ are homologous, then the corresponding chain maps $F(z), F(z')$ are chain homotopic. 
\end{prop}

\begin{proof}
    If $z = da$ is a boundary, then 
    \[
        h: x \mapsto (-1)^{\deg(x)}F(a)(x)
    \]
    gives a null-homotopy for $F(z)$.
\end{proof}

Thus $F$ induces an $R$-homomorphism
\[
    F \colon
    \Kh(H)
        \rightarrow
    \Hom(\Kh(D), \Kh(D')).
\]

\begin{prop}
\label{prop:x-ch-hopf}
    \[
        F(\zeta_0) = f_0, \quad
        F(\zeta_1) = f_1.
    \]
\end{prop}

\begin{proof}
    With the explicit description of the R2-move map given in \cite{BarNatan:2004}, the map $F$ can be described as the sum of two maps depicted in \Cref{fig:x-ch-hopf-proof}. Putting $\zeta_0$ and $\zeta_1$ into the pictures proves the result. 
\end{proof}

We call any chain map obtained from $F$ a \textit{crossing change map}. The two crossing change maps $f_0, f_1$ are canonical in the sense that they arise from the two homological generators $\zeta_0, \zeta_1$ of $\Kh(H)$. 

\begin{rem}
\label{rem:Hopf-from-cc}
    Conversely, the two cycles $\zeta_0, \zeta_1$ can be obtained as images of the crossing change maps $f_0$ and $f_1$. Consider the following sequence of moves:
    \vspace{.5em}
    \begin{center}
        \tikzset{every picture/.style={line width=0.75pt}} 

\begin{tikzpicture}[x=0.75pt,y=0.75pt,yscale=-.9,xscale=.9]

\draw  [draw opacity=0][line width=1.5]  (410.92,36.19) .. controls (412.94,33.84) and (414.17,30.79) .. (414.17,27.44) .. controls (414.17,20.06) and (408.19,14.08) .. (400.81,14.08) .. controls (393.42,14.08) and (387.44,20.06) .. (387.44,27.44) .. controls (387.44,34.83) and (393.42,40.81) .. (400.81,40.81) .. controls (402.36,40.81) and (403.86,40.54) .. (405.25,40.05) -- (400.81,27.44) -- cycle ; \draw  [line width=1.5]  (410.92,36.19) .. controls (412.94,33.84) and (414.17,30.79) .. (414.17,27.44) .. controls (414.17,20.06) and (408.19,14.08) .. (400.81,14.08) .. controls (393.42,14.08) and (387.44,20.06) .. (387.44,27.44) .. controls (387.44,34.83) and (393.42,40.81) .. (400.81,40.81) .. controls (402.36,40.81) and (403.86,40.54) .. (405.25,40.05) ;  
\draw  [draw opacity=0][line width=1.5]  (390.57,39.91) .. controls (388.62,42.17) and (387.44,45.11) .. (387.44,48.33) .. controls (387.44,55.43) and (393.2,61.19) .. (400.3,61.19) .. controls (407.41,61.19) and (413.17,55.43) .. (413.17,48.33) .. controls (413.17,41.22) and (407.41,35.46) .. (400.3,35.46) .. controls (398.81,35.46) and (397.37,35.72) .. (396.03,36.19) -- (400.3,48.33) -- cycle ; \draw  [line width=1.5]  (390.57,39.91) .. controls (388.62,42.17) and (387.44,45.11) .. (387.44,48.33) .. controls (387.44,55.43) and (393.2,61.19) .. (400.3,61.19) .. controls (407.41,61.19) and (413.17,55.43) .. (413.17,48.33) .. controls (413.17,41.22) and (407.41,35.46) .. (400.3,35.46) .. controls (398.81,35.46) and (397.37,35.72) .. (396.03,36.19) ;

\draw  [line width=1.5]  (146.28,21) .. controls (146.28,14.37) and (151.65,9) .. (158.28,9) .. controls (164.91,9) and (170.28,14.37) .. (170.28,21) .. controls (170.28,27.63) and (164.91,33) .. (158.28,33) .. controls (151.65,33) and (146.28,27.63) .. (146.28,21) -- cycle ;
\draw  [line width=1.5]  (146.28,49) .. controls (146.28,42.37) and (151.65,37) .. (158.28,37) .. controls (164.91,37) and (170.28,42.37) .. (170.28,49) .. controls (170.28,55.63) and (164.91,61) .. (158.28,61) .. controls (151.65,61) and (146.28,55.63) .. (146.28,49) -- cycle ;

\draw  [draw opacity=0][line width=1.5]  (268.82,36.71) .. controls (266.81,39.04) and (265.59,42.07) .. (265.59,45.39) .. controls (265.59,52.71) and (271.53,58.65) .. (278.86,58.65) .. controls (286.18,58.65) and (292.12,52.71) .. (292.12,45.39) .. controls (292.12,42.2) and (291,39.28) .. (289.13,37) -- (278.86,45.39) -- cycle ; \draw  [line width=1.5]  (268.82,36.71) .. controls (266.81,39.04) and (265.59,42.07) .. (265.59,45.39) .. controls (265.59,52.71) and (271.53,58.65) .. (278.86,58.65) .. controls (286.18,58.65) and (292.12,52.71) .. (292.12,45.39) .. controls (292.12,42.2) and (291,39.28) .. (289.13,37) ;  
\draw  [line width=1.5]  (265.56,23.39) .. controls (265.56,16.05) and (271.51,10.09) .. (278.86,10.09) .. controls (286.2,10.09) and (292.15,16.05) .. (292.15,23.39) .. controls (292.15,30.73) and (286.2,36.68) .. (278.86,36.68) .. controls (271.51,36.68) and (265.56,30.73) .. (265.56,23.39) -- cycle ;
\draw  [draw opacity=0][line width=1.5]  (282.75,32.71) .. controls (281.52,32.33) and (280.21,32.13) .. (278.86,32.13) .. controls (277.69,32.13) and (276.57,32.28) .. (275.49,32.56) -- (278.86,45.39) -- cycle ; \draw  [line width=1.5]  (282.75,32.71) .. controls (281.52,32.33) and (280.21,32.13) .. (278.86,32.13) .. controls (277.69,32.13) and (276.57,32.28) .. (275.49,32.56) ;  

\draw    (196,39.86) -- (234.5,39.86) ;
\draw [shift={(236.5,39.86)}, rotate = 180] [color={rgb, 255:red, 0; green, 0; blue, 0 }  ][line width=0.75]    (10.93,-3.29) .. controls (6.95,-1.4) and (3.31,-0.3) .. (0,0) .. controls (3.31,0.3) and (6.95,1.4) .. (10.93,3.29)   ;
\draw    (321,39.86) -- (364,39.86) ;
\draw [shift={(366,39.86)}, rotate = 180] [color={rgb, 255:red, 0; green, 0; blue, 0 }  ][line width=0.75]    (10.93,-3.29) .. controls (6.95,-1.4) and (3.31,-0.3) .. (0,0) .. controls (3.31,0.3) and (6.95,1.4) .. (10.93,3.29)   ;
\draw    (78,39.5) -- (116.5,39.5) ;
\draw [shift={(118.5,39.5)}, rotate = 180] [color={rgb, 255:red, 0; green, 0; blue, 0 }  ][line width=0.75]    (10.93,-3.29) .. controls (6.95,-1.4) and (3.31,-0.3) .. (0,0) .. controls (3.31,0.3) and (6.95,1.4) .. (10.93,3.29)   ;
\draw  [dash pattern={on 0.84pt off 2.51pt}][line width=0.75]  (27,21) .. controls (27,14.37) and (32.37,9) .. (39,9) .. controls (45.63,9) and (51,14.37) .. (51,21) .. controls (51,27.63) and (45.63,33) .. (39,33) .. controls (32.37,33) and (27,27.63) .. (27,21) -- cycle ;
\draw  [dash pattern={on 0.84pt off 2.51pt}][line width=0.75]  (27,49) .. controls (27,42.37) and (32.37,37) .. (39,37) .. controls (45.63,37) and (51,42.37) .. (51,49) .. controls (51,55.63) and (45.63,61) .. (39,61) .. controls (32.37,61) and (27,55.63) .. (27,49) -- cycle ;

\draw (216.25,36.46) node [anchor=south] [inner sep=0.75pt]    {$R_{2}$};
\draw (343.5,36.86) node [anchor=south] [inner sep=0.75pt]   [align=left] {c.c.};
\draw (39,70.9) node [anchor=north] [inner sep=0.75pt]    {$\varnothing $};
\draw (158.28,71.4) node [anchor=north] [inner sep=0.75pt]    {$U_{2}$};
\draw (280.15,71.4) node [anchor=north] [inner sep=0.75pt]    {$U'_{2}$};
\draw (401,73.4) node [anchor=north] [inner sep=0.75pt]    {$H$};

\end{tikzpicture}
    \end{center}
    This give rise to two chain maps:
    \[
\begin{tikzcd}
R = \CKh(\varnothing) \arrow[r, "\iota \otimes \iota"] & \CKh(U_2) \arrow[r, "\rho"] & \CKh(U'_2) \arrow[r, "f_0", dashed, shift left] \arrow[r, "f_1"', shift right] & \CKh(H)
\end{tikzcd}
    \]
    One can directly check that $\zeta_0$ and $\zeta_1$ are given by the images of $1 \in R$ by the two chain maps. 
\end{rem}

\begin{rem}
    A similar geometric description for Alishahi's crossing change map on Bar-Natan homology is given in \cite[Section 4]{Alishahi:2017}. 
\end{rem}

\subsection{Crossing changes and Reidemeister moves}

Here we prove the commutativity of the Reidemeister move maps and the crossing change maps. It is assumed that we use the Reidemeister move maps given in \cite[Section 4]{BarNatan:2004}. The first two propositions are easy to verify, so we omit the proofs. 

\begin{prop}
\label{prop:cc-rm1}
    Consider the following commutative diagram of moves:
    \begin{center}
        \tikzset{every picture/.style={line width=0.75pt}} 

\begin{tikzpicture}[x=0.75pt,y=0.75pt,yscale=-.7,xscale=.7]

\draw [color={rgb, 255:red, 208; green, 2; blue, 27 }  ,draw opacity=1 ][line width=1.5]    (56.8,147.21) -- (24,147.21) ;
\draw [shift={(60.8,147.21)}, rotate = 180] [fill={rgb, 255:red, 208; green, 2; blue, 27 }  ,fill opacity=1 ][line width=0.08]  [draw opacity=0] (8.13,-3.9) -- (0,0) -- (8.13,3.9) -- cycle    ;
\draw  [dash pattern={on 4.5pt off 4.5pt}] (152.96,42.3) .. controls (152.96,23.91) and (167.87,9) .. (186.26,9) .. controls (204.65,9) and (219.56,23.91) .. (219.56,42.3) .. controls (219.56,60.68) and (204.65,75.59) .. (186.26,75.59) .. controls (167.87,75.59) and (152.96,60.68) .. (152.96,42.3) -- cycle ;
\draw  [draw opacity=0][line width=1.5]  (166.9,16.27) .. controls (176.72,21.19) and (183.4,31.13) .. (183.31,42.55) .. controls (183.21,54.1) and (176.21,64.03) .. (166.11,68.72) -- (152.96,42.3) -- cycle ; \draw  [line width=1.5]  (166.9,16.27) .. controls (176.72,21.19) and (183.4,31.13) .. (183.31,42.55) .. controls (183.21,54.1) and (176.21,64.03) .. (166.11,68.72) ;  

\draw    (141.57,60.7) -- (73.07,106.27) ;
\draw [shift={(71.4,107.38)}, rotate = 326.36] [color={rgb, 255:red, 0; green, 0; blue, 0 }  ][line width=0.75]    (10.93,-3.29) .. controls (6.95,-1.4) and (3.31,-0.3) .. (0,0) .. controls (3.31,0.3) and (6.95,1.4) .. (10.93,3.29)   ;
\draw  [dash pattern={on 4.5pt off 4.5pt}] (12.17,148.21) .. controls (12.17,166.6) and (27.08,181.51) .. (45.47,181.51) .. controls (63.85,181.51) and (78.76,166.6) .. (78.76,148.21) .. controls (78.76,129.82) and (63.85,114.92) .. (45.47,114.92) .. controls (27.08,114.92) and (12.17,129.82) .. (12.17,148.21) -- cycle ;
\draw [line width=1.5]    (25.25,172.67) .. controls (35.61,144.83) and (69.23,117.55) .. (68.18,149.25) ;
\draw  [draw opacity=0][fill={rgb, 255:red, 255; green, 255; blue, 255 }  ,fill opacity=1 ] (41.1,154.2) .. controls (45.07,154.25) and (48.34,151.07) .. (48.39,147.09) .. controls (48.44,143.11) and (45.26,139.85) .. (41.29,139.8) .. controls (37.31,139.75) and (34.04,142.93) .. (33.99,146.9) .. controls (33.94,150.88) and (37.12,154.14) .. (41.1,154.2) -- cycle ;
\draw  [draw opacity=0][fill={rgb, 255:red, 0; green, 0; blue, 0 }  ,fill opacity=1 ] (40.95,149.46) .. controls (42.22,149.47) and (43.26,148.46) .. (43.28,147.19) .. controls (43.29,145.91) and (42.28,144.87) .. (41.01,144.85) .. controls (39.74,144.84) and (38.69,145.85) .. (38.68,147.12) .. controls (38.66,148.4) and (39.68,149.44) .. (40.95,149.46) -- cycle ;

\draw  [dash pattern={on 4.5pt off 4.5pt}] (300.17,148.21) .. controls (300.17,166.6) and (315.08,181.51) .. (333.47,181.51) .. controls (351.85,181.51) and (366.76,166.6) .. (366.76,148.21) .. controls (366.76,129.82) and (351.85,114.92) .. (333.47,114.92) .. controls (315.08,114.92) and (300.17,129.82) .. (300.17,148.21) -- cycle ;
\draw [line width=1.5]    (356.18,149.25) .. controls (355.13,180.95) and (324.87,150.18) .. (315.25,122.08) ;
\draw  [draw opacity=0][fill={rgb, 255:red, 255; green, 255; blue, 255 }  ,fill opacity=1 ] (329.1,154.2) .. controls (333.07,154.25) and (336.34,151.07) .. (336.39,147.09) .. controls (336.44,143.11) and (333.26,139.85) .. (329.29,139.8) .. controls (325.31,139.75) and (322.04,142.93) .. (321.99,146.9) .. controls (321.94,150.88) and (325.12,154.14) .. (329.1,154.2) -- cycle ;
\draw  [draw opacity=0][fill={rgb, 255:red, 0; green, 0; blue, 0 }  ,fill opacity=1 ] (328.95,149.46) .. controls (330.22,149.47) and (331.26,148.46) .. (331.28,147.19) .. controls (331.29,145.91) and (330.28,144.87) .. (329.01,144.85) .. controls (327.74,144.84) and (326.69,145.85) .. (326.68,147.12) .. controls (326.66,148.4) and (327.68,149.44) .. (328.95,149.46) -- cycle ;

\draw [line width=1.5]    (313.25,172.67) .. controls (323.61,144.83) and (357.23,117.55) .. (356.18,149.25) ;
\draw    (236.57,60.3) -- (303.58,109.01) ;
\draw [shift={(305.2,110.18)}, rotate = 216.01] [color={rgb, 255:red, 0; green, 0; blue, 0 }  ][line width=0.75]    (10.93,-3.29) .. controls (6.95,-1.4) and (3.31,-0.3) .. (0,0) .. controls (3.31,0.3) and (6.95,1.4) .. (10.93,3.29)   ;
\draw    (91.57,150.7) -- (288.6,150.7) ;
\draw [shift={(290.6,150.7)}, rotate = 180] [color={rgb, 255:red, 0; green, 0; blue, 0 }  ][line width=0.75]    (10.93,-3.29) .. controls (6.95,-1.4) and (3.31,-0.3) .. (0,0) .. controls (3.31,0.3) and (6.95,1.4) .. (10.93,3.29)   ;
\draw [line width=1.5]    (68.18,149.25) .. controls (67.13,180.95) and (36.87,150.18) .. (27.25,122.08) ;

\draw (186.26,78.99) node [anchor=north] [inner sep=0.75pt]    {$D$};
\draw (333.47,184.91) node [anchor=north] [inner sep=0.75pt]    {$D^{+}$};
\draw (104.48,80.64) node [anchor=south east] [inner sep=0.75pt]    {$R1_-$};
\draw (272.88,81.84) node [anchor=south west] [inner sep=0.75pt]    {$R1_+$};
\draw (191.08,147.7) node [anchor=south] [inner sep=0.75pt]   [align=left] {c.c.};
\draw (45.47,184.91) node [anchor=north] [inner sep=0.75pt]    {$D^{-}$};

\end{tikzpicture}
    \end{center}
    The corresponding diagram of maps commutes.
    \[
    \begin{tikzcd}[row sep=3em]
& \CKh(D) \arrow[ld, "R1_-"'] \arrow[rd, "R1_+"] & \\
\CKh(D^-) \arrow[rr, "f_0^+"] & & \CKh(D^+).
\end{tikzcd}
    \]
\end{prop}

\begin{rem}
    The commutativity of \Cref{prop:cc-rm1} does not hold for the other maps $f_1^\pm$ and $f_0^-$. 
\end{rem}

\begin{prop}
\label{prop:cc-rm2}
    Consider the following commutative diagrams of moves:
    \begin{center}
        \input{tikzpictures/R2-cc}
    \end{center}
    For both $i = 0, 1$, the corresponding diagrams of maps commute.
    \[
\begin{tikzcd}[row sep=3em]
\CKh(D) \arrow[d, "R2"'] \arrow[r, "R2"] & 
\CKh(D') \arrow[d, "f_i"] & 
\CKh(D) \arrow[r, "R2"] \arrow[d, "R2"'] & 
\CKh(D') \arrow[d, "f_i"] \\
\CKh(D'') \arrow[r, "f_i"] & 
{\CKh(D'''),} & 
\CKh(D'') \arrow[r, "f_i"]
& \CKh(D'''').                 
\end{tikzcd}
    \]
\end{prop}

\begin{prop}
\label{prop:cc-rm3}
    Consider the following commutative diagrams of moves:
    \begin{center}
        \input{tikzpictures/R3-cc}
    \end{center}
    For both $i = 0, 1$, the corresponding diagram of maps commutes up to chain homotopy.
    \[
\begin{tikzcd}[row sep=3em]
\CKh(D) \arrow[d, "f_i"', leftrightarrow] \arrow[r, "R3"] & \CKh(D') \arrow[d, "f_i", leftrightarrow] \\
\CKh(D'') \arrow[r, "R3"]                 & \CKh(D''').               
\end{tikzcd}
\]
    Here, the vertical arrows pointing both directions indicate that the statement hold for both the top-to-bottom direction and the reversed direction. 
\end{prop}

\begin{proof}
    Recall from \cite[Section 4]{BarNatan:2004} that the R3 map $\CKh(D) \rightarrow \CKh(D')$ is given by constructing a chain complex $E$ that is a strong deformation retract of both $\CKh(D)$ and $\CKh(D')$, and then composing the chain homotopy equivalences
    \[
        \CKh(D) \rightarrow E \rightarrow \CKh(D').
    \]
    Similarly, the R3 map $\CKh(D'') \rightarrow \CKh(D''')$ is given by
    \[
        \CKh(D'') \rightarrow E' \rightarrow \CKh(D''').
    \]
    Thus it suffices to prove that the following hexagon commutes up to chain homotopy.
    \[
\begin{tikzcd} & E \arrow[ld, "\simeq"'] \arrow[rd, "\simeq"] & \\ \CKh(D) \arrow[d, "f_i"', leftrightarrow] \arrow[rr, "R3", dotted] & & \CKh(D') \arrow[d, "f_i", leftrightarrow] \\ \CKh(D'') \arrow[rr, "R3", dotted] \arrow[rd, "\simeq"'] & & \CKh(D''') \arrow[ld, "\simeq"] \\ & E' & \end{tikzcd}
    \]
    This can be verified by unraveling the chain maps. 
\end{proof}

\begin{rem}
\label{rem:sign-ambiguity}
    From \Cref{prop:cc-rm2}, one can see that the direction for a crossing cannot be uniquely determined by the local orientations so that it gives the directions specified in the two commutative diagrams.  Nonetheless, for negative crossings, we may fix the direction as
    \begin{center}
        \tikzset{every picture/.style={line width=0.75pt}} 

\begin{tikzpicture}[x=0.75pt,y=0.75pt,yscale=-.75,xscale=.75]

\draw [color={rgb, 255:red, 208; green, 2; blue, 27 }  ,draw opacity=1 ][line width=1.5]    (21.33,40) -- (56.67,40) ;
\draw [shift={(60.67,40)}, rotate = 180] [fill={rgb, 255:red, 208; green, 2; blue, 27 }  ,fill opacity=1 ][line width=0.08]  [draw opacity=0] (9.29,-4.46) -- (0,0) -- (9.29,4.46) -- cycle    ;
\draw [color={rgb, 255:red, 0; green, 0; blue, 0 }  ,draw opacity=1 ][line width=1.5]    (61.58,16.81) -- (44.75,33.4) -- (13.71,64) ;
\draw [shift={(64.43,14)}, rotate = 135.41] [fill={rgb, 255:red, 0; green, 0; blue, 0 }  ,fill opacity=1 ][line width=0.08]  [draw opacity=0] (6.97,-3.35) -- (0,0) -- (6.97,3.35) -- cycle    ;
\draw  [dash pattern={on 4.5pt off 4.5pt}] (4,40) .. controls (4,20.12) and (20.12,4) .. (40,4) .. controls (59.88,4) and (76,20.12) .. (76,40) .. controls (76,59.88) and (59.88,76) .. (40,76) .. controls (20.12,76) and (4,59.88) .. (4,40) -- cycle ;
\draw  [draw opacity=0][fill={rgb, 255:red, 255; green, 255; blue, 255 }  ,fill opacity=1 ] (33.71,39.71) .. controls (33.71,36.36) and (36.43,33.64) .. (39.79,33.64) .. controls (43.14,33.64) and (45.86,36.36) .. (45.86,39.71) .. controls (45.86,43.07) and (43.14,45.79) .. (39.79,45.79) .. controls (36.43,45.79) and (33.71,43.07) .. (33.71,39.71) -- cycle ;
\draw  [draw opacity=0][fill={rgb, 255:red, 0; green, 0; blue, 0 }  ,fill opacity=1 ] (35.9,39.71) .. controls (35.9,38.12) and (37.19,36.83) .. (38.79,36.83) .. controls (40.38,36.83) and (41.67,38.12) .. (41.67,39.71) .. controls (41.67,41.31) and (40.38,42.6) .. (38.79,42.6) .. controls (37.19,42.6) and (35.9,41.31) .. (35.9,39.71) -- cycle ;
\draw [color={rgb, 255:red, 0; green, 0; blue, 0 }  ,draw opacity=1 ][line width=1.5]    (16.9,16.83) -- (64.07,64) ;
\draw [shift={(14.07,14)}, rotate = 45] [fill={rgb, 255:red, 0; green, 0; blue, 0 }  ,fill opacity=1 ][line width=0.08]  [draw opacity=0] (6.97,-3.35) -- (0,0) -- (6.97,3.35) -- cycle    ;

\draw (33.67,10.24) node [anchor=north west][inner sep=0.75pt]    {$-$};

\end{tikzpicture}
    \end{center}
    so that it matches the directions specified in the two commutative diagrams. We refer to this direction as the \textit{preferred direction} for a negative crossing. 
\end{rem}

As stated in \Cref{sec:intro}, a homotopy of a normally immersed surface can be realized by a finite sequence of smooth ambient isotopies, together with the following three moves and their inverses: (i) \textit{negative twist moves}, (ii) \textit{positive twist moves}, and (iii) \textit{finger moves}. Furthermore, these moves can be realized by combinations of Reidemeister moves and crossing change moves, and hence we may associate maps accordingly. The following results are again easy to verify, and will be used to prove \Cref{formula_of_moves}. 

\begin{prop}
\label{prop:twist-move-map}
    Consider the following sequence of local moves:
    \begin{center}
        \tikzset{every picture/.style={line width=0.75pt}} 

\begin{tikzpicture}[x=0.75pt,y=0.75pt,yscale=-.8,xscale=.8]

\draw  [dash pattern={on 4.5pt off 4.5pt}] (8.96,43.21) .. controls (8.96,24.82) and (23.87,9.92) .. (42.26,9.92) .. controls (60.65,9.92) and (75.56,24.82) .. (75.56,43.21) .. controls (75.56,61.6) and (60.65,76.51) .. (42.26,76.51) .. controls (23.87,76.51) and (8.96,61.6) .. (8.96,43.21) -- cycle ;
\draw  [draw opacity=0][line width=1.5]  (22.9,17.18) .. controls (32.72,22.1) and (39.4,32.05) .. (39.31,43.47) .. controls (39.21,55.02) and (32.21,64.95) .. (22.11,69.63) -- (8.96,43.21) -- cycle ; \draw  [line width=1.5]  (22.9,17.18) .. controls (32.72,22.1) and (39.4,32.05) .. (39.31,43.47) .. controls (39.21,55.02) and (32.21,64.95) .. (22.11,69.63) ;  

\draw  [dash pattern={on 4.5pt off 4.5pt}] (274.17,42.21) .. controls (274.17,60.6) and (289.08,75.51) .. (307.47,75.51) .. controls (325.85,75.51) and (340.76,60.6) .. (340.76,42.21) .. controls (340.76,23.82) and (325.85,8.92) .. (307.47,8.92) .. controls (289.08,8.92) and (274.17,23.82) .. (274.17,42.21) -- cycle ;
\draw [line width=1.5]    (330.18,43.25) .. controls (329.13,74.95) and (298.87,44.18) .. (289.25,16.08) ;
\draw  [draw opacity=0][fill={rgb, 255:red, 255; green, 255; blue, 255 }  ,fill opacity=1 ] (303.1,48.2) .. controls (307.07,48.25) and (310.34,45.07) .. (310.39,41.09) .. controls (310.44,37.11) and (307.26,33.85) .. (303.29,33.8) .. controls (299.31,33.75) and (296.04,36.93) .. (295.99,40.9) .. controls (295.94,44.88) and (299.12,48.14) .. (303.1,48.2) -- cycle ;
\draw  [draw opacity=0][fill={rgb, 255:red, 0; green, 0; blue, 0 }  ,fill opacity=1 ] (302.95,43.46) .. controls (304.22,43.47) and (305.26,42.46) .. (305.28,41.19) .. controls (305.29,39.91) and (304.28,38.87) .. (303.01,38.85) .. controls (301.74,38.84) and (300.69,39.85) .. (300.68,41.12) .. controls (300.66,42.4) and (301.68,43.44) .. (302.95,43.46) -- cycle ;

\draw [line width=1.5]    (287.25,66.67) .. controls (297.61,38.83) and (331.23,11.55) .. (330.18,43.25) ;

\draw    (88,45.5) -- (132,45.5) ;
\draw [shift={(134,45.5)}, rotate = 180] [color={rgb, 255:red, 0; green, 0; blue, 0 }  ][line width=0.75]    (10.93,-3.29) .. controls (6.95,-1.4) and (3.31,-0.3) .. (0,0) .. controls (3.31,0.3) and (6.95,1.4) .. (10.93,3.29)   ;
\draw    (224.88,45.5) -- (264,45.5) ;
\draw [shift={(266,45.5)}, rotate = 180] [color={rgb, 255:red, 0; green, 0; blue, 0 }  ][line width=0.75]    (10.93,-3.29) .. controls (6.95,-1.4) and (3.31,-0.3) .. (0,0) .. controls (3.31,0.3) and (6.95,1.4) .. (10.93,3.29)   ;
\draw    (353.01,45.5) -- (392,45.53) ;
\draw [shift={(394,45.53)}, rotate = 180.04] [color={rgb, 255:red, 0; green, 0; blue, 0 }  ][line width=0.75]    (10.93,-3.29) .. controls (6.95,-1.4) and (3.31,-0.3) .. (0,0) .. controls (3.31,0.3) and (6.95,1.4) .. (10.93,3.29)   ;
\draw  [dash pattern={on 4.5pt off 4.5pt}] (147.17,42.21) .. controls (147.17,60.6) and (162.08,75.51) .. (180.47,75.51) .. controls (198.85,75.51) and (213.76,60.6) .. (213.76,42.21) .. controls (213.76,23.82) and (198.85,8.92) .. (180.47,8.92) .. controls (162.08,8.92) and (147.17,23.82) .. (147.17,42.21) -- cycle ;
\draw [line width=1.5]    (160.25,66.67) .. controls (170.61,38.83) and (204.23,11.55) .. (203.18,43.25) ;
\draw  [draw opacity=0][fill={rgb, 255:red, 255; green, 255; blue, 255 }  ,fill opacity=1 ] (176.1,48.2) .. controls (180.07,48.25) and (183.34,45.07) .. (183.39,41.09) .. controls (183.44,37.11) and (180.26,33.85) .. (176.29,33.8) .. controls (172.31,33.75) and (169.04,36.93) .. (168.99,40.9) .. controls (168.94,44.88) and (172.12,48.14) .. (176.1,48.2) -- cycle ;
\draw  [draw opacity=0][fill={rgb, 255:red, 0; green, 0; blue, 0 }  ,fill opacity=1 ] (175.95,43.46) .. controls (177.22,43.47) and (178.26,42.46) .. (178.28,41.19) .. controls (178.29,39.91) and (177.28,38.87) .. (176.01,38.85) .. controls (174.74,38.84) and (173.69,39.85) .. (173.68,41.12) .. controls (173.66,42.4) and (174.68,43.44) .. (175.95,43.46) -- cycle ;

\draw [line width=1.5]    (203.18,43.25) .. controls (202.13,74.95) and (171.87,44.18) .. (162.25,16.08) ;

\draw  [dash pattern={on 4.5pt off 4.5pt}] (403.96,43.21) .. controls (403.96,24.82) and (418.87,9.92) .. (437.26,9.92) .. controls (455.65,9.92) and (470.56,24.82) .. (470.56,43.21) .. controls (470.56,61.6) and (455.65,76.51) .. (437.26,76.51) .. controls (418.87,76.51) and (403.96,61.6) .. (403.96,43.21) -- cycle ;
\draw  [draw opacity=0][line width=1.5]  (417.9,17.18) .. controls (427.72,22.1) and (434.4,32.05) .. (434.31,43.47) .. controls (434.21,55.02) and (427.21,64.95) .. (417.11,69.63) -- (403.96,43.21) -- cycle ; \draw  [line width=1.5]  (417.9,17.18) .. controls (427.72,22.1) and (434.4,32.05) .. (434.31,43.47) .. controls (434.21,55.02) and (427.21,64.95) .. (417.11,69.63) ;

\draw (42.26,79.91) node [anchor=north] [inner sep=0.75pt]    {$D$};
\draw (437.26,79.91) node [anchor=north] [inner sep=0.75pt]    {$D$};
\draw (307.47,78.91) node [anchor=north] [inner sep=0.75pt]    {$D^{+}$};
\draw (180.47,78.91) node [anchor=north] [inner sep=0.75pt]    {$D^{-}$};
\draw (245.44,40) node [anchor=south] [inner sep=0.75pt]   [align=left] {c.c.};
\draw (111,42.1) node [anchor=south] [inner sep=0.75pt]    {$R1_-$};
\draw (373.51,42.12) node [anchor=south] [inner sep=0.75pt]    {$R1_+^{-1}$};

\end{tikzpicture}
    \end{center}
    The corresponding sequence of maps on $\CKh$:
    \[
    \begin{tikzcd}
        \CKh(D) \arrow[r, "R1_-"] & \CKh(D^-) \arrow[r, "f^+_i"] & \CKh(D^+) \arrow[r, "R1^{-1}_+"] & \CKh(D)
    \end{tikzcd}    
    \]
    composes to $\id$ for $i = 0$ (using the preferred direction), and $0$ for $i = 1$. 
    For the reversed move, the corresponding sequence of maps on $\CKh$:
    \[
    \begin{tikzcd}
    \CKh(D) & \CKh(D^-) \arrow[l, "R1^{-1}_-"'] & \CKh(D^+) \arrow[l, "f^-_i"'] & \CKh(D) \arrow[l, "R1_+"']
    \end{tikzcd}
    \]
    composes to $0$ for $i = 0$, and $\pm(2X - h)$ for $i = 1$. 
\end{prop}

\begin{prop}
\label{prop:finger-move-map}
    Consider the following sequence of local moves:
    \begin{center}
        \tikzset{every picture/.style={line width=0.75pt}} 

\begin{tikzpicture}[x=0.75pt,y=0.75pt,yscale=-.8,xscale=.8]

\draw    (88,45.5) -- (132,45.5) ;
\draw [shift={(134,45.5)}, rotate = 180] [color={rgb, 255:red, 0; green, 0; blue, 0 }  ][line width=0.75]    (10.93,-3.29) .. controls (6.95,-1.4) and (3.31,-0.3) .. (0,0) .. controls (3.31,0.3) and (6.95,1.4) .. (10.93,3.29)   ;
\draw    (224.88,45.5) -- (264,45.5) ;
\draw [shift={(266,45.5)}, rotate = 180] [color={rgb, 255:red, 0; green, 0; blue, 0 }  ][line width=0.75]    (10.93,-3.29) .. controls (6.95,-1.4) and (3.31,-0.3) .. (0,0) .. controls (3.31,0.3) and (6.95,1.4) .. (10.93,3.29)   ;
\draw    (353.01,45.5) -- (392,45.53) ;
\draw [shift={(394,45.53)}, rotate = 180.04] [color={rgb, 255:red, 0; green, 0; blue, 0 }  ][line width=0.75]    (10.93,-3.29) .. controls (6.95,-1.4) and (3.31,-0.3) .. (0,0) .. controls (3.31,0.3) and (6.95,1.4) .. (10.93,3.29)   ;
\draw  [dash pattern={on 4.5pt off 4.5pt}] (12,45.7) .. controls (12,27.32) and (26.91,12.41) .. (45.3,12.41) .. controls (63.68,12.41) and (78.59,27.32) .. (78.59,45.7) .. controls (78.59,64.09) and (63.68,79) .. (45.3,79) .. controls (26.91,79) and (12,64.09) .. (12,45.7) -- cycle ;
\draw  [draw opacity=0][line width=1.5]  (21.14,66.27) .. controls (26.44,60.38) and (35.3,56.53) .. (45.34,56.55) .. controls (55.54,56.58) and (64.51,60.58) .. (69.74,66.65) -- (45.3,79) -- cycle ; \draw  [line width=1.5]  (21.14,66.27) .. controls (26.44,60.38) and (35.3,56.53) .. (45.34,56.55) .. controls (55.54,56.58) and (64.51,60.58) .. (69.74,66.65) ;  
\draw  [draw opacity=0][line width=1.5]  (69.75,25.13) .. controls (64.52,31.38) and (55.56,35.52) .. (45.38,35.55) .. controls (35.03,35.59) and (25.93,31.39) .. (20.69,25.01) -- (45.3,12.41) -- cycle ; \draw  [line width=1.5]  (69.75,25.13) .. controls (64.52,31.38) and (55.56,35.52) .. (45.38,35.55) .. controls (35.03,35.59) and (25.93,31.39) .. (20.69,25.01) ;  

\draw  [dash pattern={on 4.5pt off 4.5pt}] (146.51,45.71) .. controls (146.51,27.32) and (161.42,12.42) .. (179.81,12.42) .. controls (198.2,12.42) and (213.1,27.32) .. (213.1,45.71) .. controls (213.1,64.1) and (198.2,79.01) .. (179.81,79.01) .. controls (161.42,79.01) and (146.51,64.1) .. (146.51,45.71) -- cycle ;
\draw  [draw opacity=0][line width=1.5]  (206.72,26.88) .. controls (202.17,40.17) and (191.62,49.44) .. (179.43,49.31) .. controls (167.13,49.19) and (156.69,39.53) .. (152.49,25.95) -- (179.81,12.42) -- cycle ; \draw  [line width=1.5]  (206.72,26.88) .. controls (202.17,40.17) and (191.62,49.44) .. (179.43,49.31) .. controls (167.13,49.19) and (156.69,39.53) .. (152.49,25.95) ;  
\draw  [draw opacity=0][fill={rgb, 255:red, 255; green, 255; blue, 255 }  ,fill opacity=1 ] (201.58,44.3) .. controls (201.55,48.28) and (198.3,51.47) .. (194.32,51.44) .. controls (190.35,51.41) and (187.15,48.16) .. (187.18,44.18) .. controls (187.21,40.21) and (190.46,37.01) .. (194.44,37.04) .. controls (198.42,37.07) and (201.61,40.32) .. (201.58,44.3) -- cycle ;
\draw  [draw opacity=0][fill={rgb, 255:red, 0; green, 0; blue, 0 }  ,fill opacity=1 ] (196.84,44.05) .. controls (196.83,45.32) and (195.8,46.34) .. (194.53,46.33) .. controls (193.25,46.32) and (192.23,45.28) .. (192.24,44.01) .. controls (192.25,42.74) and (193.29,41.72) .. (194.56,41.73) .. controls (195.83,41.74) and (196.85,42.78) .. (196.84,44.05) -- cycle ;

\draw  [draw opacity=0][fill={rgb, 255:red, 255; green, 255; blue, 255 }  ,fill opacity=1 ] (170.59,43.05) .. controls (170.56,47.03) and (167.31,50.23) .. (163.33,50.19) .. controls (159.36,50.16) and (156.16,46.91) .. (156.19,42.94) .. controls (156.22,38.96) and (159.47,35.76) .. (163.45,35.79) .. controls (167.42,35.82) and (170.62,39.07) .. (170.59,43.05) -- cycle ;
\draw  [draw opacity=0][line width=1.5]  (150.77,63) .. controls (155.03,47.68) and (165.77,36.83) .. (178.3,36.88) .. controls (190.94,36.94) and (201.66,48.09) .. (205.7,63.65) -- (178.11,77.51) -- cycle ; \draw  [line width=1.5]  (150.77,63) .. controls (155.03,47.68) and (165.77,36.83) .. (178.3,36.88) .. controls (190.94,36.94) and (201.66,48.09) .. (205.7,63.65) ;  
\draw  [draw opacity=0][fill={rgb, 255:red, 0; green, 0; blue, 0 }  ,fill opacity=1 ] (165.85,42.8) .. controls (165.84,44.07) and (164.81,45.09) .. (163.53,45.08) .. controls (162.26,45.07) and (161.24,44.03) .. (161.25,42.76) .. controls (161.26,41.49) and (162.3,40.47) .. (163.57,40.48) .. controls (164.84,40.49) and (165.86,41.53) .. (165.85,42.8) -- cycle ;

\draw  [draw opacity=0][line width=1.5]  (279.77,63) .. controls (284.03,47.68) and (294.77,36.83) .. (307.3,36.88) .. controls (319.94,36.94) and (330.66,48.09) .. (334.7,63.65) -- (307.11,77.51) -- cycle ; \draw  [line width=1.5]  (279.77,63) .. controls (284.03,47.68) and (294.77,36.83) .. (307.3,36.88) .. controls (319.94,36.94) and (330.66,48.09) .. (334.7,63.65) ;  
\draw  [dash pattern={on 4.5pt off 4.5pt}] (275.51,45.71) .. controls (275.51,27.32) and (290.42,12.42) .. (308.81,12.42) .. controls (327.2,12.42) and (342.1,27.32) .. (342.1,45.71) .. controls (342.1,64.1) and (327.2,79.01) .. (308.81,79.01) .. controls (290.42,79.01) and (275.51,64.1) .. (275.51,45.71) -- cycle ;
\draw  [draw opacity=0][fill={rgb, 255:red, 255; green, 255; blue, 255 }  ,fill opacity=1 ] (330.58,44.3) .. controls (330.55,48.28) and (327.3,51.47) .. (323.32,51.44) .. controls (319.35,51.41) and (316.15,48.16) .. (316.18,44.18) .. controls (316.21,40.21) and (319.46,37.01) .. (323.44,37.04) .. controls (327.42,37.07) and (330.61,40.32) .. (330.58,44.3) -- cycle ;
\draw  [draw opacity=0][fill={rgb, 255:red, 255; green, 255; blue, 255 }  ,fill opacity=1 ] (299.59,43.05) .. controls (299.56,47.03) and (296.31,50.23) .. (292.33,50.19) .. controls (288.36,50.16) and (285.16,46.91) .. (285.19,42.94) .. controls (285.22,38.96) and (288.47,35.76) .. (292.45,35.79) .. controls (296.42,35.82) and (299.62,39.07) .. (299.59,43.05) -- cycle ;
\draw  [draw opacity=0][fill={rgb, 255:red, 0; green, 0; blue, 0 }  ,fill opacity=1 ] (294.85,42.8) .. controls (294.84,44.07) and (293.81,45.09) .. (292.53,45.08) .. controls (291.26,45.07) and (290.24,44.03) .. (290.25,42.76) .. controls (290.26,41.49) and (291.3,40.47) .. (292.57,40.48) .. controls (293.84,40.49) and (294.86,41.53) .. (294.85,42.8) -- cycle ;

\draw  [draw opacity=0][line width=1.5]  (335.72,26.88) .. controls (331.17,40.17) and (320.62,49.44) .. (308.43,49.31) .. controls (296.13,49.19) and (285.69,39.53) .. (281.49,25.95) -- (308.81,12.42) -- cycle ; \draw  [line width=1.5]  (335.72,26.88) .. controls (331.17,40.17) and (320.62,49.44) .. (308.43,49.31) .. controls (296.13,49.19) and (285.69,39.53) .. (281.49,25.95) ;  
\draw  [draw opacity=0][fill={rgb, 255:red, 0; green, 0; blue, 0 }  ,fill opacity=1 ] (325.84,44.05) .. controls (325.83,45.32) and (324.8,46.34) .. (323.53,46.33) .. controls (322.25,46.32) and (321.23,45.28) .. (321.24,44.01) .. controls (321.25,42.74) and (322.29,41.72) .. (323.56,41.73) .. controls (324.83,41.74) and (325.85,42.78) .. (325.84,44.05) -- cycle ;

\draw  [dash pattern={on 4.5pt off 4.5pt}] (402,46.7) .. controls (402,28.32) and (416.91,13.41) .. (435.3,13.41) .. controls (453.68,13.41) and (468.59,28.32) .. (468.59,46.7) .. controls (468.59,65.09) and (453.68,80) .. (435.3,80) .. controls (416.91,80) and (402,65.09) .. (402,46.7) -- cycle ;
\draw  [draw opacity=0][line width=1.5]  (411.14,67.27) .. controls (416.44,61.38) and (425.3,57.53) .. (435.34,57.55) .. controls (445.54,57.58) and (454.51,61.58) .. (459.74,67.65) -- (435.3,80) -- cycle ; \draw  [line width=1.5]  (411.14,67.27) .. controls (416.44,61.38) and (425.3,57.53) .. (435.34,57.55) .. controls (445.54,57.58) and (454.51,61.58) .. (459.74,67.65) ;  
\draw  [draw opacity=0][line width=1.5]  (459.75,26.13) .. controls (454.52,32.38) and (445.56,36.52) .. (435.38,36.55) .. controls (425.03,36.59) and (415.93,32.39) .. (410.69,26.01) -- (435.3,13.41) -- cycle ; \draw  [line width=1.5]  (459.75,26.13) .. controls (454.52,32.38) and (445.56,36.52) .. (435.38,36.55) .. controls (425.03,36.59) and (415.93,32.39) .. (410.69,26.01) ;

\draw (245.44,42.5) node [anchor=south] [inner sep=0.75pt]   [align=left] {c.c.};
\draw (111,42.1) node [anchor=south] [inner sep=0.75pt]    {$R2$};
\draw (373.51,42.12) node [anchor=south] [inner sep=0.75pt]    {$R2^{-1}$};
\draw (45.3,82.4) node [anchor=north] [inner sep=0.75pt]    {$D$};
\draw (435.3,83.4) node [anchor=north] [inner sep=0.75pt]    {$D$};
\draw (179.81,82.41) node [anchor=north] [inner sep=0.75pt]    {$D'$};
\draw (308.81,82.41) node [anchor=north] [inner sep=0.75pt]    {$D''$};

\end{tikzpicture}
    \end{center}
    The corresponding sequence of maps on $\CKh$:
    \[
    \begin{tikzcd}
        \CKh(D) \arrow[r, "R2"] & \CKh(D') \arrow[r, "{(f_i, f_j)}"] & \CKh(D'') \arrow[r, "R2^{-1}"] & \CKh(D)
    \end{tikzcd}    
    \]
    composes to the following map: 
    \[
    \begin{cases}
        \pm \Phi & \text{\ if $(i, j) = (0, 1)$, } \\
        \pm \Psi & \text{\ if $(i, j) = (1, 0)$, } \\
        0    & \text{\ otherwise.}
    \end{cases}
    \]
    Here, $\Phi$ is the endomorphism defined in \Cref{subsec:x-ch-map}, and $\Psi$ is another endomorphism defined by: 
    \begin{center}
        \tikzset{every picture/.style={line width=0.75pt}} 

\begin{tikzpicture}[x=0.75pt,y=0.75pt,yscale=-.8,xscale=.8]

\draw  [dash pattern={on 4.5pt off 4.5pt}] (69,34.43) .. controls (69,20.37) and (80.4,8.97) .. (94.46,8.97) .. controls (108.52,8.97) and (119.91,20.37) .. (119.91,34.43) .. controls (119.91,48.49) and (108.52,59.89) .. (94.46,59.89) .. controls (80.4,59.89) and (69,48.49) .. (69,34.43) -- cycle ;
\draw  [draw opacity=0][line width=1.5]  (114.56,19.91) .. controls (109.19,23.49) and (102.17,25.66) .. (94.48,25.66) .. controls (86.83,25.66) and (79.84,23.51) .. (74.48,19.96) -- (94.48,2.39) -- cycle ; \draw  [line width=1.5]  (114.56,19.91) .. controls (109.19,23.49) and (102.17,25.66) .. (94.48,25.66) .. controls (86.83,25.66) and (79.84,23.51) .. (74.48,19.96) ;  
\draw  [draw opacity=0][line width=1.5]  (74.4,48.96) .. controls (79.77,45.37) and (86.79,43.2) .. (94.48,43.2) .. controls (102.13,43.2) and (109.12,45.35) .. (114.48,48.9) -- (94.48,66.48) -- cycle ; \draw  [line width=1.5]  (74.4,48.96) .. controls (79.77,45.37) and (86.79,43.2) .. (94.48,43.2) .. controls (102.13,43.2) and (109.12,45.35) .. (114.48,48.9) ;  

\draw  [fill={rgb, 255:red, 0; green, 0; blue, 0 }  ,fill opacity=1 ] (91,25.43) .. controls (91,23.66) and (92.43,22.23) .. (94.2,22.23) .. controls (95.97,22.23) and (97.41,23.66) .. (97.41,25.43) .. controls (97.41,27.2) and (95.97,28.64) .. (94.2,28.64) .. controls (92.43,28.64) and (91,27.2) .. (91,25.43) -- cycle ;

\draw  [dash pattern={on 4.5pt off 4.5pt}] (170,34.43) .. controls (170,20.37) and (181.4,8.97) .. (195.46,8.97) .. controls (209.52,8.97) and (220.91,20.37) .. (220.91,34.43) .. controls (220.91,48.49) and (209.52,59.89) .. (195.46,59.89) .. controls (181.4,59.89) and (170,48.49) .. (170,34.43) -- cycle ;
\draw  [draw opacity=0][line width=1.5]  (215.56,19.91) .. controls (210.19,23.49) and (203.17,25.66) .. (195.48,25.66) .. controls (187.83,25.66) and (180.84,23.51) .. (175.48,19.96) -- (195.48,2.39) -- cycle ; \draw  [line width=1.5]  (215.56,19.91) .. controls (210.19,23.49) and (203.17,25.66) .. (195.48,25.66) .. controls (187.83,25.66) and (180.84,23.51) .. (175.48,19.96) ;  
\draw  [draw opacity=0][line width=1.5]  (175.4,48.96) .. controls (180.77,45.37) and (187.79,43.2) .. (195.48,43.2) .. controls (203.13,43.2) and (210.12,45.35) .. (215.48,48.9) -- (195.48,66.48) -- cycle ; \draw  [line width=1.5]  (175.4,48.96) .. controls (180.77,45.37) and (187.79,43.2) .. (195.48,43.2) .. controls (203.13,43.2) and (210.12,45.35) .. (215.48,48.9) ;  

\draw  [fill={rgb, 255:red, 0; green, 0; blue, 0 }  ,fill opacity=1 ] (192.25,43.63) .. controls (192.25,41.86) and (193.69,40.43) .. (195.46,40.43) .. controls (197.23,40.43) and (198.66,41.86) .. (198.66,43.63) .. controls (198.66,45.4) and (197.23,46.84) .. (195.46,46.84) .. controls (193.69,46.84) and (192.25,45.4) .. (192.25,43.63) -- cycle ;

\draw  [dash pattern={on 4.5pt off 4.5pt}] (277,34.43) .. controls (277,20.37) and (288.4,8.97) .. (302.46,8.97) .. controls (316.52,8.97) and (327.91,20.37) .. (327.91,34.43) .. controls (327.91,48.49) and (316.52,59.89) .. (302.46,59.89) .. controls (288.4,59.89) and (277,48.49) .. (277,34.43) -- cycle ;
\draw  [draw opacity=0][line width=1.5]  (322.56,19.91) .. controls (317.19,23.49) and (310.17,25.66) .. (302.48,25.66) .. controls (294.83,25.66) and (287.84,23.51) .. (282.48,19.96) -- (302.48,2.39) -- cycle ; \draw  [line width=1.5]  (322.56,19.91) .. controls (317.19,23.49) and (310.17,25.66) .. (302.48,25.66) .. controls (294.83,25.66) and (287.84,23.51) .. (282.48,19.96) ;  
\draw  [draw opacity=0][line width=1.5]  (282.4,48.96) .. controls (287.77,45.37) and (294.79,43.2) .. (302.48,43.2) .. controls (310.13,43.2) and (317.12,45.35) .. (322.48,48.9) -- (302.48,66.48) -- cycle ; \draw  [line width=1.5]  (282.4,48.96) .. controls (287.77,45.37) and (294.79,43.2) .. (302.48,43.2) .. controls (310.13,43.2) and (317.12,45.35) .. (322.48,48.9) ;

\draw (121.91,32.27) node [anchor=west] [inner sep=0.75pt]    {$\ \ \ +\ \ \ $};
\draw (222.91,33.27) node [anchor=west] [inner sep=0.75pt]    {$\ \ \ -\ h$};

\end{tikzpicture}
    \end{center}
\end{prop}

\subsection{Immersed cobordism maps}
\label{subsubsec:imm-cob-combi}

Let $S$ be a normally immersed (possibly non-orientable) cobordism in $I \times \R^3$ between links $L, L'$ in $S^3$. Let $\chi$ denote the Euler characteristic, $e$ the normal Euler number and $s_\pm$ the positive and negative double points of $S$. By standard arguments, $S$ may be isotoped so that it decomposes into elementary cobordisms $S_1, \ldots, S_N$, such that 
\begin{enumerate}
    \item each $S_i$ has links $L_i, L_{i+1}$ in both ends of its boundary, 
    \item 
    those projections $D_i, D_{i+1}$ in $\R^2$ are regular link diagrams, and 
    \item the transition from $D_i$ to $D_{i+1}$ is realized by a single elementary move: a Reidemeister move, a Morse move or a crossing change. 
\end{enumerate}
Here, we assume that $S$ is decomposed in this form, and that each double point is assigned a direction.

The corresponding cobordism map on $\CKh$ is defined by the composition of the chain maps corresponding to the elementary cobordisms. For the Reidemeister moves and the Morse moves, the maps are explicitly given in \cite{BarNatan:2004}. If $S$ has no double points, then the composition gives the standard embedded cobordism map $\phi_S$, which has 
\[
    \deg \phi_S = (e/2,\ \chi - 3e/2)
\]
(see \cite[Corollary 3.3]{LS:2022-mixed}). If $S$ has double points, then we may choose any crossing change map for each of them, which corresponds to a cycle $z \in \CKh(H^\pm)$. Thus, given $(s_+ + s_-)$-tuple of cycles $\mathbf{z} = (z^+_1, \ldots, z^+_{s_+}; z^-_1, \ldots, z^-_{s_-})$, we obtain an immersed cobordism map
\[
    \phi(S; \mathbf{z})\colon \CKh(D) \to \CKh(D').
\]
We say $\phi(S; \mathbf{z})$ is \textit{homogeneous} if each $z_i \in \CKh(H^\pm)$ is homogeneous. In such case the bigradings of the cycles $z^\pm_i$ add up to give the bidegree of $\deg \phi(S; \mathbf{z})$. 
Note that the directions are used to fix the sign for the crossing change map $f_0$, so if $\phi(S; \mathbf{z})$ is homogeneous, then the map can be defined, up to sign, without the directions. 

One natural choice for a homogeneous immersed cobordism would be to choose $\zeta_0$ for all positive double points and $\zeta_1$ for all negative double points, that is 
\[
    \phi^{\bal}_S := \phi(S; \zeta_0, \ldots, \zeta_0; \zeta_1, \ldots, \zeta_1)
\]
which has 
\[
    \deg \phi^{\bal}_S = (e/2,\ \chi - 3e/2 - 2s_-).  
\]
We call $\phi^{\bal}_S$ the immersed cobordism map of \textit{balanced homological degree}. We may also consider two extreme choices, where we take either one of $\zeta_0$ or $\zeta_1$ for all double points of $S$, which are 
\begin{align*}
    \phi^{\low}_S &:= \phi(S; \zeta_0, \ldots, \zeta_0; \zeta_0, \ldots, \zeta_0), \\
    \phi^{\hi}_S &:= \phi(S; \zeta_1, \ldots, \zeta_1; \zeta_1, \ldots, \zeta_1).    
\end{align*}
These maps have
\begin{align*}
    \deg \phi^{\low}_S &= (e/2 - 2s_-,\ \chi - 3 e/2 - 6s_-),\\
    \deg \phi^{\hi}_S &= (e/2 + 2s_+,\ \chi - 3 e/2 + 4s_+ - 2s_-).   
\end{align*}
We call $\phi^{\low}_S$ (resp.\ $\phi^{\hi}_S$) the immersed cobordism map of \textit{lowest (resp.\ highest) homological degree}. 

The induced map of $\phi(S; \mathbf{z})$ on homology is denoted $\Kh(S; \mathbf{z})$, and in particular those of $\phi^{\bal}_S$, $\phi^{\low}_S$ and $\phi^{\hi}_S$ are denoted $\Kh^{\bal}(S)$, $\Kh^{\low}(S)$ and $\Kh^{\hi}(S)$ respectively. 
The following proposition implies the first part of \Cref{thm:isotopy-invariance}. 

\begin{prop}
\label{prop:isotopy-inv-up-to-sign}
    Any homogeneous immersed cobordism map $\phi(S; \mathbf{z})$ is invariant (up to chain homotopy and sign) under smooth isotopies of $S$ rel $\partial$.
\end{prop}

\begin{proof}
    The \textit{extended movie moves} MM1 - MM18 for immersed surfaces are given in \cite[Theorem 4.1]{CCKK:2025-immersed}. It suffices to prove that for each of the moves, the corresponding pair of cobordisms $S, S'$ give equivalent maps $\phi(S; z)$ and $\phi(S'; z)$, up to chain homotopy and sign. Here, $z$ is chosen arbitrary for the additional moves MM16 - MM18 that involve a single double point. For the ordinary movie moves MM1 - MM15, the results are already proved in \cite[Theorem 4]{BarNatan:2004}. For MM16 and MM17, the results are exactly \Cref{prop:cc-rm2,prop:cc-rm3}. For MM18, it is obvious that the two maps are identical on the chain level. Note that movie moves involving far-commutativity give identical maps, since the maps for the elementary moves are all defined locally.
\end{proof}

\begin{proof}[Proof of \Cref{formula_of_moves}]
    Immediate from the definitions of the immersed cobordism maps, combined with \Cref{prop:isotopy-inv-up-to-sign,prop:twist-move-map,prop:finger-move-map}.
\end{proof}

Now we prove the second part of \Cref{thm:isotopy-invariance}, where we restrict to oriented cobordisms, and focus on the map $\phi^{\bal}_S$. Note that in this case, $\phi^{\bal}_S$ preserves the homological grading. Recall from \Cref{rem:sign-ambiguity} that there is a preferred choice of a direction for each negative-to-positive crossing change. There is no need for a direction for a positive-to-negative crossing change (since we use $f_1$), so $\phi^{\bal}_S$ can be defined solely from $S$ and $\mathbf{z}$. With this definition of $\phi^{\bal}_S$, we have the following. 

\begin{prop}
    For any oriented normally immersed cobordism $S$, the sign of the immersed cobordism map $\phi^\bal(S)$ can be adjusted, so that it is invariant (up to chain homotopy) under smooth isotopies of $S$ rel $\partial$, and is functorial with respect to composition of cobordisms.
\end{prop}

\begin{proof}
    The idea is similar to the proof of \cite[Theorem 1]{Sano:2020-b}, where the second order proved that the sign ambiguity of the embedded and oriented cobordism map on Khovanov homology can be fixed by focusing on the \textit{Lee class}. Here, we only sketch the proof. We assume that the cobordism $S$ is decomposed into elementary pieces, and the intermediate links are oriented coherently with $S$. 
    
    First, let us consider the (localized) $U(1)$-equivariant Khovanov homology, given by the base ring $R = \Z[H^\pm]$ and the Frobenius algebra $A = R[X]/(X^2 - HX)$. The \textit{Lee cycle} $\alpha(D)$ of a link diagram $D$ is a specific cycle in the complex $\CKh_H(D)$, obtained from the oriented resolution $D_o$ of $D$ with the circles labeled by $X$ or $Y = X - H$, according to the orientation and the checkerboard coloring of $D_o$ (see \cite[Definition 2.9]{Sano:2020-b}). Its homology class, the \textit{Lee class} $[\alpha(D)]$, is always nonzero and has homological grading $0$. As stated in \cite[Proposition 3.4]{Sano:2020-b}, its behavior under the Reidemeister moves and the Morse moves can be described explicitly, so we may adjust the signs of the corresponding maps so that the Lee classes of the two diagrams correspond without signs. We claim that the same thing can be done for the crossing change maps.     
    
    For the positive-to-negative crossing change map $f_1^-$, we have 
    \[
        f_1^-(\alpha(D)) = \alpha(D')
    \]
    so there is no need to fix the sign. For the negative-to-positive crossing change map $f_0^+$, we have 
    \[
        f_0^+(\alpha(D)) = \pm H \alpha(D')
    \]
    which follows from 
    \[
        X^2 = HX,\quad XY = 0,\quad Y^2 = -HY.
    \]
    Thus, depending on $D$, we adjust the sign of the crossing change map so that the Lee classes correspond positively. With the signs adjusted for all of the elementary cobordism maps, their composition, redefined as $\phi^\bal(S)$, sends
    \[
        [\alpha(D)] \mapsto H^j [\alpha(D')] + (\cdots)
    \]
    for some exponent $j \in \Z$. 
    Now, if $S'$ is another cobordism isotopic to $S$, and we redefine $\phi^\bal(S')$ by the same procedure, we know from \Cref{prop:isotopy-inv-up-to-sign} that the induced maps correspond as $\Kh^\bal(S) = \pm \Kh^\bal(S')$. However, having adjusted the signs of both maps as above, the images of $[\alpha(D)]$ have positive coefficients for $ [\alpha(D')]$, which forces $\Kh^\bal(S) = \Kh^\bal(S')$. 

    The above sign adjustment can be performed uniformly for any general Frobenius extension $(R, A)$, even if the Lee class cannot be defined in the corresponding homology group, so the claim holds in general.  
\end{proof}

\begin{rem}
    In \cite[Section 6.11]{Ren-Willis:2024}, Ren and Willis define a cobordism map for Khovanov homology $\Kh = \Kh_{0, 0}$ induced from an embedded link cobordism $S$ in the twice-punctured $\C P^2$. Namely, let $X = \C P^2 \setminus (D^4 \sqcup D^4)$ and consider a framed oriented surface $S \subset X$ with boundary links $L$ and $L'$. Then a cobordism map
    \[
        \Kh^{RW}(S) \colon \Kh(L) \to \Kh(L')
    \]
    of bidegree $(0, \chi(S) - S \cdot S + |[S]|)$ is defined in two ways: one by using the $\mathfrak{gl}_2$ Khovanov--Rozansky skein lasagna modules (over a field $\bbF$), and another by a direct construction (over $\Z$). Here, we relate the latter description with our immersed cobordism map. (A similar argument can be found in the proof of \cite[Theorem 6.10]{MMSW:2023}.)

    Suppose $L, L'$ are two links in $S^3$ related by a full left twist along two strands. The transformation from $L$ to $L'$ can be realized as a $(+1)$-surgery along an unknot bounding a disk that intersects the two strands of $L$ in two points. Suppose that the intersection has $p$ positive points and $q$ negative points, where $p + q = 2$. The surgery can be regarded as a link cobordism $S$ in $X = (I \times S^3) \# \C P^2$ between $L$ and $L'$. Let $N$ be of the tubular neighborhood of the core $\C P^1 \subset \C P^2$. The intersection of $S$ and the boundary $\partial N \approx S^3$ is the Hopf link $H_{p, q}$, which is negatively oriented if $(p, q) = (2, 0)$ or $(0, 2)$, and positively if $(p, q) = (1, 1)$. By removing the interior of $N$ and tubing $\partial N$ with the input boundary of $X$ gives an embedded cobordism $S^\circ$ in $I \times S^3$ from $L \sqcup H_{p, q}$ to $L'$, which induces 
    \[
        \Kh(S^\circ)\colon \Kh(L) \otimes \Kh(H_{p, q}) \to \Kh(L'). 
    \]
    Now, choose a generator $z$ of 
    \[
        \Kh^{0, -(p-q)^2 + 2 \max(p, q)}(H_{p, q}) \cong \Z
    \]
    either by
    \[
        \zeta_0 \in \Kh^{0, 2}(H^+) \text{\ \ or \ }
        \zeta_1 \in \Kh^{0, 0}(H^-).
    \]
    Then Ren--Willis' map for $S$ is given by
    \[
        \Kh^{RW}(S) := \Kh(S^\circ)( - \otimes z),
    \]
    which is exactly how we define the immersed cobordism map $\Kh^{\bal}(S)$, when $S$ is realized as a crossing change between two link diagrams. 
\end{rem}
\section{Khovanov homology of the negative \texorpdfstring{$(2, q)$}{(2, q)}-torus knots}
\label{sec:Kh-structure}

Here we study the structure of $\Kh(T_{2, q})$ for the $(2, q)$-torus knot $T_{2, q}$ with odd $q$. The statements will be proved for $q$-twist tangles $T_q$, and for that purpose we consider the category $\Cob^3_{\bullet / l}(B)$ of \textit{dotted cobordisms} with boundary points $B \subset \partial D^2$, introduced by Bar-Natan in \cite{BarNatan:2004}. Here, we consider the local relations corresponding to the $U(2)$-equivariant theory. For any tangle diagram $T$, let $[T]$ denote the \textit{formal Khovanov complex} of $T$, which is an object of $\operatorname{Kob}_{\bullet / l} (\partial T) = \operatorname{Kom}(\operatorname{Mat}(\Cob_{\bullet / l} (\partial T)))$, i.e.\ a complex in the additive closure of the preadditive category $\Cob_{\bullet / l} (\partial T)$. The crossing change maps $f_0, f_1$ defined in \Cref{sec:kh-x-ch} can be extended to chain maps between tangle complexes using their cobordism forms given in \Cref{sec:kh-x-ch}. 

Throughout this section, most technical details of the diagrammatic calculations are omitted. 
The basic methods we use here are \textit{delooping} \cite[Lemma 4.1]{Bar-Natan:2007} and \textit{Gaussian elimination} \cite[Lemma 4.2]{Bar-Natan:2007}. See \cite[Section 2]{Sano:2025} for a comprehensive exposition.

\subsection{Structure of \texorpdfstring{$\Kh(T^*_{2, 2k + 1})$}{Kh(mT(2, 2k + 1))}}

First, we collect some of the results obtained in \cite[Section 4.1]{Sano:2025}. For any $q \in \Z \setminus \{0\}$, let $T_q$ denote the unoriented tangle diagram obtained from a pair of crossingless vertical strands by adding $|q|$ half-twists, twisted positively or negatively depending on the sign of $q$. 
\begin{center}
    \tikzset{every picture/.style={line width=0.75pt}} 

\begin{tikzpicture}[x=0.75pt,y=0.75pt,yscale=-1,xscale=1]

\begin{knot}[
  clip width=3pt,
  end tolerance=1pt,
]

\strand [line width=1.5]    (91.05,10.83) .. controls (90.61,25.83) and (71.06,24.83) .. (71.33,40.16) .. controls (71.61,55.49) and (91.33,56.16) .. (91.33,70.83) .. controls (91.33,85.49) and (71.33,84.83) .. (70.67,100.16) ;
\strand [line width=1.5]    (72.24,10.16) .. controls (72.64,25.16) and (91.59,24.83) .. (91.33,40.16) .. controls (91.08,55.49) and (71.98,55.49) .. (71.98,70.16) .. controls (71.98,84.83) and (90.05,85.49) .. (90.67,100.83) ;


\end{knot}

\draw  [color={rgb, 255:red, 255; green, 255; blue, 255 }  ,draw opacity=1 ][fill={rgb, 255:red, 255; green, 255; blue, 255 }  ,fill opacity=1 ] (67,37.17) -- (95,37.17) -- (95,70.83) -- (67,70.83) -- cycle ;
\draw  [color={rgb, 255:red, 128; green, 128; blue, 128 }  ,draw opacity=1 ] (104,93) .. controls (108.67,93) and (111,90.67) .. (111,86) -- (111,65.73) .. controls (111,59.06) and (113.33,55.73) .. (118,55.73) .. controls (113.33,55.73) and (111,52.4) .. (111,45.73)(111,48.73) -- (111,25.45) .. controls (111,20.78) and (108.67,18.45) .. (104,18.45) ;

\draw (9,44.4) node [anchor=north west][inner sep=0.75pt]    {$T_{q} \ =$};
\draw (72.67,44.57) node [anchor=north west][inner sep=0.75pt]    {$\vdots $};
\draw (125,44.4) node [anchor=north west][inner sep=0.75pt]    {$q$};

\end{tikzpicture}
\end{center}
Let $\mathsf{E}_0$ and $\mathsf{E}_1$ denote the following unoriented $4$-end crossingless tangle diagrams, 
\begin{center}
    \tikzset{every picture/.style={line width=0.75pt}} 

\begin{tikzpicture}[x=0.75pt,y=0.75pt,yscale=-.75,xscale=.75]

\draw  [dash pattern={on 4.5pt off 4.5pt}] (206,34.77) .. controls (206,20.71) and (217.4,9.31) .. (231.46,9.31) .. controls (245.52,9.31) and (256.91,20.71) .. (256.91,34.77) .. controls (256.91,48.83) and (245.52,60.22) .. (231.46,60.22) .. controls (217.4,60.22) and (206,48.83) .. (206,34.77) -- cycle ;
\draw  [draw opacity=0][line width=1.5]  (251.56,20.25) .. controls (246.19,23.83) and (239.17,26) .. (231.48,26) .. controls (223.83,26) and (216.84,23.85) .. (211.48,20.3) -- (231.48,2.73) -- cycle ; \draw  [line width=1.5]  (251.56,20.25) .. controls (246.19,23.83) and (239.17,26) .. (231.48,26) .. controls (223.83,26) and (216.84,23.85) .. (211.48,20.3) ;  
\draw  [draw opacity=0][line width=1.5]  (211.4,49.29) .. controls (216.77,45.71) and (223.79,43.54) .. (231.48,43.54) .. controls (239.13,43.54) and (246.12,45.69) .. (251.48,49.24) -- (231.48,66.81) -- cycle ; \draw  [line width=1.5]  (211.4,49.29) .. controls (216.77,45.71) and (223.79,43.54) .. (231.48,43.54) .. controls (239.13,43.54) and (246.12,45.69) .. (251.48,49.24) ;  

\draw  [dash pattern={on 4.5pt off 4.5pt}] (86.48,9.29) .. controls (100.54,9.29) and (111.94,20.69) .. (111.94,34.75) .. controls (111.94,48.81) and (100.54,60.21) .. (86.48,60.21) .. controls (72.42,60.21) and (61.02,48.81) .. (61.02,34.75) .. controls (61.02,20.69) and (72.42,9.29) .. (86.48,9.29) -- cycle ;
\draw  [draw opacity=0][line width=1.5]  (101,54.85) .. controls (97.42,49.48) and (95.25,42.46) .. (95.25,34.77) .. controls (95.25,27.12) and (97.4,20.13) .. (100.95,14.77) -- (118.52,34.77) -- cycle ; \draw  [line width=1.5]  (101,54.85) .. controls (97.42,49.48) and (95.25,42.46) .. (95.25,34.77) .. controls (95.25,27.12) and (97.4,20.13) .. (100.95,14.77) ;  
\draw  [draw opacity=0][line width=1.5]  (71.96,14.69) .. controls (75.54,20.06) and (77.71,27.08) .. (77.71,34.77) .. controls (77.71,42.42) and (75.56,49.41) .. (72.01,54.77) -- (54.43,34.77) -- cycle ; \draw  [line width=1.5]  (71.96,14.69) .. controls (75.54,20.06) and (77.71,27.08) .. (77.71,34.77) .. controls (77.71,42.42) and (75.56,49.41) .. (72.01,54.77) ;

\draw (7,22.4) node [anchor=north west][inner sep=0.75pt]    {$\mathsf{E}_{0} \ =\ $};
\draw (149,23.4) node [anchor=north west][inner sep=0.75pt]    {$\mathsf{E}_{1} \ =\ $};
\draw (121,30) node [anchor=north west][inner sep=0.75pt]   [align=left] {,};

\end{tikzpicture}
\end{center}
Let $e$ denote the saddle morphism from $\mathsf{E}_0$ to $\mathsf{E}_1$ and also for the other way round.
\[
\begin{tikzcd}
    \mathsf{E}_0 
        \arrow[r, "e", shift left] & 
    \mathsf{E}_1 
        \arrow[l, "e", shift left]
\end{tikzcd}    
\]
Let $\Phi$ denote the endomorphism on $\mathsf{E}_1$ defined in \Cref{subsec:x-ch-map}. We define another endomorphism\footnote{
    Morphisms $\Psi, \Phi$ in this paper are denoted $a, b$ respectively in \cite{Sano:2025}, following the notations of \cite{Thompson:2018}. 
} 
$\Psi$ on $\mathsf{E}_1$ by 
\begin{center}
    \tikzset{every picture/.style={line width=0.75pt}} 

\begin{tikzpicture}[x=0.75pt,y=0.75pt,yscale=-.8,xscale=.8]

\draw  [dash pattern={on 4.5pt off 4.5pt}] (69,34.43) .. controls (69,20.37) and (80.4,8.97) .. (94.46,8.97) .. controls (108.52,8.97) and (119.91,20.37) .. (119.91,34.43) .. controls (119.91,48.49) and (108.52,59.89) .. (94.46,59.89) .. controls (80.4,59.89) and (69,48.49) .. (69,34.43) -- cycle ;
\draw  [draw opacity=0][line width=1.5]  (114.56,19.91) .. controls (109.19,23.49) and (102.17,25.66) .. (94.48,25.66) .. controls (86.83,25.66) and (79.84,23.51) .. (74.48,19.96) -- (94.48,2.39) -- cycle ; \draw  [line width=1.5]  (114.56,19.91) .. controls (109.19,23.49) and (102.17,25.66) .. (94.48,25.66) .. controls (86.83,25.66) and (79.84,23.51) .. (74.48,19.96) ;  
\draw  [draw opacity=0][line width=1.5]  (74.4,48.96) .. controls (79.77,45.37) and (86.79,43.2) .. (94.48,43.2) .. controls (102.13,43.2) and (109.12,45.35) .. (114.48,48.9) -- (94.48,66.48) -- cycle ; \draw  [line width=1.5]  (74.4,48.96) .. controls (79.77,45.37) and (86.79,43.2) .. (94.48,43.2) .. controls (102.13,43.2) and (109.12,45.35) .. (114.48,48.9) ;  

\draw  [fill={rgb, 255:red, 0; green, 0; blue, 0 }  ,fill opacity=1 ] (91,25.43) .. controls (91,23.66) and (92.43,22.23) .. (94.2,22.23) .. controls (95.97,22.23) and (97.41,23.66) .. (97.41,25.43) .. controls (97.41,27.2) and (95.97,28.64) .. (94.2,28.64) .. controls (92.43,28.64) and (91,27.2) .. (91,25.43) -- cycle ;

\draw  [dash pattern={on 4.5pt off 4.5pt}] (170,34.43) .. controls (170,20.37) and (181.4,8.97) .. (195.46,8.97) .. controls (209.52,8.97) and (220.91,20.37) .. (220.91,34.43) .. controls (220.91,48.49) and (209.52,59.89) .. (195.46,59.89) .. controls (181.4,59.89) and (170,48.49) .. (170,34.43) -- cycle ;
\draw  [draw opacity=0][line width=1.5]  (215.56,19.91) .. controls (210.19,23.49) and (203.17,25.66) .. (195.48,25.66) .. controls (187.83,25.66) and (180.84,23.51) .. (175.48,19.96) -- (195.48,2.39) -- cycle ; \draw  [line width=1.5]  (215.56,19.91) .. controls (210.19,23.49) and (203.17,25.66) .. (195.48,25.66) .. controls (187.83,25.66) and (180.84,23.51) .. (175.48,19.96) ;  
\draw  [draw opacity=0][line width=1.5]  (175.4,48.96) .. controls (180.77,45.37) and (187.79,43.2) .. (195.48,43.2) .. controls (203.13,43.2) and (210.12,45.35) .. (215.48,48.9) -- (195.48,66.48) -- cycle ; \draw  [line width=1.5]  (175.4,48.96) .. controls (180.77,45.37) and (187.79,43.2) .. (195.48,43.2) .. controls (203.13,43.2) and (210.12,45.35) .. (215.48,48.9) ;  

\draw  [fill={rgb, 255:red, 0; green, 0; blue, 0 }  ,fill opacity=1 ] (192.25,43.63) .. controls (192.25,41.86) and (193.69,40.43) .. (195.46,40.43) .. controls (197.23,40.43) and (198.66,41.86) .. (198.66,43.63) .. controls (198.66,45.4) and (197.23,46.84) .. (195.46,46.84) .. controls (193.69,46.84) and (192.25,45.4) .. (192.25,43.63) -- cycle ;

\draw  [dash pattern={on 4.5pt off 4.5pt}] (277,34.43) .. controls (277,20.37) and (288.4,8.97) .. (302.46,8.97) .. controls (316.52,8.97) and (327.91,20.37) .. (327.91,34.43) .. controls (327.91,48.49) and (316.52,59.89) .. (302.46,59.89) .. controls (288.4,59.89) and (277,48.49) .. (277,34.43) -- cycle ;
\draw  [draw opacity=0][line width=1.5]  (322.56,19.91) .. controls (317.19,23.49) and (310.17,25.66) .. (302.48,25.66) .. controls (294.83,25.66) and (287.84,23.51) .. (282.48,19.96) -- (302.48,2.39) -- cycle ; \draw  [line width=1.5]  (322.56,19.91) .. controls (317.19,23.49) and (310.17,25.66) .. (302.48,25.66) .. controls (294.83,25.66) and (287.84,23.51) .. (282.48,19.96) ;  
\draw  [draw opacity=0][line width=1.5]  (282.4,48.96) .. controls (287.77,45.37) and (294.79,43.2) .. (302.48,43.2) .. controls (310.13,43.2) and (317.12,45.35) .. (322.48,48.9) -- (302.48,66.48) -- cycle ; \draw  [line width=1.5]  (282.4,48.96) .. controls (287.77,45.37) and (294.79,43.2) .. (302.48,43.2) .. controls (310.13,43.2) and (317.12,45.35) .. (322.48,48.9) ;

\draw (121.91,32.27) node [anchor=west] [inner sep=0.75pt]    {$\ \ \ +\ \ \ $};
\draw (222.91,33.27) node [anchor=west] [inner sep=0.75pt]    {$\ \ \ -\ h$};

\end{tikzpicture}
\end{center}
which can also be described as a cobordism:
\begin{center}
    \tikzset{every picture/.style={line width=0.75pt}} 

\begin{tikzpicture}[x=0.75pt,y=0.75pt,yscale=-.75,xscale=.75]

\draw [color={rgb, 255:red, 0; green, 0; blue, 0 }  ,draw opacity=1 ][line width=1.5]    (149.96,34.56) .. controls (150,33.25) and (154,36.75) .. (159.14,28.74) ;
\draw [color={rgb, 255:red, 0; green, 0; blue, 0 }  ,draw opacity=1 ][line width=1.5]    (141.5,64.56) .. controls (145.5,64.75) and (148.23,63.24) .. (154.32,78.52) ;
\draw  [dash pattern={on 4.5pt off 4.5pt}] (150.5,13.33) .. controls (158.05,17.59) and (162.83,36.78) .. (161.18,56.2) .. controls (159.52,75.62) and (152.06,87.92) .. (144.5,83.66) .. controls (136.95,79.41) and (132.17,60.21) .. (133.82,40.79) .. controls (135.48,21.37) and (142.94,9.07) .. (150.5,13.33) -- cycle ;
\draw [color={rgb, 255:red, 0; green, 0; blue, 0 }  ,draw opacity=1 ][line width=1.5]    (64.96,17.56) .. controls (70.84,35.28) and (75.32,38.82) .. (84.14,28.74) ;
\draw [color={rgb, 255:red, 0; green, 0; blue, 0 }  ,draw opacity=1 ][line width=1.5]    (60.14,67.34) .. controls (68.59,61.65) and (73.23,63.24) .. (79.32,78.52) ;
\draw  [dash pattern={on 4.5pt off 4.5pt}] (75.5,13.33) .. controls (83.05,17.59) and (87.83,36.78) .. (86.18,56.2) .. controls (84.52,75.62) and (77.06,87.92) .. (69.5,83.66) .. controls (61.95,79.41) and (57.17,60.21) .. (58.82,40.79) .. controls (60.48,21.37) and (67.94,9.07) .. (75.5,13.33) -- cycle ;
\draw [line width=0.75]    (159.14,28.74) -- (84.14,28.74) ;
\draw [line width=0.75]    (70.5,67.34) -- (60.14,67.34) ;
\draw [line width=0.75]    (141.5,64.56) -- (69.96,64.56) ;
\draw [line width=0.75]    (154.32,78.52) -- (79.32,78.52) ;
\draw [line width=0.75]    (149.96,34.56) -- (74.96,34.56) ;
\draw [line width=0.75]    (139.96,17.56) -- (64.96,17.56) ;
\draw [line width=0.75]    (106.46,34.75) -- (106.46,64.31) ;
\draw [line width=0.75]    (114.96,34.75) -- (114.96,64.31) ;
\draw [color={rgb, 255:red, 0; green, 0; blue, 0 }  ,draw opacity=1 ][line width=1.5]    (139.96,17.56) .. controls (141.15,21.14) and (142.28,24.14) .. (143.4,26.57) ;


\end{tikzpicture}
\end{center}
Note that both $\Psi$ and $\Phi$ have quantum degree $-2$, and the compositions $\Psi \Phi$, $\Phi \Psi$, $e \Phi$, $\Phi e$ are all $0$. The following lemma is the key to prove \Cref{prop:twist-tangle-retract} and other propositions in \Cref{subsec:relating-t2q}.

\begin{lem}
\label{lem:Tq-lem1}
    Consider the following diagrams $T, T'$ and morphisms $m, \Delta, \Phi, \Psi$:
    \begin{center}
        \tikzset{every picture/.style={line width=0.75pt}} 

\begin{tikzpicture}[x=0.75pt,y=0.75pt,yscale=-.9,xscale=.9]

\draw  [color={rgb, 255:red, 0; green, 0; blue, 0 }  ,draw opacity=1 ][line width=1.5]  (177.91,28.77) .. controls (177.91,23.46) and (182.41,19.16) .. (187.96,19.16) .. controls (193.5,19.16) and (198,23.46) .. (198,28.77) .. controls (198,34.07) and (193.5,38.37) .. (187.96,38.37) .. controls (182.41,38.37) and (177.91,34.07) .. (177.91,28.77) -- cycle ;
\draw  [dash pattern={on 4.5pt off 4.5pt}] (162.5,33.77) .. controls (162.5,19.71) and (173.9,8.31) .. (187.96,8.31) .. controls (202.02,8.31) and (213.41,19.71) .. (213.41,33.77) .. controls (213.41,47.83) and (202.02,59.22) .. (187.96,59.22) .. controls (173.9,59.22) and (162.5,47.83) .. (162.5,33.77) -- cycle ;
\draw  [draw opacity=0][line width=1.5]  (168.89,50.71) .. controls (174.12,48.05) and (180.76,46.45) .. (187.98,46.45) .. controls (194.91,46.45) and (201.3,47.92) .. (206.43,50.39) -- (187.98,65.81) -- cycle ; \draw  [line width=1.5]  (168.89,50.71) .. controls (174.12,48.05) and (180.76,46.45) .. (187.98,46.45) .. controls (194.91,46.45) and (201.3,47.92) .. (206.43,50.39) ;  
\draw  [dash pattern={on 4.5pt off 4.5pt}] (37,33.77) .. controls (37,19.71) and (48.4,8.31) .. (62.46,8.31) .. controls (76.52,8.31) and (87.91,19.71) .. (87.91,33.77) .. controls (87.91,47.83) and (76.52,59.22) .. (62.46,59.22) .. controls (48.4,59.22) and (37,47.83) .. (37,33.77) -- cycle ;
\draw  [draw opacity=0][line width=1.5]  (42.4,48.29) .. controls (47.77,44.71) and (54.79,42.54) .. (62.48,42.54) .. controls (70.13,42.54) and (77.12,44.69) .. (82.48,48.24) -- (62.48,65.81) -- cycle ; \draw  [line width=1.5]  (42.4,48.29) .. controls (47.77,44.71) and (54.79,42.54) .. (62.48,42.54) .. controls (70.13,42.54) and (77.12,44.69) .. (82.48,48.24) ;  
\draw    (100,33) -- (147,33) ;
\draw [shift={(149,33)}, rotate = 180] [color={rgb, 255:red, 0; green, 0; blue, 0 }  ][line width=0.75]    (10.93,-4.9) .. controls (6.95,-2.3) and (3.31,-0.67) .. (0,0) .. controls (3.31,0.67) and (6.95,2.3) .. (10.93,4.9)   ;
\draw    (151,44) -- (100,44) ;
\draw [shift={(98,44)}, rotate = 360] [color={rgb, 255:red, 0; green, 0; blue, 0 }  ][line width=0.75]    (10.93,-4.9) .. controls (6.95,-2.3) and (3.31,-0.67) .. (0,0) .. controls (3.31,0.67) and (6.95,2.3) .. (10.93,4.9)   ;
\draw    (223.77,25.32) .. controls (232.08,17.65) and (256.35,7.43) .. (256.99,33.62) .. controls (257.6,58.63) and (237.61,49.89) .. (225.44,44.01) ;
\draw [shift={(223.77,43.2)}, rotate = 25.94] [color={rgb, 255:red, 0; green, 0; blue, 0 }  ][line width=0.75]    (6.56,-1.97) .. controls (4.17,-0.84) and (1.99,-0.18) .. (0,0) .. controls (1.99,0.18) and (4.17,0.84) .. (6.56,1.97)   ;

\draw (179,69.65) node [anchor=north west][inner sep=0.75pt]    {$T'$};
\draw (53,69.65) node [anchor=north west][inner sep=0.75pt]    {$T$};
\draw (117,11.4) node [anchor=north west][inner sep=0.75pt]    {$\Delta$};
\draw (117,47.4) node [anchor=north west][inner sep=0.75pt]    {$m$};
\draw (263,23.4) node [anchor=north west][inner sep=0.75pt]    {$\Psi,\ \Phi$};

\end{tikzpicture}
    \end{center}
    By delooping the circle appearing in tangle $T'$, the morphisms $m, \Delta$ can be described as
    \[
\begin{tikzcd}[row sep=.3em]
 & & q^2 T & & T \arrow[dd, "\oplus", phantom] \arrow[rrd, "I"] & &     \\
T \arrow[rru, "Y"] \arrow[rrd, "I"'] & & \oplus & & & & T, \\ 
& & T & & q^{-2}T \arrow[rru, "X"'] & &  
\end{tikzcd}        
    \]
    Similarly, the endomorphism $\Psi, \Phi$ can be described as
    \[
\begin{tikzcd}[row sep=2em]
T \arrow[d, "\oplus" description, no head, phantom] \arrow[rrr, "Y"] \arrow[rrrd, "I" description, pos=.8] & & & q^2 T \arrow[d, "\oplus" description, no head, phantom] & T \arrow[d, "\oplus" description, no head, phantom] \arrow[rrrd, "I" description, pos=.8] \arrow[rrr, "-X"] & & & q^2 T \arrow[d, "\oplus" description, no head, phantom] \\
q^{-2} T \arrow[rrr, "X"] \arrow[rrru, "t" description, pos=.75] & & & T & q^{-2}T \arrow[rrr, "-Y"] \arrow[rrru, "t" description, pos=.75] & & & T.
\end{tikzcd}
    \]
    Here, $Y$ denotes the cobordism corresponding to the multiplication of $Y = X - h$. 
\end{lem}

\begin{prop}
\label{prop:twist-tangle-retract}
    For each $q \geq 1$, there is a strong deformation retract from the complex $[T_q]$ to a complex $\mathsf{E}_q$ of length $q + 1$ defined as 
    \[
\begin{tikzcd}
\underline{\mathsf{E}_0} \arrow[r, "e"] & q^{1}\mathsf{E}_1\{1\} \arrow[r, "\Phi"] & q^3\mathsf{E}_1 \arrow[r, "\Psi"] & q^5\mathsf{E}_1 \arrow[r, "\Phi"] & \cdots \arrow[r, ""] & q^{2q - 1}\mathsf{E}_1.
\end{tikzcd}
    \]
    Similarly, there is a strong deformation retract from the complex $[T_{-q}]$ to a complex $\mathsf{E}_{-q}$ of length $q + 1$ defined as 
    \[
\begin{tikzcd}
q^{-2q + 1}\mathsf{E}_1 \arrow[r, ""] & \cdots \arrow[r, "\Phi"] & q^{-5}\mathsf{E}_1 \arrow[r, "\Psi"] & q^{-3}\mathsf{E}_1 \arrow[r, "\Phi"] & q^{-1}\mathsf{E}_1 \arrow[r, "e"] & \underline{\mathsf{E}_0}.
\end{tikzcd}
    \]
    Here, the bigradings are relative with respect to the underlined object $\underline{\mathsf{E}_0}$, and $q^a$ denotes the quantum grading shift by $a$.
\end{prop}

\Cref{prop:twist-tangle-retract} will be reproved partially in \Cref{subsec:relating-t2q}. It immediately implies the following corollary, which have been proved in \cite[Section 6.2]{Khovanov:2000}, \cite[Proposition 4.1]{Thompson:2018} and \cite[Proposition 5.1]{Schuetz:2021} under various specializations. 

\begin{cor}
\label{cor:CKh-T_2q}
    Let $U$ denote the map
    \[
        U = 2X - h: A \rightarrow q^2 A.
    \]
    For any $k > 0$, the negative $(2, 2k + 1)$-torus knot $T^*_{2, 2k + 1}$ has
    \[
        \CKh(T^*_{2, 2k + 1}) 
            \ \simeq \ 
        \bigoplus_{i = 1}^k t^{-2i-1}q^{-2k - 4i - 1} \left(A \xrightarrow{U} t q^2 A\right)
            \ \oplus\ 
        q^{-2k + 1}\underline{A}.
    \]
    Here, the underlined $\underline{A}$ indicates the homological grading $0$ part, and $t^a q^b$ denote the homological and quantum grading shift by $(a, b)$.
\end{cor}

\begin{proof}
    The negative torus knot $T^*_{2, 2k + 1}$ is obtained by closing the four ends of $T^*_{2k + 1} = T_{-2k - 1}$ vertically. This turns $\mathsf{E}_0$, $\mathsf{E}_1$ into $\bigcirc\bigcirc$, $\bigcirc$, and the morphisms $e, \Psi, \Phi$ into $\Delta, U, 0$ respectively. Thus the sequence of \Cref{prop:twist-tangle-retract} splits into segments of length 2,
    \[
\begin{tikzcd}
q^{-4k - 1}\bigcirc \arrow[r, "U"] & q^{-4k} \bigcirc \arrow[r, dotted, no head] & \cdots \arrow[r, dotted, no head] & q^{-5}\bigcirc \arrow[r, "U"] & q^{-3}\bigcirc \arrow[r, dotted, no head] & q^{-1}\bigcirc \arrow[r, "\Delta"] & \underline{\bigcirc\bigcirc}.
\end{tikzcd}
    \]
    For the rightmost segment, from \Cref{lem:Tq-lem1}, we have 
    \[
        \left(q^{-1} \bigcirc \xrightarrow{\ \Delta\ } \underline{\bigcirc\bigcirc} \right)
            \ \simeq \ 
        \left(0 \xrightarrow{\ 0\ } q^1 \underline{\bigcirc} \right).
    \]
    The absolute bigrading shift for the rightmost object is given by 
    \[
        (0,\ 2n^+ - n^-) = (0,\ -2k - 1)
    \]
    so the rightmost $q^1 \underline{\bigcirc}$ should be $q^{-2k + 1} \underline{A}$ after applying the TQFT. The description for the remaining part is obvious. 
\end{proof}

The map $U$ is represented by the matrix
\[
    \begin{pmatrix}
        -h & 2t \\ 2 & h
    \end{pmatrix}
\]
with respect to the basis $\{1, X\}$ for $A$. Its  determinant is $-h^2 - 4t$, equal to the negative of the discriminant $\mathcal{D}$ of $X^2 - hX - t$. In particular, when $\mathcal{D}$ is a unit in $R$, we see that 
\[
    \CKh(T^*_{2, 2k + 1}) 
        \ \simeq \ 
    q^{-2k + 1}\underline{A}
\]
For other typical cases, the homology groups can be easily computed as follows.
\begin{align*}
    \Kh_{0, 0}(T^*_{2, 2k + 1}; \Z) \ &\cong \ 
    \bigoplus_{i=1}^k (\Z \oplus \Z \oplus \Z/2) \oplus \Z^2,\\
    \Kh_{0, t}(T^*_{2, 2k + 1}; \Q[t]) \ &\cong \ 
    \bigoplus_{i=1}^k (\Q[t]/(t)) 
    \oplus \Q[t]^2,\\
    \Kh_{h, 0}(T^*_{2, 2k + 1}; \bbF_2[h]) \ &\cong \ 
    \bigoplus_{i=1}^k (\bbF_2[h]/(h))^2 
    \oplus \bbF_2[h]^2.
\end{align*}
Here, $\Kh_{h, t}(-; R)$ denotes the Khovanov homology obtained from the Frobenius algebra $A = R[X]/(X^2 - hX - t)$ over $R$. The first one is the original Khovanov homology \cite{Khovanov:2000}, the second is the (bigraded) Lee homology over $\Q$ \cite{Lee:2005}, and the third is the (bigraded) Bar-Natan homology over $\bbF_2$ \cite{BarNatan:2004}. 

\subsection{Relating \texorpdfstring{$\Kh(T^*_{2, q})$}{Kh(mT(2, q))} and \texorpdfstring{$\Kh(T^*_{2, q+2})$}{Kh(mT(2, q+2))}}
\label{subsec:relating-t2q}

Let $q \geq 1$ be an positive odd integer, and let $E_{-q}$ denote the complex 
\[
\begin{tikzcd}
E_{-q} \arrow[r, "=", phantom] & \{\ q^{-2q + 1}\mathsf{E}_1 \arrow[r, ""] & \cdots \arrow[r, "\Phi"] & q^{-5}\mathsf{E}_1 \arrow[r, "\Psi"] & q^{-3}\mathsf{E}_1 \arrow[r, "\Phi"] & q^{-1}\mathsf{E}_1 \arrow[r, "e"] & \underline{\mathsf{E}_0} \ \}
\end{tikzcd}
\]
obtained as a strong deformation retract of $[T_{-q}]$ in \Cref{prop:twist-tangle-retract}. For each odd $q \geq 1$, we inductively construct a strong deformation retraction 
\[
    r_{-q}\colon 
    [T_{-q}] \xrightarrow{\ \simeq \ } E_{-q}.
\]
A similar argument can be found in \cite[Section 4]{Sano:2025}, so here we only sketch the construction. 

First, take $r_{-1} = \id$. Next, $r_{-2}\colon [T_{-2}] \to E_{-2}$ is defined as follows. The tangle $T_{-2}$ can be written as $D(T_{-1}, T_{-1})$, where $D$ is a 2-input planar arc diagram that vertically connects the two tangles. Then the complex $[T_{-2}]$ can be described as a square
\[
\begin{tikzcd}
\mathsf{E}_{11} \arrow[r, "m_2"] \arrow[d, "m_1"] & 
\mathsf{E}_1 \arrow[d, "e"] \\
\mathsf{E}_1 \arrow[r, "-e"] & \underline{\mathsf{E}_0}.
\end{tikzcd}    
\]
Here, diagrams $D(\mathsf{E}_0, \mathsf{E}_i)$ and $D(\mathsf{E}_i, \mathsf{E}_0)$ are identified with $\mathsf{E}_i$ ($i = 0, 1$). The diagram $D(E_1, E_1)$ is denoted $\mathsf{E}_{11}$, which has the form of a circle inserted between the two arcs of $\mathsf{E}_1$. The underlined diagram $\underline{\mathsf{E}_0}$ indicates the one with highest homological grading. Subscripts on the labels indicate which of the two input holes of $D$ are being used, i.e.\ $m_1 = D(m, I)$ and $m_2 = D(I, m)$. Delooping isomorphism gives $\mathsf{E}_{11} \cong \mathsf{E}_1 \oplus \mathsf{E}_1$, and with \Cref{lem:Tq-lem1}, one can see that the square can be collapsed as 
\[
\begin{tikzcd}
\mathsf{E}_1 \arrow[d, "\Phi"] \\
\mathsf{E}_1 \arrow[r, "e"] & \underline{\mathsf{E}_0}.
\end{tikzcd}    
\]
This gives the retraction $r_{-2}$. 

For odd $q > 1$, suppose we have obtained the retraction $r_{-q}$. The tangle $T_{-q-2}$ can be decomposed as $D(T_{-2}, T_{-q})$, and the complex $[T_{-q-2}] = D([T_{-2}], [T_{-q}])$ retracts to $D(E_{-2}, E_{-q})$ by the retraction $D(r_{-2}, r_{-q})$. The complex $D(E_{-2}, E_{-q})$ is described as
\[
\begin{tikzcd}
\mathsf{E}_{11} \arrow[d, "\Phi_1"] \arrow[r, "\Psi_2"] & \mathsf{E}_{11} \arrow[d, "\Phi_1"] \arrow[r, "\Phi_2"] & \cdots \arrow[r, "\Phi_2"] & \mathsf{E}_{11} \arrow[d, "\Phi_1"] \arrow[r, "e_2"] & \mathsf{E}_1 \arrow[d, "\Phi_1"] \\
\mathsf{E}_{11} \arrow[r, "-\Psi_2"] \arrow[d, "m_1"] & \mathsf{E}_{11} \arrow[r, "-\Phi_2"] \arrow[d, "m_1"] & \cdots \arrow[r, "-\Phi_2"] & \mathsf{E}_{11} \arrow[r, "-e_2"] \arrow[d, "m_1"] & \mathsf{E}_1 \arrow[d, "e_1"] \\
\mathsf{E}_1 \arrow[r, "\Psi"] & \mathsf{E}_1 \arrow[r, "\Phi"] & \cdots \arrow[r, "\Phi"] & \mathsf{E}_1 \arrow[r, "e"] & \underline{\mathsf{E}_0}.
\end{tikzcd}
    \]
    Again, using \Cref{lem:Tq-lem1}, we may collapse the squares from the upper right, resulting in a sequence
    \[
\begin{tikzcd}
\mathsf{E}_1 \arrow[d, "\Psi"] \\
\mathsf{E}_1 \arrow[d, "\Phi"] \\
\mathsf{E}_1 \arrow[r, "\Psi"] & \mathsf{E}_1 \arrow[r, "\Phi"] & \cdots \arrow[r, "\Phi"] & \mathsf{E}_1 \arrow[r, "e"] & \underline{\mathsf{E}_0},
\end{tikzcd}
\]
giving the complex $E_{-q - 2}$. The retraction $r_{-q-2}$ is defined by the composition
\[
\begin{tikzcd}
{[T_{-q-2}] = D([T_{-2}], [T_{-q}])} 
    \arrow[r, "\simeq"] 
& {D(E_{-2}, E_{-q})}
    \arrow[r, "\simeq"] 
& {E_{-q-2}}.
\end{tikzcd}
\]

Now, as we have seen in \Cref{prop:cc-rm2}, the left-hand full twist that transforms $T_{-q}$ to $T_{-q-2}$ can be realized by the following sequence of moves
\[
    \begin{tikzcd}
    {T_{-q}} \arrow[r, "R2"] & {T'_{-q}} \arrow[r, "\text{c.c.}"] & {T_{-q-2}}.
    \end{tikzcd}
\]
On the corresponding chain complex, consider the two crossing change chain maps
\[
    \begin{tikzcd}
    {[T_{-q}]} \arrow[r, "\rho"] & {[T'_{-q}]} \arrow[r, "f_0", dashed, shift left] \arrow[r, "f_1"', shift right] & {[T_{-q-2}]}.
    \end{tikzcd}
\]
Let $g_i := f_i \circ \rho$ for $i = 0, 1$. We define maps $g'_0, g'_1$ between the simplified complexes that makes the following diagram commute.
\[
    \begin{tikzcd}[row sep=4em, column sep=5em]
{[T_{-q}]} 
    \arrow[r, "r_{-q}"] 
    \arrow[d, "g_0"', dashed, shift right] 
    \arrow[d, "g_1", shift left] 
& E_{-q}
    \arrow[d, "g'_0"', dashed, shift right] 
    \arrow[d, "g'_1", shift left] \\
{[T_{-q - 2}]} 
    \arrow[r, "r_{-q-2}"]
& {E_{-q - 2}}.
    \end{tikzcd}
\]
\begin{lem}
\label{lem:g01-base}
    When $q = 0$, the maps $g'_1, g'_0$ are given as follows: 
    \[
\begin{tikzcd}[row sep=3em]
{[T_0]} \arrow[d, "g_0"', dashed, shift right] \arrow[d, "g_1", shift left] & = & \{\ 0 \arrow[r] & 0 \arrow[r] & \underline{\mathsf{E}_1} \arrow[lld, "e" description, dashed] \arrow[d, "I" description, shift right = 2]\ \} \\
{[T_{-2}]} & \simeq & \{\ \mathsf{E}_0 \arrow[r] & \mathsf{E}_0 \arrow[r] & \underline{\mathsf{E}_1}\ \}.
\end{tikzcd}
    \]
    The dashed diagonal arrow gives $g'_0$ and the solid vertical arrow gives $g'_1$. 
\end{lem}

\begin{proof}
    Straightforward from the explicit definitions. 
\end{proof}

\begin{prop}
\label{prop:Tq-Tq+2}
    For odd $q \geq 1$, the maps $g'_1, g'_0$ are given as follows: 
    \[
\begin{tikzcd}[row sep=3em]
{[T_{-q}]} \arrow[d, "g_0"', dashed, shift right] \arrow[d, "g_1", shift left] & \simeq\ & \{\ 0 \arrow[r] & 0 \arrow[r] & \mathsf{E}_0 \arrow[r] \arrow[d] \arrow[lld, dashed] & \cdots \arrow[r] \arrow[lld, dashed] & \mathsf{E}_0 \arrow[r] \arrow[d] & \mathsf{E}_0 \arrow[r, "e"] \arrow[d] \arrow[lld, dashed] & \underline{\mathsf{E}_1} \ \} \arrow[d] \arrow[lld, "e" description, dashed, pos=0.25] \\
{[T_{-q-2}]} & \simeq & \{\ \mathsf{E}_0 \arrow[r] & \mathsf{E}_0 \arrow[r] & \mathsf{E}_0 \arrow[r] & \cdots \arrow[r] & \mathsf{E}_0 \arrow[r] & \mathsf{E}_0 \arrow[r, "e"] & \underline{\mathsf{E}_1} \ \}.
\end{tikzcd}
    \]
    The dashed vertical arrows give $g'_0$ and the solid vertical arrows give $g'_1$. All of the vertical and diagonal arrows are identity morphisms, except for the rightmost diagonal arrow $e: \mathsf{E}_1 \to \mathsf{E}_0$. 
\end{prop}

\begin{proof}
        First, the result for $g_1$ is obvious from the following commutative diagram
    \[
\begin{tikzcd}[row sep=3em, column sep=5em]
{[T_{-q}] = D([T_{-}], [T_{-q}])} \arrow[d, "\simeq"] \arrow[r, "{D(g_1, I)}"] & {D([T_{-2}], [T_{-q}])} \arrow[r, "\simeq"] & {D(E_{-2}, E_{-q})} \arrow[d, "\simeq"] \\
E_{-q} \arrow[r, equal] & {D(\mathsf{E}_1, E_{-q})} \arrow[r, hook] \arrow[ru, hook] & {E_{-q - 2}}. 
\end{tikzcd}
    \]
    The commutativity of the left trapezoid follows from \Cref{lem:g01-base} and the locality of $D$, and the right triangle from the construction of the vertical retraction showing that $E_{-q}$ is unchanged. Next, for $g_0$, we observe how the degree $-2$ arrows from $S_q$ to $D(S_2, S_{q+2})$
    \[
\begin{tikzcd}[row sep=1.5em]
{D(T_0, T_{-q})} \arrow[r, "=", phantom] \arrow[dd, "{D(g_0, I)}", dashed] & \mathsf{E}_1 \arrow[r] \arrow[dd, "\Delta_1", dashed] & \mathsf{E}_1 \arrow[r] \arrow[dd, "\Delta_1", dashed] & \cdots \arrow[r] & \mathsf{E}_1 \arrow[r] \arrow[dd, "\Delta_1", dashed] & \underline{\mathsf{E}_0} \arrow[dd, "e_1", dashed] \\ 
& & & & & \\
{D(T_{-2}, T_{-q})} \arrow[r, "=", phantom] & \mathsf{E}_{11} \arrow[d] \arrow[r] & \mathsf{E}_{11} \arrow[d] \arrow[r] & \cdots \arrow[r] & \mathsf{E}_{11} \arrow[d] \arrow[r] & \mathsf{E}_1 \arrow[d] \\ & \mathsf{E}_{11} \arrow[r] \arrow[d] & \mathsf{E}_{11} \arrow[r] \arrow[d] & \cdots \arrow[r] & \mathsf{E}_{11} \arrow[r] \arrow[d] & \mathsf{E}_1 \arrow[d] \\ & \mathsf{E}_1 \arrow[r] & \mathsf{E}_1 \arrow[r] & \cdots \arrow[r] & \mathsf{E}_1 \arrow[r] & \underline{\mathsf{E}_0}
\end{tikzcd}
    \]
    are modified by the reduction. Note that each of the vertical arrows $\Delta_1$ splits a center circle from the upper arc of $\mathsf{E}_1$. Consider the following parts in the diagram:
    \[
\begin{tikzcd}
 & & \mathsf{E}_1 \arrow[d, "\Delta_1", dashed] & & \mathsf{E}_1 \arrow[d, "\Delta_1", dashed] & & \underline{\mathsf{E}_0} \arrow[d, "e_1", dashed] \\
 & \mathsf{E}_{11} \arrow[r, "\Psi_2"] \arrow[d, "\Phi_1"] & \mathsf{E}_{11} & \mathsf{E}_{11} \arrow[d, "\Phi_1"] \arrow[r, "\Phi_2"] & \mathsf{E}_{11} & \mathsf{E}_{11} \arrow[r, "e_2"] \arrow[d, "\Phi_1"] & \mathsf{E}_1 \\
\mathsf{E}_{11} \arrow[r, "\Phi_2"] \arrow[d, "m_1"] & \mathsf{E}_{11} & \mathsf{E}_{11} \arrow[d, "m_1"] \arrow[r, "\Psi_2"] & \mathsf{E}_{11} & \mathsf{E}_{11} \arrow[r, "\Phi_2"] \arrow[d, "m_1"] & \mathsf{E}_{11} & \\
\mathsf{E}_1 & & \mathsf{E}_1 & & \mathsf{E}_1 & & \end{tikzcd}
    \]
    With \Cref{lem:Tq-lem1}, the reduction in the top row of $D(E_{-2}, E_{-q})$ transforms these parts into
    \[
\begin{tikzcd}
 & & \mathsf{E}_1 \arrow[ldd, "I" description, dashed, bend left] & & \mathsf{E}_1 \arrow[ldd, "I" description, dashed, bend left] & & \mathsf{E}_0 \arrow[ldd, "e" description, dashed, bend left] \\
 & & & & & & \\
\mathsf{E}_1 \arrow[r, "I"] \arrow[d, "I"] & \mathsf{E}_1 & \mathsf{E}_1 \arrow[d, "I"] \arrow[r, "I"] & \mathsf{E}_1 & \mathsf{E}_1 \arrow[r, "I"] \arrow[d, "I"] & \mathsf{E}_1 & \\
\mathsf{E}_1 & & \mathsf{E}_1 & & \mathsf{E}_1 & &                                        
\end{tikzcd}    
    \]
    and by the reduction in the middle row 
    \[
\begin{tikzcd}
 & & \mathsf{E}_1 \arrow[llddd, "I" description, dashed, bend left] & & \mathsf{E}_1 \arrow[llddd, "I" description, dashed, bend left] & & \mathsf{E}_0 \arrow[llddd, "e" description, dashed, bend left] \\
 & & & & & & \\
 & & & & & & \\
\mathsf{E}_1 & & \mathsf{E}_1 & & \mathsf{E}_1 & &.                                          
\end{tikzcd}
    \]
    At the left end, the part
    \[
\begin{tikzcd}
\mathsf{E}_1 \arrow[d, "\Delta_1", dashed] \arrow[r] & \mathsf{E}_1 \arrow[d, "\Delta_1", dashed] \arrow[r] & \mathsf{E}_1 \arrow[d, "\Delta_1", dashed] \\
\mathsf{E}_{11} \arrow[r] \arrow[d, "\Phi_1"] & \mathsf{E}_{11} \arrow[r] \arrow[d, "\Phi_1"] & \mathsf{E}_{11} \\
\mathsf{E}_{11} \arrow[r] \arrow[d, "m_1"] & \mathsf{E}_{11} & \\
\mathsf{E}_1 & &                           
\end{tikzcd}
    \]
    transforms into
    \[
\begin{tikzcd}[column sep=1em]
\mathsf{E}_1 \arrow[d, "I", dashed] \arrow[rr] & & \mathsf{E}_1 \arrow[rr] \arrow[lldd, "I" description, dashed, bend left] & & \mathsf{E}_1 \arrow[lldd, "I" description, dashed, bend left] \\
\mathsf{E}_1 \arrow[d, "\Psi"] & & & &                                            \\
\mathsf{E}_1 \arrow[d, "\Phi"] \arrow[r, "\oplus", phantom] & \mathsf{E}_1 \arrow[r, "I"] \arrow[ld, "I"] & \mathsf{E}_1 & &                                            \\
\mathsf{E}_1 & & & &
\end{tikzcd}
    \]
    and then into 
    \[
\begin{tikzcd}
\mathsf{E}_1 \arrow[d, "I" description, dashed] \arrow[r] & \mathsf{E}_1 \arrow[r] \arrow[ldd, "I" description, dashed, bend left] & \mathsf{E}_1 \arrow[llddd, "I" description, dashed, bend left] \\
\mathsf{E}_1 \arrow[d, "\Psi"] & & \\
\mathsf{E}_1 \arrow[d, "\Phi"] & & \\
\mathsf{E}_1 & & .                                          
\end{tikzcd}    
    \]
\end{proof}

For each $k \geq 0$, the maps
\[
    g_0, g_1: T_{-2k-1} \rightarrow T_{-2k-3}
\]
give rise to the maps
\[
    \bar{g}_0, \bar{g}_1: \CKh(T^*_{2, 2k+1}) \rightarrow \CKh(T^*_{2, 2k+3})
\]
by taking their closures. 

\begin{cor}
\label{cor:map-g-on-ckh}
    The two maps $\bar{g}_0, \bar{g}_1$ are described as follows:
    \[
\begin{tikzcd}[row sep=1em, column sep=1.5em]
{\CKh(T^*_{2, q})} \arrow[dd, "\bar{g}_0"', dashed, shift right] \arrow[dd, "\bar{g}_1", shift left] \arrow[r, "\simeq", phantom] & \{ & & A \arrow[r, "U"] \arrow[dd] \arrow[lldd, dashed] & A \arrow[lldd, dashed] \arrow[r, "\cdots", phantom] \arrow[dd] & A \arrow[r, "U"] \arrow[dd] & A \arrow[dd] \arrow[dd] \arrow[ld, no head, dashed] & A \arrow[r, "U"] \arrow[lldd, dashed] \arrow[dd] & A \arrow[lldd, dashed] \arrow[dd] & 0 \arrow[r, dotted] \arrow[dd, dotted] & {\underline{A} \, \}} \arrow[dd] \arrow[lldd, dashed] \\
 & & & & {} \arrow[ld, dashed] & {} & & & & & \\
\CKh(T^*_{q+2}) \arrow[r, "\simeq", phantom] & \{ A \arrow[r, "U"'] & A & A \arrow[r, "U"'] & A \arrow[r, "\cdots", phantom] & A \arrow[r, "U"'] & A & A \arrow[r, "U"'] & A & 0 \arrow[r, dotted] & {\underline{A} \, \}}
\end{tikzcd}
    \]
    Here, the dashed diagonal arrows indicate $\bar{g}_0$ and the solid vertical arrows indicate $\bar{g}_1$. On the right side, all of the vertical and diagonal arrows are identity maps on $A$. 
\end{cor}

\begin{proof}
    Immediate from \Cref{lem:Tq-lem1,cor:CKh-T_2q,prop:Tq-Tq+2}.
\end{proof}

\begin{cor}
\label{cor:seq-from-trefoil}
    For any $k \geq 1$, the homology group $\Kh(T^*_{2, 2k + 1})$ is generated by the images of all combinations of the composite maps 
    \[
    \begin{tikzcd}
    {\Kh(T^*_{2, 3})} 
        \arrow[r, "\bar{g}_0", dashed, shift left] \arrow[r, "\bar{g}_1"', shift right] & 
    {\Kh(T^*_{2, 5})} 
        \arrow[r, "\bar{g}_0", dashed, shift left] \arrow[r, "\bar{g}_1"', shift right] & 
    \cdots
        \arrow[r, "\bar{g}_0", dashed, shift left] \arrow[r, "\bar{g}_1"', shift right] & 
    {\Kh(T^*_{2, 2k + 1})}.
    \end{tikzcd}
    \]
\end{cor}

\begin{rem}
    If we start from $T^*_{2, 1}$, which is the unknot, then the former statement of \Cref{cor:seq-from-trefoil} does not hold: it only generates the even homological grading part of $\Kh(T^*_{2, 2k + 1})$.
\end{rem}

\begin{cor}
\label{cor:special-imm-cob-from-unknot}
    Let $k$ be a positive integer, and $s_-$ be an integer such that $0 \leq s_- \leq k$. Let $U_1$ denote the unknot, and consider an immersed oriented cobordism $S$ from $U_1$ to $T^*_{2, 2k + 1}$, with genus $g = k - s_-$ and $s_-$ negative double points, obtained from the following sequence
    \[
\begin{tikzcd}
U_1 \arrow[r, "R1_R"] & {T^*_{2, 1}} \arrow[r, "S^{(\times)}_1"] & {T^*_{2, 3}} \arrow[r, "S^{(\times)}_2"] & \cdots \arrow[r, "S^{(\times)}_k"] & {T^*_{2, 2k + 1}}.
\end{tikzcd}
    \]
    Here, among the $k$ arrows after the first twist, we assume that $g$ of them are  genus 1 embedded cobordisms of the form
    \[
        S_i\colon {T^*_{2, 2i+1}} 
            \xrightarrow{\ R2\ } 
        {(T'_{2, 2i+1})^*}
            \xrightarrow{\ \ \ } 
        {T^*_{2, 2i+3}}    
    \]
    where the latter arrow is given by the move described in \Cref{prop:f0-explicit}. The remaining $s_-$ of them are genus 0 immersed cobordisms of the form 
    \[
        S^\times_i\colon {T^*_{2, 2i+1}} 
            \xrightarrow{\ R2\ } 
        {(T'_{2, 2i+1})^*}        
            \xrightarrow{\ \text{c.c.}\ } 
        {T^*_{2, 2i+3}}.
    \]
    Then the immersed cobordism map of lowest degree
    \[
        \Kh^{\low}(S)\colon
        \Kh(U_1) \cong A \rightarrow \Kh(T^*_{2, 2k+1})
    \]
    is surjective onto the homological grading $-2s_-$ part of $\Kh(T^*_{2, 2k + 1})$, which is 
    \[
        \Kh^{-2s_-}(T^*_{2, 2k + 1}) \cong A / \operatorname{Im}(U). 
    \]
\end{cor}

\begin{proof}
    On the chain level, we have $\phi^{\low}_{S^\times_i} = \bar{g}_0$ by definition, and $\phi^{\low}_{S_i} = \bar{g}_1$ from \Cref{prop:f0-explicit}. 
\end{proof}

\section{Immersed cobordism maps for instanton cube complexes}

\label{subsection:Review of instanton knot homology}

\subsection{Statement of the main result}

\label{subsect: Cobordism maps on spectral sequences}
In this section, we shall construct the chain map $\phi^\sharp _S$ which appeared in the main theorem, restated as follows. 

\begin{thm} [\cref{thm:induced-map}] \label{thm:restated}

Let $L, L'$ be links in $\R^3$ with diagrams $D, D'$, and $S$ a normally immersed cobordism in $[0,1]\times \R^3$ from $L$ to $L'$ with a decomposition into elementary cobordisms (as in \Cref{subsubsec:imm-cob-combi}). 
Then, there is a cobordism map 
\[
\phi^\sharp_{S} : \CKh^\sharp (D) \to \CKh^\sharp (D') 
\]
of order
\[
    \geq \left(-2s_+ + \frac{1}{2} (S\cdot S),\ 
    \chi(S) + \frac{3}{2} (S\cdot S) - 6s_+ \right)
\] 
satisfying the following properties: 
\begin{enumerate}
    \item The following diagram commutes 
  \[
  \begin{CD}
    H_*( \CKh^\sharp (D)) @>{(\phi^\sharp_{S})_*}>> H_*(\CKh^\sharp (D')) \\
  @VVV    @VVV \\
     I^\sharp(L)   @>{I^\sharp(S)}>>   I^\sharp(L').
  \end{CD}
\]
    \item The induced map on $E_2$-term coincides with $\Kh^{\low}(S)$ defined in Khovanov homology. 
\end{enumerate}

\end{thm}

\begin{rem}\label{normal_Euler}
 The term $S\cdot S$ appearing in \Cref{thm:induced-map} is defined as follows. 
 For a normal immersion $\iota: S \to [0,1]\times \R^3$, let $\nu(S) \to S$ denote its normal bundle, and $e(S)$ the relative Euler class of $\nu(S)$ in the local coefficient cohomology $H^2(S, \partial S; o(S))$ with respect to the orientation local system $o(S)$ of $S$ and non-zero sections on the boundaries determined from Seifert framings. We define 
\[
 S\cdot S := \langle e(S), [S,\partial S;  o(S)] \rangle \in \Z,  
 \]
 where $[S,\partial S;  o(S)]$ is the relative fundamental class of $ S$ with respect to the  orientation local system $o(S)$.
 \end{rem}

The cobordism map $\phi^\sharp_S$ is defined as the composition of the maps $\phi^\sharp_{S_i}$ associated to each elemental cobordism $S_i$. 
In \cite{ISST1}, the corresponding cobordism map $\phi^\sharp_{S_i}$ has been defined except for the crossing change cobordisms satisfying (1) and (2) of \cref{thm:restated}. Therefore, it is sufficient to construct the corresponding cobordism maps for the crossing change cobordisms satisfying (1) and (2) of \cref{thm:restated}.
Following the geometric description of the immersed cobordism maps given in \Cref{subsect: Geometric description}, we shall make use of the Hopf link. 

\subsection{Some computations for the Hopf link in instanton theory}\label{gradings}

We follow the notations given in \cite{KM11u, ISST1}.

\subsubsection{Framed instanton homology of Hopf link}
Let $H$ be the unoriented diagram of the Hopf link shown in \Cref{fig:hopf-instanton-cube},
which is the unoriented mirror image of \Cref{fig:ckh-hopf}.
We calculate the framed instanton homology $I^\sharp (H)$ of $H$. 

Note that the traceless $SU(2)$-character variety of $H$
\[
R(H) := \{ \rho \in \operatorname{Hom} (\pi_1(S^3 \setminus H), SU(2) ) | \operatorname{Tr}\rho (m)= \operatorname{Tr}\rho (m') =0\}  / SU(2)
\]
consists of two reducibles 
\[
R(H) = \{[ \rho_+],[\rho_-]\},
\]
where $m$ and $m'$ are the meridians of the components of $H$. 
We have two orientations $\mathfrak{o}_\pm$ up to overall sign so that the crossings are positive or negative and these orientations correspond to the representations $[ \rho_+],[\rho_-]$ as order of eigenvalues $i$ and $-i$.  It is observed in \cite{daemi2024unoriented} that the Floer gradings of $[ \rho_+]$ and $[\rho_-]$ differ by $2$. 
Now, we consider the link 
\[
H^\sharp := H \sqcup H_\omega \subset S^3
\]
and the corresponding representation space 
\[
R^\sharp (H) := 
\left\{
\begin{aligned}
 &\rho \in \operatorname{Hom}\!\bigl(\pi_1(S^3 \setminus H \cup H \cup \omega), SU(2)\bigr) \ \big| \\[2mm]
 &\operatorname{Tr}\rho (m)= \operatorname{Tr}\rho (m') =0,\ 
   \operatorname{Tr}\rho (m_H)= \operatorname{Tr}\rho (m'_H) =0,\ 
   \rho ( m_\omega ) = -\id
\end{aligned}
\right\} \big/ SU(2),  
\]
where $\omega$ is a small arc connecting different comoponents of $H$ and $m_\omega$ is a small merdian loop of $\omega$.

Roughly speaking, the framed instanton Floer homology of $H$ is an infinite-dimensional Morse homology whose critical point set is $R^\sharp (H)$.

On the other hand, we have a fiber-product formula. 
\[
R^\sharp (H) = R^{\mathrm{fr}}(H) \times_{SU(2)}   R^{\mathrm{fr}}(H_\omega) = S^2 \sqcup S^2, 
\]
where 
\begin{align*}
    &R^{\mathrm{fr}}(H)  = \{ \rho \in \operatorname{Hom} (\pi_1(S^3 \setminus H), SU(2) ) | \operatorname{Tr}\rho (m)= \operatorname{Tr}\rho (m') =0\} = S^2 \sqcup S^2  \\ 
    & R^{\mathrm{fr}}(H_\omega) = \{ \rho \in \operatorname{Hom} (\pi_1(S^3 \setminus H \cup \omega ), SU(2) ) | \operatorname{Tr}\rho (m)= \operatorname{Tr}\rho (m') =0, \rho( m_\omega ) = - \id\} =SO(3) .
\end{align*}
Therefore, we can take a Morse perturbation of the singular Chern--Simons functional for $H^\sharp$ so that the critical point set is identified with four points with even Morse gradings, hence having perfect Morse homology. This observation directly shows 
\begin{align}\label{sharp_of_H}
I^\sharp (H ) \cong \Z_{(0)}^2 \oplus \Z_{(2)}^2.  
\end{align}

For the latter section, we always fix a small perfect perturbation for the Hopf link $H$ so that there is no differential on $C^\sharp (H)$.
\begin{rem}
    This isomorphism \eqref{sharp_of_H} can also be verified using two different ways: 
    \begin{itemize}
        \item applying skein exact triangle to $H$ and  
        \item comparing with Daemi--Scaduto's equivariant link singular instanton Floer $\mathcal{S}$-complex \cite{daemi2024unoriented}: 
        \[
        (\tilde{C} (H ), \wt{d}, \chi) 
        \]
       equipped with the trivial local coefficient with the framed instanton Floer homology $I^\sharp (H)$, i.e. 
        \[
        I^\sharp (H) \cong H_* ( \operatorname{Cone} (2\chi =0 : \wt{C} (H) \to \wt{C} (H) ) ) \cong \Z_{(0)}^2 \oplus \Z_{(2)}^2 . 
        \]
    \end{itemize}
\end{rem}

\subsection{Computations of cube complexes of $H$}

Let $K_{(2,2)}$ be the unoriented diagram of the Hopf link shown in \Cref{fig:hopf-instanton-cube},
which is the unoriented mirror image of \Cref{fig:ckh-hopf}.
Then, $K_{(i,j)}$ for $i,j \in \{0,1\}$ corresponds to the resolutions of the Hopf link on each crossing. 
We put 
\[
C^\sharp_{i,j}:= C^\sharp (K_{(i,j)}) . 
\]

We give concrete computations of gradings for the positive/negative Hopf link here, which are identical to the gradings of Khovanov homology of its mirror. 
Then, an instanton cube complex of $K_{(2,2)}$ is described as 
\[
\CKh^\sharp(K_{(2,2)}) = C^\sharp_{1,1} \oplus ( C^\sharp_{1, 0}\oplus C^\sharp_{0, 1} )  \oplus C^\sharp_{0,0} = V^2 \oplus ( V \oplus  V) \oplus V^2,  
\]
where 
$V=\langle \bv_+, \bv_- \rangle$ is a free abelian group. 
As it is observed that 
\[
E^2 (\CKh^\sharp (K_{(2,2)}) ) = \Kh (K_{(2,2)}) = \Z^4, 
\]
there cannot be higher differential on the complex $E^2 (\CKh^\sharp (K_{(2,2)}) )$. 

Also, we shall use the intermediate complex 
\[
\mathbf{C}[(0, 2), (1,2)] = C^\sharp_{1,2} \oplus    C^\sharp_{0, 2} = V \oplus V  
\]
which is also introduced in \cite{KM11u}. 
We calculate the homological/quantum gradings of generators. 
Note that the defitions of homological and quantum gradings of the cube complexes are given as 
\begin{align*}
&h|_{C^\sharp(D_v)} = -|v|_1 + n_- \\
&q|_{C^\sharp(D_v)} = Q- |v|_1  - n_+ + 2 n_-, 
\end{align*}
where the grading $Q$ is defined as the integer grading of a generator of $C^\sharp (D_v)$ obtained from an identification 
\[
\gamma_v \colon C (D_v) \to V^{\otimes r(D_v)},
\]
$n_+, n_-$ are the numbers of positive and negative crossings of $D$ respectively, and $r(D_v)$ denotes the number of components of $D_v$. Note that $Q \operatorname{mod} 4$ coincides with the mod four grading of each unreduced Floer complex. 
Here $V=\langle \bv_+, \bv_- \rangle$ is a free graded abelian group with $\bv_+$ and $\bv_-$ in degrees 1 and $-1$ respectively, and give $V^{\otimes r(D_v)}$ the tensor-product grading.

We first fix an orientation of $K_{(2,2)}$ as  $K^+_{(2,2)}$. 
In this case, we have $n_+=2, n_-=0$.
We have 
\[
\CKh^\sharp(K_{(2,2)}^+) = V^2 \{ -4\} \oplus (V\{-3\} \oplus V\{-3\}) \oplus V^2\{-2\}
\]
with an identification
\[
E^2 \CKh^\sharp(K_{(2,2)}^+) = R\{-2, -6\} \oplus R\{0, -2\}
\]
in the notations in the earlier section. 

If we consider $K_{(2,2)}^-$, the oriented resolution is $K_{(1,2)}$ with $n_+=0$ and $n_-=2$. For the other orientation $K^-_{(2,2)}$, we have 
\[
\CKh^\sharp(K_{(2,2)}^-) = V^2 \{ 2\} \oplus (V\{3\} \oplus V\{3\}) \oplus V^2\{4\} 
\]
with an identification
\[
E^2 \CKh^\sharp(K_{(2,2)}^+) = R\{0, 2\} \oplus R\{2, 6\}
\]
in the notations in the earlier section.

    In the first case, we only regard the upper crossing as a diagrammatic crossing. 
Then, the oriented resolution in this case is $K_{(0,2)}$ in Figure 7. 

For $K_{(2,2)}^+$ with $n_+=1$ and $n_-=0$, we have an identification 
\[
\mathbf{C}[(0, 2), (1,2)] = C^\sharp_{1,2} \oplus    C^\sharp_{0, 2} \cong V\{-2\} \oplus V\{-1\}. 
\]
For $K_{(2,2)}^-$ with $n_+=0$ and $n_-=1$, we have 
\[
\mathbf{C}[(0, 2), (1,2)] = C^\sharp_{1,2} \oplus    C^\sharp_{0, 2} \cong V\{1\} \oplus V\{2\}. 
\]

Also, Kronheimer--Mrowka \cite{KM11u} constructed two comparison maps 
\begin{align*}
&\Phi : C^\sharp (H) \to \mathbf{C}[(0, 2), (1,2)](D) 
\\ 
&\Phi =
\begin{pmatrix}
f_{(2,2)(1,2)}\\
j_{(2,2)(0,2)}
\end{pmatrix} \text{ and } \\ 
&\Phi' : \mathbf{C}[(0, 2), (1,2)] \to \CKh^\sharp ( K_{2,2} ) \\
&\Phi' =
\begin{pmatrix}
f_{(1,2)(1,1)} & 0 \\
j_{(1,2) (1,0)} & 0 \\
f_{(1,2)(0,1)} & f_{(0,2)(0,1)}\\
j_{(1,2)(0,0)} & j_{(0,2)(0,0)}
\end{pmatrix}
\end{align*}
constructed by counings parameterized moduli spaces such that 
$\Phi$ and $\Phi'$ are chain maps that induce quasi-isomorphisms on the homologies. Since these induce isomorphisms, we see the differential of $\mathbf{C}[(0, 2), (1,2)] = \Z^2$ is trivial. 
We will detect some of these components later in order to compare them with Khovanov theory.    

\subsection{Crossing change map}

Suppose $D'$ is obtained from $D$ by a diagrammatic crossing change. Let us denote a natural immersed cobordism $S$ in $[0,1]\times S^3$ from $D$ to $D'$. As it is illustrated in \Cref{fig:x-ch-hopf}, we consider the following 3-steps: 
\begin{itemize}
    \item[(I)]  We put an unoriented Hopf link $H$ near the crossing point equipped with blown-up cobordism which is described as a map 
    \[
   \phi^\sharp_{\sqcup H} : \CKh^\sharp (D) \to  \CKh^\sharp (D \sqcup H) ,
    \]
    \item[(II)] We do two 1-handle surgeries with respect to both components of $H$, the two red bands described in \cref{fig:x-ch-hopf-proof}, which is described as a map 
    \[
   \phi^\sharp_{h^1_1\sqcup h^1_2} :  \CKh^\sharp (D\sqcup H) \to  \CKh^\sharp (D''), 
    \]
    \item[(III)] We do the Reidemeister move described in \cref{fig:x-ch-hopf-proof} to get the desired crossing change move, which is described as a map 
    \[
     \phi^\sharp_{\mathrm{RII}} :  \CKh^\sharp (D'')\to \CKh^\sharp (D')  
    \]
    \end{itemize}
By making use of the above three maps, we define 
\[
\phi^\sharp_S := \phi^\sharp_{\mathrm{RII}} \circ \phi^\sharp_{h^1_1\sqcup h^1_2} \circ \phi^\sharp_{\sqcup H} : \CKh^\sharp (D) \to \CKh^\sharp (D'). 
\]
    
For (II) and (III), we use the corresponding cobordism maps in \cite{ISST1}, which have already been verified to satisfy the compatibility conditions \cref{thm:restated}(i) and (ii). Therefore, we only need to care about (I). 
In order to introduce a cobordism map for (I), 
we introduce a new cube complex for a pseudo diagram $D$,
\[
\CKh^\sharp_{\sqcup H} (D) := \bigoplus_{v \in \{0,1\}^N} (D_v \sqcup H),  
\]
where $N$ is the number of crossings of $D$, with the differential $d^\sharp_{\sqcup H}$ which counts the parametrized instanton moduli spaces for 
\[
S_{vu} \sqcup [0,1]\times H  : D_v \sqcup H  \to D_u \sqcup H. 
\]
Here, we take a perfect Morse function for the $H$-component so that we have two critical points of even gradings for the Hopf link $H$.
One can check $(d^\sharp_{\sqcup H})^2 =0 $ as it is shown in \cite{KM11u, KM14}. 

\begin{defn}
For the cobordism map of (I), we define 
\[
\phi^\sharp_{\sqcup H} : \CKh^\sharp (D) \to  \CKh^\sharp (D \sqcup H)
\]
 by 
\[
\phi^\sharp_{\sqcup H } = \Phi'_D \circ \Phi_D \circ C^\sharp_{D_H}
\]
where the two maps $\Phi_D$, $\Phi'_D$ and $C^\sharp_{D_H} $ are given as follows: 

\begin{enumerate}[leftmargin=2em]
    \item 
{\bf Definition of $\Phi_D$.}

Suppose $D$ has $N$-crossings. We shall define 
\[
\mathbf{C}[(0, 2), (1,2)](D) := \bigoplus_{i \in \{0,1\}, v \in \{0,1\}^N } C^\sharp (D_v \sqcup K_{(i,2) }) = \bigoplus_{ v \in \{0,1\}^N } C^\sharp (D_v \sqcup K_{(1,2) }) \bigoplus_{ v \in \{0,1\}^N } C^\sharp (D_v \sqcup K_{(0,2) })  . 
\]

We define a chain map 
\[
\Phi_D : \CKh^\sharp_{\sqcup H} (D) \to \mathbf{C}[(0, 2), (1,2)](D) 
\]
defined as 
\[
\Phi_D =
\begin{pmatrix}
f_{(2,2)(1,2)}\\
j_{(2,2)(0,2)}
\end{pmatrix}, 
\]
where the maps $f_{(2,2)(1,2)}$ and $j_{(2,2)(0,2)}$ are the cobordism maps for 
\[
S_{(2,2)(1,2)} \sqcup [0,1]\times D_v : K_{2,2} \sqcup D_v \to K_{1,2} \sqcup D_v 
\]
and the families cobordism map for 
\[
S_{(2,2)(0,2)} \sqcup [0,1]\times D_v: K_{2,2} \sqcup D_v \to K_{0,2} \sqcup D_v 
\]
with 1-parameter family of metrics. 
We follow the notations in \cite[Proposition 6.11]{KM11u}. 

\item 
{\bf Definition of $\Phi_D'$.}

We next define the chain map 
\[
\Phi'_D : \mathbf{C}[(0, 2), (1,2)](D) \to \CKh^\sharp (D \sqcup K_{2,2} ).  
\]
With the order of crossings of $K_{2,2}$, we have
\[
\CKh^\sharp (D \sqcup K_{2,2} ) = \bigoplus_{ v \in \{0,1\}^N } C^\sharp (D_v \sqcup K_{(1,1) })  \bigoplus_{ v \in \{0,1\}^N } C^\sharp (D_v \sqcup K_{(1,0) })  \bigoplus_{ v \in \{0,1\}^N } C^\sharp (D_v \sqcup K_{(0,1) })  \bigoplus_{ v \in \{0,1\}^N } C^\sharp (D_v \sqcup K_{(0,0) })  . 
\]
With this decomposition, $\Phi'_D$ is defined as 
\begin{align*}
\Phi'_D =
\begin{pmatrix}
f_{(1,2)(1,1)} & 0 \\
j_{(1,2) (1,0)} & 0 \\
f_{(1,2)(0,1)} & f_{(0,2)(0,1)}\\
j_{(1,2)(0,0)} & j_{(0,2)(0,0)}
\end{pmatrix}, 
\end{align*}
where $f_{(1,2)(1,1)}, f_{(1,2)(0,1)}$, and $f_{(0,2)(0,1)}$ are the cobordism maps for 
\begin{align*}
    S_{(i,j)(i', j')} \sqcup [0,1]\times D_v : K_{i,j} \sqcup D_v \to K_{i',j'} \sqcup D_v 
\end{align*}
and $j_{(1,2) (1,0)}$, $j_{(1,2)(0,0)}$ and $j_{(0,2)(0,0)}$ are the 1-parameter families cobordism maps for 
\begin{align*}
    S_{(i,j)(i', j')} \sqcup [0,1]\times D_v : K_{i,j} \sqcup D_v \to K_{i',j'} \sqcup D_v . 
\end{align*} 

\item 
{\bf Definition of $C^\sharp_{D_H}$.}

Let $(([0,1]\times S^3)\#\overline{\C P^2}, D_H) \colon 
(S^3,\emptyset ) \to (S^3, K_{(2,2)})$ be a blown-up cobordism of normally immersed disks intersecting positively/negatively at one point \footnote{Note that this notion does not depend on the choices of orientations of $D$.}, where 
\[
D_H \cong D^2_+ \cup D^2_- \subset ([0,1]\times S^3)\#\overline{\C P^2}.
\]

The map 
\[
C^\sharp_{D_H} := C^\sharp_{[0,1]\times D \sqcup D_H}:  \CKh^\sharp (D) \to \CKh^\sharp_{\sqcup H} (D)
\]
is the (unoriented) cobordism map with respect to the blown-up cobordism $ D_H$ over $\Z$, which is just the counting of (un-parametrized) instanton moduli spaces 
\[
M([\a], ([0,1]\times S^3 \# \overline{\C P^2}, [0,1]\times D_v \sqcup D_H), [\b])  
\]
for each resolution  $v \in \{0,1\}^N$, here our notation follows that of \cite{ISST1}. 

\end{enumerate}

\end{defn}
We give two remarks on $\phi^\sharp_S$ for a crossing change cobordism. 
\begin{rem}
    The definition of an immersed cobordism map does not depend on the choices of orientations of surface cobordisms as we have not assumed the existence of orientations on the surface cobordisms. In order to compute the homological and quantum gradings of the map, we need to fix orientations of $D$ and $D'$.
\end{rem}

\begin{rem}
When $D $ is the empty diagram, the composition 
\[
\Phi' \circ \Phi : C^\sharp (K_{(2,2)}=H) \to \CKh^\sharp (H)  
\]
induces the Kronheimer--Mrowka's identification  of $I^\sharp (H) \cong H_*( \CKh^\sharp (H))$. Also, note that the map $\Phi'$ can be regarded as the {\it adding crossing map} introduced in \cite{KM14} since the domain and codomain are the cube complexes of pseudo diagrams. 
\end{rem}

\begin{lem}
    The maps $\Phi'_D$, $\Phi_D$ and $C^\sharp_{D_H}$ are chain maps. Thus, $\phi^\sharp_{\sqcup H}$ is also a chain map. 
\end{lem}
\begin{proof}
    The maps $\Phi'_D$ and $\Phi_D$ are shown to be chain maps by Kronheimer--Mrowka \cite{KM11u} when they gave a concrete quasi-isomorphism $I^\sharp(K) \cong H_*(\CKh^\sharp(D))$. 
    Also, from the definition, $\phi^\sharp_{\sqcup H}$ commutes with differentials since the cobordism $D_H$ does not touch $D$ at all. Thus, the certain counting of the ends of 1-dimensional moduli space shows $\phi^\sharp_{\sqcup H}$ is a chain map. Here, we used the fact that $H\sqcup H_\omega$ has a perfect Morse perturbation with even Floer degrees so that there is no differential on its complex.
    This completes the proof. 
\end{proof}

We calculate gradings.

\begin{lem}\label{compatibility_w_instanton}
Let $D$ and $D'$ be oriented pseudo diagrams. 
Let $S$ be a diagrammatic (positive/negative) crossing change from $D$ to $D'$ compatible with some fixed orientations.
The chain map $\phi^\sharp_S$ has the degree shift $\geq (-2,-6)$ (resp. $\geq (0,0)$) for a positive double point (negative double point). 
\end{lem}

\begin{rem}
We use the following computations. 
The normal Euler number can be computed from the writhe of boundary links. 
We note the normal Euler number has the formula: 
\[
e (S) = - w (D_2) +  w(D_1) 
\]
which is pointed out in \cite{Sa19}, if we have a unoriented $H(2)$-move from $D_1$ to $D_2$, where $w$ denotes the writhe.  For the convention, see \cite[Figure 11]{Sa19}. 
\end{rem}

\begin{proof}
We first compute homological gradings. 
We first fix a label of crossings of $D$ written as $N\cup \{c_*\}$, where $c_*$ denotes the crossing corresponding to the crossing where we do the crossing change. 
Then, a label of crossings of $D'$ can also be written in the same way, $N\cup \{c_*\}$. 
Then, we have decompositions of complexes: 
\begin{align*}
   & \CKh^\sharp (D) = \bigoplus_{v \in \{0,1\} ^N } C^\sharp (D_{v0}) \oplus \bigoplus_{v \in \{0,1\} ^N } C^\sharp (D_{v1}) \\
     &  \CKh^\sharp (D') = \bigoplus_{v \in \{0,1\} ^N } C^\sharp (D'_{v0}) \oplus \bigoplus_{v \in \{0,1\} ^N } C^\sharp (D'_{v1}) .
\end{align*}
Note that $D_{v0}=D'_{v1}$ and $D_{v1}=D'_{v0}$ as diagrams.

From the definition of the crossing change map, we see $\phi^\sharp_S$ preserves $v \in N$, i.e., the map $\phi^\sharp_S$ can be decomposed into 
\[
\phi^\sharp_S: C^\sharp (D_{v0}) \oplus C^\sharp (D_{v1}) \to C^\sharp (D'_{v0}) \oplus C^\sharp (D'_{v1})
\]
for each $v \in \{0,1\}^N$.

The homological gradings of each component are 
\[
-|v|_1 + n_-(D), -|v|_1-1 + n_-(D), -|v|_1 + n_-(D) +1 , -|v|_1+ n_-(D)
\]
for a positive to negative crossing change. This shows the homological grading of $\phi^\sharp_S \geq 0$ in this case.

For a negative-to-positive crossing change, the homological gradings are 
\[
-|v|_1 + n_-(D), -|v|_1-1 + n_-(D), -|v|_1 +n_-(D)-1, -|v|_1 +n_-(D)-2. 
\]
This shows the homological grading of $\phi^\sharp_S \geq -2$ in this case.

Next, we compute quantum gradings.
Note that $\phi^\sharp_S$ is decomposed into 
\[
\phi^\sharp_{\mathrm{RII}} \circ \phi^\sharp_{h^1_1\sqcup h^1_2} \circ \phi^\sharp_{\sqcup H}. 
\]
Note that the shifts of quantum gradings of $\phi^\sharp_{\mathrm{RII}} $ and $\phi^\sharp_{h^1_1\sqcup h^1_2}$ are already computed in \cite{KM14} as 
\[
\gr_q( \phi^\sharp_{\mathrm{RII}} )\geq 0 \text{ and }\gr_q(\phi^\sharp_{h^1_1\sqcup h^1_2})\geq -2. 
\]
Therefore, we only need to give lower bounds of $\gr_q(\phi^\sharp_{\sqcup H})$ as 
\[
\begin{cases}
\gr_q(\phi^\sharp_{\sqcup H}) \geq     -4 \text{ for a negative-to-positive c.c} \\  
  \gr_q(\phi^\sharp_{\sqcup H}) \geq    2 \text{ for positive-to-negative c.c } .
\end{cases}
\]
We regard $\phi^\sharp_{\sqcup H}$ as a map from $CKh^\sharp(D) \to CKh^\sharp(D \sqcup H)$.
In this computation, we take a resolution in $v \in \{0,1\}^{N+1}$ corresponding to all crossings of $D$. 
Let $i, j \in \{0,1\}$ be resolutions for the Hopf link. 
Take a critical point of a perturbed Chern--Simons functional $\beta$ for $D_{v} \sqcup H_\omega$. By taking a concrete Morse perturbation of it, $Q(\beta)$ is defined as the integer-valued Morse degree of the perturbation, which is the same as the gradings that come from
\[
C^\sharp (D_{v}) = H_*(S^2\times \cdots \times S^2) =  V^{r(D_{vi})} . 
\]
Now, the quantum grading of it is given as 
\[
Q(\beta) - |v|_1  - n_+(D) + 2n_-(D). 
\]
We also take a critical point $\alpha$ for $D_{v} \sqcup H_{(i,j)}  \sqcup H_\omega$. Then, for a negative-to-positive crossing change, we have 
\[
Q(\alpha) - |v|_1 -i - j - n_+(D)+ 2n_-(D)-2. 
\]
Therefore, the difference of the gradings is bounded by 
\[
Q(\alpha) - Q(\beta) -2 - i -j. 
\]
For a positive-to-negative crossing change, we have 
\[
Q(\alpha) - |v|_1 -i-j - n_+(D) + 2n_-(D) +4. 
\]
Therefore, the minimum difference of the gradings is 
\[
 Q(\alpha) -Q(\beta)  +4 - i -j.
\]
One can regard $\langle \phi^\sharp_S(\beta), \alpha \rangle$ as a certain counting of parameterized moduli spaces with respect to a geometric cobordism 
\[
T : D_{v} \to D_{v} \sqcup H_{(i,j)} \subset [0,1]\times S^3 \# - \C P^2.  
\]
In such a case, Kronheimer--Mrokwa \cite{KM14} generally gave how to give lower bounds of shifts of the homological and quantum gradings.

For computing the shifts for quantum gradings, 
we generally have 
    \begin{align*}
Q(\alpha) - Q(\beta)  \geq   & \chi(T)+T\cdot T-4\left\lfloor \frac{T\cdot T}{8}\right\rfloor + \dim G
    \end{align*}
    as the shifts of quantum gradings, where $T$ is a geometric cobordism and $G$ is a space of Riemann metrics for parametrized moduli spaces, appeared as the compositions of surface cobordisms corresponding to $\phi^\sharp_{\sqcup H}$.
    Let $T$ be 
    \[
   [0,1]\times D_v \sqcup  ( S_{(i', j'),(i,j)} \circ S_{(2, 2),(i', j')} \circ D_H ) \subset [0,1] \times S^3 \# \overline{\C P^2}, 
    \]
    corresponding to the compositions 
\[
 \Phi_D'\circ \Phi_D\circ C^\sharp _{D_H}, 
\]
where $i,j, i', j'$ are the 
subscripts appeared in the definitions of $\Phi$ and $\Phi'$.

We have the following cases: 
\begin{align*}
(T, \dim G)= 
\begin{cases}
   [0,1]\times D_{v}  \sqcup  ( S_{(1,2), (1, 1)} \circ S_{(2, 2), (1, 2)} \circ D_H ) , \quad \dim G  =0 \\ 
    [0,1]\times D_{v}  \sqcup ( S_{(1,2)(1,0)}\circ S_{(2,2)(1,2)} \circ D_H) , \quad \dim G  =1 \\
      [0,1]\times D_{v}  \sqcup (S_{(1,2)(0,1)}\circ S_{(2,2)(1,2)}\circ  D_H ), \quad \dim G  =  1 \\
    [0,1]\times D_{v}  \sqcup  (S_{(0,2)(0,1)}\circ S_{(2,2)(0,2)}\circ D_H ) , \quad \dim G  = 1 \\
    [0,1]\times D_{v}  \sqcup     (S_{(1,2)(0,0)}\circ S_{(2,2)(1,2)} \circ D_H), \quad \dim G  =2 \\
    [0,1]\times D_{v}  \sqcup     (S_{(0,2)(0,0)}\circ S_{(2,2)(0,2)} \circ D_H), \quad \dim G  = 2 \\
    \end{cases},
\end{align*}
where we have used 
\[
\dim G = \|(i,j)-(i',j')\|_1+ \|(i',j')-(i'',j'')\|_1 -2. 
\]
We calculate $\chi(T)$ and $T\cdot T$ one by one. 
For a given cobordism $T: D_1\to D_2$, if we choose orientations on boundaries with respect to Seifert framings, we have
\[
T \cdot T = w(D_1) - w(D_2)
\]
where $w(D)= n^+(D) - n^- (D)$.

Suppose we consider the negative-to-positive crossing change with respect to the fixed orientations. We have the corresponding positive orientation on $H$. 
We see 
$D_H\cdot D_H = -4$ in this case. 
From the direct computations of the writhes, we see 
\[
\left(S_{(i,j), (i', j')} \circ S_{(i', j'), (i'', j'')}\right) \cdot \left(S_{(i,j), (i', j')} \circ S_{(i', j'), (i'', j'')} \right)  =4
\]
for the above cobordism. 

We next consider the positive-to-negative crossing change with respect to fixed orientations.
We see 
$D_H\cdot D_H = 0$ in this case. 
From the computations of the writhe, we see 
\[
\left( S_{(i,j), (i', j')} \circ S_{(i', j'), (i'', j'')} \right)\cdot \left( S_{(i,j), (i', j')} \circ S_{(i', j'), (i'', j'')} \right) =0
\]
for the above cobordism. 
Note that $\chi(T)=i+j -2$ for the above surface cobordisms.
    If $S$ is a negative-to-positive crossing change, 
    we have 
    \[
    \chi(T) + T \cdot T -  4\left\lfloor \frac{T\cdot T}{8}\right\rfloor + \dim G \geq  i+j-2-4+4+ 2-i-j = 0. 
    \]
    If $S$ is a positive-to-negative crossing change, 
    we have 
       \[
   \chi(T) + T \cdot T -  4\left\lfloor \frac{T\cdot T}{8}  \right\rfloor + \dim G \geq i+j-2 + 0 + 2-i-j =0 . 
    \]
    Therefore, in both cases, we have 
    \begin{align*}
        \operatorname{gr}_q (\phi^\sharp_S) &\geq Q(\alpha) - Q(\beta) -2 - i-j\\
        &\geq \chi(T)+T\cdot T-4\left\lfloor \frac{T\cdot T}{8}\right\rfloor + \dim G -4  \\
        & \geq -4  \text{ for the positive crossing change and } \\ 
        \operatorname{gr}_q (\phi^\sharp_S) &\geq Q(\alpha) - Q(\beta) + 4 -i-j\\
        &\geq \chi(T)+T\cdot T-4\left\lfloor \frac{T\cdot T}{8}\right\rfloor + \dim G +4- i-j  \\
        & \geq 2  \text{ for the negative crossing change. } \\ 
    \end{align*}
    This completes the proof. 
\end{proof}

From the grading computation \cref{gradings}, if we choose a suitable orientation of $H =K_{(2,2)}^-$, we see:  
\begin{lem}\label{E1term}The induced map on the $E^1$-terms:
\[
    E^1 (\phi^\sharp_{\sqcup H}) : E^1 ( \CKh^\sharp (D)) \to  E^1 ( \CKh^\sharp (D \sqcup H) ) 
\]
is given as the compositions of the following forms: 
\[
\begin{pmatrix} f_{(1,2)(1,1)} \circ f_{(2,2)(1,2)}\circ C^\sharp_{D_H} \\ 
0\\
0 \\
0
\end{pmatrix}
\]
with repspect to the decomposition
\[
\CKh^\sharp ( D \sqcup K_{(2,2)}^-) = \bigoplus_{ v \in \{0,1\}^N } C^\sharp (D_v \sqcup K_{(1,1) })  \bigoplus_{ v \in \{0,1\}^N } C^\sharp (D_v \sqcup K_{(1,0) })  \bigoplus_{ v \in \{0,1\}^N } C^\sharp (D_v \sqcup K_{(0,1) })  \bigoplus_{ v \in \{0,1\}^N } C^\sharp (D_v \sqcup K_{(0,0) }). 
\]

\end{lem}

Now, we prove that $\phi^\sharp_S$ recovers $I^\sharp_S$. 
\begin{prop}
    For an immersed surface cobordism $S \subset [0,1]\times \R^3$ from $D$ to $D'$ with a decomposition to a sequence of elementary cobordisms,  we have the commutative diagram :
    \[
      \begin{CD}
     H_*(\CKh^\sharp(D)) @>{\phi_{S}^\sharp}>> H_*(\CKh^\sharp(D')) \\
  @V{}VV    @VVV \\
     I_*^\sharp(K) = H_* ( \CKh^\sharp(D), d^\sharp)   @>{I^\sharp(S)}>>  I_*^\sharp(K') = H_* ( \CKh^\sharp(D'), d^\sharp)
  \end{CD}
    \]
    where the vertical arrows are the identifications $I_*^\sharp(K) = H_* ( \CKh^\sharp(D), d^\sharp) $ given in \cite{KM11u}. 
\end{prop}

\begin{proof}
Let \[
S = S_1 \circ S_1 \circ \cdots \circ S_N
\]
be a decomposition of $S$ into elementary cobordisms. If $S_i$ is either a Reidemeister move or a Morse move, the corresponding commutativity is already proven in \cite{ISST1}. 
Moreover, the instanton cobordism maps satisfy the following composition law 
\[
I^\sharp_{S_1 \circ S_1 \circ \cdots \circ S_N}  = I^\sharp_{S_1} \circ I^\sharp_{S_2}\circ \cdots \circ I^\sharp_{S_N}.  
\]
Therefore, we only need to confirm the corresponding commutative diagram for a diagramatic crossing change cobordism $S : D \to D'$.

Since $\phi^\sharp_S$ is the composition $ \phi^\sharp_{\mathrm{RII}} \circ \phi^\sharp_{h^1_1\sqcup h^1_2} \circ \phi^\sharp_{\sqcup H} $, it is sufficient to see each component has the same property. For the maps $\phi^\sharp_{\mathrm{RII}}$ and $\phi^\sharp_{h^1_1\sqcup h^1_2}$, again the corresponding commutativity has been established in \cite{ISST1}. Therefore, it is sufficient to see that $\phi^\sharp_{\sqcup H}$ has the desired commutative diagram. From its construction, the following diagram is commutative 
        \[
      \begin{CD}
     H_*(\CKh^\sharp(D)) @>{\phi^\sharp_{\sqcup H}}>> H_*(\CKh^\sharp(D\sqcup H))  \\
  @VVV    @VVV \\
     I_*^\sharp (K)   @>{ I^\sharp_{[0,1]\times K \sqcup D_H \subset [0,1]\times S^3 \# \overline{\C P}^2 }}>>  I_*^\sharp(K \sqcup H) 
  \end{CD}. 
    \]
    where the vertical maps are Kronheimer--Mrowka's quasi-isomorphism $I_*^\sharp(K) \cong  H_*(\CKh^\sharp(D))$ and $I ^\sharp_S$ denotes the cobordism map for the blown-uped surface. 

On the other hand, $I^\sharp_{[0,1]\times K \sqcup D_H \subset [0,1]\times S^3 \# \overline{\C P}^2}$ coincides with blown-up cobordism map in instanton Floer theory.  Therefore, from the composition law of instanton cobordism maps, we get the desired diagram. 
\end{proof}

\subsection{Computing $E_2$-term of the immersed cobordism map}

We shall prove: 
\begin{thm}
   For an immersed surface cobordism $S \subset [0,1]\times \R^3$ from $D$ to $D'$ with a decomposition to a sequence of elementary cobordisms, 
    we have the commutative diagram :
    \[
      \begin{CD}
     E_2(\CKh^\sharp(D)) @>{E_2(\phi_{S}^\sharp)}>> E_2(\CKh^\sharp(D')) \\
  @VVV    @VVV \\
      \Kh (D^*)   @>{\Kh^{\low}(S)}>>   \Kh ((D')^*)
  \end{CD},
    \]
    where the vertical arrows $E_2(\CKh^\sharp(D^*)) \to \Kh (D^*)$ are the identifications given in \cite{KM11u}. 
\end{thm}

Again, we take a decomposition of the surface $S = S_1 \circ S_1 \circ \cdots \circ S_N$ into elementary cobordisms. If $S_i$ is either a Reidemeister move or a Morse move, the corresponding commutativity is proven in \cite{ISST1}. Therefore, we only need to treat the case $S$ is a diagrammatic crossing change from $D$ to $D'$. In order to prove it, we need some excision-type argument, described as follows: 
\begin{prop} \label{excision theorem}The induced maps on $E^1$ pages of the diagrams: 
    \[
  \begin{CD}
     E^1(\CKh^\sharp (D))  @>{E^1(\phi^\sharp_{\sqcup H}) }>> E^1(\CKh^\sharp (D \sqcup H )) \\
  @V{E^1(\Psi_1)}VV    @V{E^1(\Psi_2)}VV \\
     E^1( \CKh^\sharp (D) \otimes  \CKh^\sharp (\emptyset ))  @>{E^1( \CKh^\sharp([0,1]\times D) \otimes  \phi^\sharp_{\sqcup H}) }>>   E^1(\CKh^\sharp (D) \otimes  \CKh^\sharp (H))
  \end{CD}
\]
are commutative up to homotopy, where the vertical maps are the excision maps defined for instanton cube complexes with the grading shift of $\Psi_i$ being $\geq (0,0)$ for $i=1,2$.
For the definitions of the $E^1$-term tensor product $E^1(\CKh^\sharp (D) \otimes  \CKh^\sharp (D'))$, see \cite{ISST1}. 

Moreover, the induced maps $E^1(\Psi_1)$ and $E^1(\Psi_2)$ coincide respectively with $\id$ and the disjoint isomorphism in Khovanov theory mentioned in \cite[Proposition 2.2]{ISST1}. 

\end{prop}

\begin{rem}
    Note that the same commutativity holds for a general link $H$ admitting a perturbed Chern-Simons functional whose perturbation is small enough, and critical points have only even Floer grading. 
\end{rem}

\begin{proof}[Proof of \Cref{excision theorem}]

We first recall how the vertical cobordism maps
\begin{align*}
&\Psi_1: \CKh^\sharp (D\sqcup \emptyset) \otimes \CKh^\sharp (\emptyset ) \to \CKh^\sharp (D )\otimes \CKh^\sharp (\emptyset ) \\
& \Psi_2: \CKh^\sharp (D\sqcup H) \otimes \CKh^\sharp (\emptyset ) \to \CKh^\sharp (D) \otimes \CKh^\sharp (H)
\end{align*}
are defined for a pseudo diagram $D$.

For arbitrary disjoint link $K_1 \sqcup K_2 \subset \R^3$, we first consider the geometric cobordism called excision cobordism 
\[
(W, S) : (S^3, K_1 \sqcup K_2 \sqcup  H_\omega ) \sqcup (S^3,   H_\omega  )  \to  (S^3, K_1 \sqcup H_\omega )  \sqcup (S^3, K_2 \sqcup  H_\omega ) ,
\]
where $H_\omega$ is again a Hopf link with admissible $SO(3)$-bundle represented by an arc $\omega$ connecting the different components of the Hopf link, see \cite[Section 4]{ISST1}. In \cite[Section 4]{ISST1}, we considered the cobordisms from $(S^3, K_1 \sqcup H_\omega )  \sqcup (S^3, K_2 \sqcup  H_\omega )$ to $(S^3, K_1 \sqcup K_2 \sqcup  H_\omega ) \sqcup (S^3,   H_\omega  )$. However, the non-trivial surgery in the construction of $W$ happens around the Hopf links $H_\omega$. Therefore, we have the freedom to choose concordances in W connecting $K_i$.

For a general pseudo diagram $D_1 \sqcup D_2 \subset \R^3$ and resolutions $u\in \{0, 1\}^{N_1}$, $v\in \{0, 1\}^{N_{2}}$ and $w\in \{0, 1\}^{N_1+N_2}$, we define an excision cobordism between resolved link diagrams:
\[(W, S_{w;uv}): (S^3, (D_1\sqcup D_2)_{w}\sqcup H_{\omega})\sqcup (S^3, H_{\omega}) \to  (S^3, (D_1)_{u}\sqcup H_{\omega})\sqcup (S^3, (D_2)_{v}\sqcup H_{\omega})  . \]
As it is given in \cite{ISST1}, one can construct an orbifold metric $\breve{g}$ on 
\[
\overline{W} =(-\infty, -1] \times ( S^3 \sqcup S^3) \cup W \cup [1, \infty)  \times ( S^3 \sqcup S^3)
\]
with order 2 orbifold singularity: 
\[
\overline{S}_{w;uv} = (-\infty, -1] \times ( (D_1)_u \sqcup H_\omega  \sqcup (D_2)_u \sqcup H_\omega) \cup {S}_{w;uv} \cup [1, \infty)  \times ( (D_1 \sqcup D_2)_w  \sqcup   H_\omega  \sqcup H_\omega ). 
\]
For critical points $\alpha$, $\beta$ and $\gamma$ of perturbed Chern--Simons functionals of $(D_1 \sqcup D_2)_w$, $(D_1)_u$ and $(D_2)_v$ respectively, we consider the moduli space
\[
\breve{M}_{w;uv}(\alpha; \beta, \gamma) \]
which is the instanton moduli space for the orbifold metric $\check{g}$ with asymptotic conditions on ends determined by $(\alpha, \beta, \gamma)$. 
Using these moduli spaces, the excision map is defined as 
\[ 
E^1(\Psi_{w;uv})(\alpha ):=\sum_{\beta \in \mathfrak{C}_{\pi}(D_1)_u, \gamma \in \mathfrak{C}_{\pi}(D_2)_v}\#\breve{M}_{w;uv}(\alpha; \beta, \gamma)_{0}\cdot \beta \otimes \gamma.\]
Note that the higher version of $\Psi_{w;uv}$ can be defined in a similar way to \cite{ISST1}. 
Since $E^1(\Psi_{w;uv})$ can be regarded as the usual cobordism map, we see that $E^1(\Psi_{w;uv})$ is a chain map.  
Now,  we define $\Psi_1$ and $\Psi_2$ as the maps $\Psi_{w;uv}$ for $D_1= D, D_2 = \emptyset$ and $D_1 = D, D_2 =H$.

Next, we shall make a homotopy between $
E^1(\Psi_2) \circ E^1(\phi^\sharp_{\sqcup H})$ and $ E^1 (\CKh^\sharp([0,1]\times D) \otimes \phi^\sharp_{\sqcup H} ) \circ E^1(\Psi_1)$. 
From \cref{E1term}, 
we see 
\[
E^1(\phi^\sharp_{\sqcup H } ) = f_{(1,2)(1,1)} \circ f_{(2,2)(1,2)}\circ C^\sharp_{D_H}.
\]
Therefore, it is sufficient to make homotopies connecting 
\begin{align*}
&E^1\Psi_2 \circ E^1C^\sharp_{D_H} \text{ and } \left(E^1 \CKh^\sharp([0,1]\times D) \otimes C^\sharp_{D_H} \right)\circ E^1\Psi_1\\
&E^1\Psi_2 \circ f_{(2,2)(1,2)} \text{ and } \left( E^1 \CKh^\sharp([0,1]\times D) \otimes f_{(2,2)(1,2)}\right) \circ E^1\Psi_1  \\
&E^1\Psi_2 \circ f_{(1,2)(1,1)} \text{ and } \left( E^1 \CKh^\sharp([0,1]\times D) \otimes f_{(1,2)(1,1)} \right) \circ E^1\Psi_1 . 
\end{align*}

On the other hand, these maps are just cobordism maps, therefore, the commutativity follows from the usual composition law in framed instanton Floer homology. 

Finally, we show that 
the induced maps $E^1(\Psi_1)$ and $E^1(\Psi_2)$ respectively coincide with $\id$ and the disjoint isomorphism in Khovanov theory. 
This follows from the same discussion as \cite[Proposition 5.1]{ISST1}. Also the grading shifts of $\Psi_i$ have also been computed in \cite[Proposition 4.7]{ISST1}.  

This completes the proof. 

\end{proof}

From \cref{excision theorem}, the computation of
\[
E^1 ( \phi^\sharp_{\sqcup H}) :  E^1(\CKh^\sharp (D)) \to E^1(\CKh^\sharp (D\sqcup H )), 
\]
reduces to the computation of 
\[
E^1 ( \phi^\sharp_{\sqcup H}) :  E^1(\CKh^\sharp (\emptyset)) \to E^1(\CKh^\sharp (H ). 
\]

\begin{figure}[tbp]
\hspace{0mm}
\begin{minipage}[]{0.4\hsize}
\hspace{-3mm}
\includegraphics[scale = 0.7]{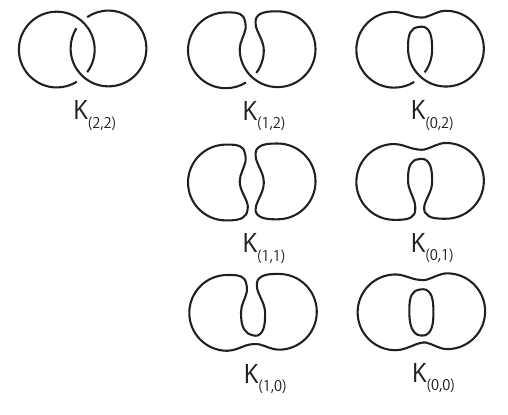}
\caption{\label{fig:hopf-instanton-cube} }
 \end{minipage}
 \begin{minipage}[]{0.4\hsize}
 \vspace{31mm}
\begin{center}
\hspace{-3mm}
\includegraphics[scale = 0.7]{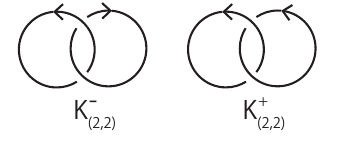}
\caption{\label{fig:hopf-instanton-oriented}}
\end{center}
 \end{minipage}
 \includegraphics[scale = 0.7]{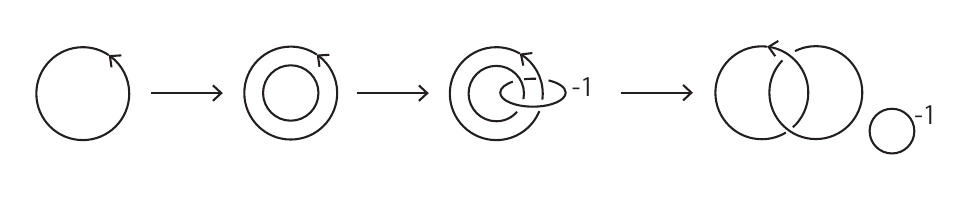}
\caption{\label{fig:cob-cp2} }
\end{figure}

\begin{figure}
    \centering
     \includegraphics[scale = 0.7]{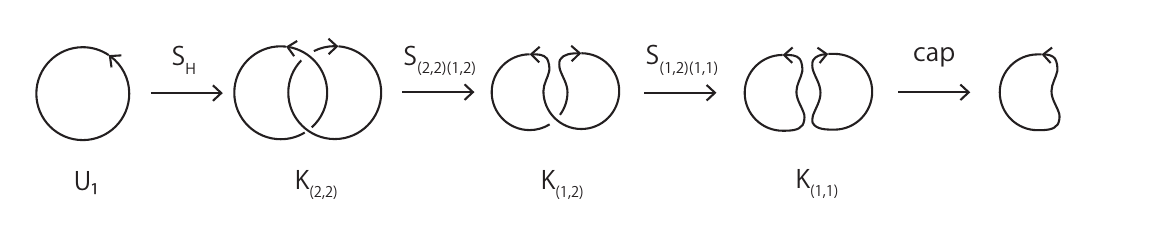}
    \caption{The cobordism $S$ in the proof of \Cref{prop:hopf-compute-h}. Here the orientation of each diagram is induced from $U_1$ so that all orientations are compatible with an orientation of $S$.}
    \label{fig:cob-cp2-2}
\end{figure}

Let 
\[
\iota^\sharp \colon \Z= \CKh^\sharp(U_0 = \emptyset ) \to \CKh^\sharp(U_1)
\]
be the map induced from a trivial disk in $[0,1] \times S^3$, which is a filtered map of order $\geq (0,1)$. 
Let $D^\pm_H$ denote $D_H$ with orientation whose boundary is $K^\pm_{(2,2)}$ in \Cref{fig:hopf-instanton-oriented},
and $\hat{S}^\pm_H$ be the composition of trivial disk 
in $[0,1] \times S^3$ and $D^\pm_H$. 
Note that $\CKh^\sharp(\hat{D}^\pm_H) = \CKh^\sharp(D^\pm_H) \circ \iota^\sharp$.
\begin{prop}
\label{prop:hopf-compute-h}
The following diagram commutes:
\[
\xymatrix@C=80pt{
E^{1}_0(\CKh^\sharp(U_0))
\ar[r]^-{E^1_0(\CKh^\sharp(\hat{S}^\pm_H))}
\ar[d]_-{\cong}
& E^{1}_{-n_+}(\CKh^\sharp(K^\pm_{(2,2)}))
\ar[d]^-{\cong}\\
\Z
\ar[r]^-{1 \mapsto \xi^\mp}
& \CKh^{-n_+,*}(H^{\mp})
}
\]
Here $\xi^{\mp}$ in the bottom horizontal is cycles defined in \Cref{subsec:x-ch-map}.
\end{prop}

To prove \Cref{prop:hopf-compute-h},
we use the following lemmas.

\begin{prop}
\label{prop:unlink-properties}
The following properties hold regarding the maps in framed instanton Floer homology:

\begin{itemize}[leftmargin=2em]
  \item[\textnormal{(1)}] \textbf{Negative twist invariance} \textup{(\cite[Proposition 5.2]{KM11})}:
  If \( S' \) is obtained from an immersed link cobordism \( S : K \to K' \) by a negative twist move, then
  \[
  I^\sharp(S') = I^\sharp(S)
  \]

  \item[\textnormal{(2)}] \textbf{Unlink identification} \textup{(\cite[Corollary 8.5]{KM11u})}:
  Let \( V = \langle \bv_+, \bv_- \rangle \cong \mathbb{Z}^2 \) be a \( \mathbb{Z}/4 \)-graded group with \( \gr(\bv_\pm) = \pm 1 \).
  Then there are isomorphisms of \( \mathbb{Z}/4 \)-graded abelian groups
  \[
  \Phi_n \colon V^{\otimes n} \to I^\sharp(U_n)
  \]
  for all \( n \), satisfying
  \[
  (\iota^\sharp)_*^{\otimes n}(\bu_0) = \Phi_n(\bv_+ \otimes \cdots \otimes \bv_+),
  \]
  where \( \bu_0 \) is the chosen generator for \( I^\sharp(U_0) \cong \mathbb{Z} \).

  \item[\textnormal{(3)}] \textbf{Concordance identity} \textup{(\cite[Lemma 8.9]{KM11u})}:
  Let \( S \) be an oriented concordance from the standard unlink \( U_n \) to itself, consisting of \( n \) oriented annuli in \( [0,1] \times \mathbb{R}^3 \), and preserving the order of the components of \( U_n \). Then the induced map
  \[
  I^\sharp(S) \colon I^\sharp(U_n) \to I^\sharp(U_n)
  \]
  is the identity.
\end{itemize}
\end{prop}
\begin{proof}[Proof of \Cref{prop:hopf-compute-h}]
Note that $hE^1_{-n_+}(\CKh^\sharp(K^\pm_{(2,2)}))=I^\sharp(K_{(1,1)})$, and hence
we have
\begin{align*}
&hE^1_0(\CKh^\sharp(S^{\pm}_H)) \\[2mm]
&= (f_{(1,2)(1,1)})_* \circ (f_{(2,2)(1,2)})_* \circ I^\sharp(D_H) \colon I^\sharp(U_1) 
\to I^\sharp(K_{(1,1)}), 
\end{align*}
where 
\begin{align*}
\begin{cases}
&I^\sharp(D_H) :  I^\sharp(U_1)  \to I^\sharp(K_{(2,2)}) \\ 
&(f_{(2,2)(1,2)})_* : I^\sharp(K_{(2,2)}) \to I^\sharp(K_{(1,2)}) \\ 
& (f_{(1,2)(1,1)})_*: I^\sharp(K_{(1,2)}) \to I^\sharp(K_{(1,1)})
\end{cases}
\end{align*}
are the corresponding induced cobordism maps on the $E^1$-page from \cref{E1term}.  
Moreover, both $(f_{(1,2)(1,1)})_*$ and $(f_{(2,2)(1,2)})_*$
can be regarded as cobordism maps induced from 
$S_{(2,2)(1,2)}$ and $S_{(1,2)(1,1)}$ respectively.
Therefore, to prove the proposition, it suffices to compute the image of $\bv_+$ (in \Cref{prop:unlink-properties}(2)) by 
\[
I^\sharp (S_{(1,2)(1,1)} \circ S_{(2,2)(1,2)}\circ D_H) = (f_{(1,2)(1,1)})_* \circ (f_{(2,2)(1,2)})_* \circ I^\sharp(D_H)  ,
\]
which is denoted by $z$ in the following.

Next, note that $hE^1(\CKh^\sharp(S^{\pm}_H))$ is a chain map with codomain  
$hE^1(\CKh^\sharp(K^{\pm}_{(2,2)})) \cong \CKh(H^{\mp})$,
and hence the element $z \in I^\sharp(K_{(1,1)})$ is isomorphically mapped to a cycle of $\CKh(H^{\mp})$ with h-grading $-n_+$.
This implies that 
\[
z= a (\bv_+ \otimes \bv_- -\bv_- \otimes \bv_+) + b (\bv_- \otimes \bv_-) \in I^\sharp(K_{(1,1)})
\]
for some $a,b \in \Z$, where we order the components of $K_{(1,1)}$ so that the left component is the first.
Moreover, if we consider the cobordism 
\[
S_{\varepsilon} = [0,1] \times U_1 \sqcup D_0 : K_{(1,1)} \to U_1 \subset [0,1]\times \R^3
\]
corresponding to the rightmost arrow in \Cref{fig:cob-cp2-2},
then the image of $z$ by $
I^\sharp(S_{\varepsilon})$ is 
\[
a \bv_+ + b \bv_- = I^\sharp(S_{\varepsilon}) (z).
\]
Thus, to prove the proposition, it suffices to compute the image of $\bv_+$ by the map induced from the cobordism
\[
S := S_{\varepsilon} \circ S_{(1,2)(1,1)} \circ S_{(2,2)(1,2)}\circ D_H,
\]
shown in \Cref{fig:cob-cp2-2}. Here the orientation of each diagram in \Cref{fig:cob-cp2-2} is induced from the leftmost diagram $U_1$ so that all orientations are compatible with an orientation of $S$. 
On the other hand, the cobordism $S$ is diffeomorphic to $[0,1]\times S^1$ and the geometric and algebraic intersection numbers of $S$ with $\C P^1$ are 2 and 0 respectively.
Note that $S$ is a blow-up of an immersed surface $S_*$ and $S_*$ is a negative or positive twist move of the product cobordism. 
Now, it follows from \Cref{prop:unlink-properties}(1) and
\Cref{prop:unlink-properties}(3) that \[
I^\sharp(S)(\bv_+) = \bv_+,
\]
which shows $(a,b)=(1,0)$ and completes the proof.
\end{proof}

\subsubsection{Proof of \Cref{thm:induced-map}}

Since \Cref{thm:induced-map} for embedded cobordisms is already proved, we only need to prove that the following diagram commutes:

\begin{align} \label{eq:diag-h}
\begin{split}
\xymatrix@C=120pt{
E^{1}_p(\CKh^\sharp(D))
\ar[r]^-{E^1_p\left(\CKh^\sharp(([0,1] \times D )\sqcup_i \hat{S}^+_H \sqcup_i \hat{S}^-_H)\right)}
\ar[d]_-{\cong}
& E^{1}_{p-2s_+}(\CKh^\sharp(D \sqcup_i (K_{(2,2)}^+) \sqcup_i (K_{(2,2)}^-)))
\ar[d]^-{\cong}\\
\CKh^{p,*}(D^*)
\ar[r]^-{1 \otimes_i \xi^- \otimes_i \xi^+}
& \CKh^{p,*}(D^*) \otimes_i \CKh^{-2,*}(H^-_i) 
\otimes_i \CKh^{0,*}(H^+_i) 
}
\end{split}
\end{align}

Applying \Cref{excision theorem} and \Cref{lem:tensor-spectral-map}, we have identifications
\begin{align*}
&E^1_p\left(\CKh^\sharp(([0,1] \times D )\sqcup_i \hat{S}^+_H \sqcup_i \hat{S}^-_H)\right)\\[2mm]
&= 
E^1_p\left(\CKh^\sharp(([0,1] \times D )) \otimes_i \CKh^\sharp(\hat{S}^+_H) \otimes_i \CKh^\sharp(\hat{S}^-_H)\right)\\[2mm]
&=
E^1_p \left(\CKh^\sharp([0,1]\times D) \right)\otimes_i 
E^1_0 \left(\CKh^\sharp(\hat{S}^+_H) \right) \otimes_i 
E^1_0 \left(\CKh^\sharp(\hat{S}^-_H) \right).
\end{align*}
Now, the commutativity of the diagram
(\ref{eq:diag-h}) follows from \Cref{prop:hopf-compute-h}.
\qed

\begin{rem}
    In \cite[Theorem 1.1]{ISST1}, we have proven the corresponding statements for filtrations on quantum gradings, i.e. 
    the $E^1$-term with respect to the quantum filtration of $\phi^\sharp_S$ induces the corresponding map in Khovanov theory.  
    We believe this holds as well without essential change. 
\end{rem}
\section{Computation using equivariant instanton theory}
\label{sec:computation}

Note that the cube complexes are written in terms of framed instanton theory. On the other hand, as we mentioned in the introduction, we shall use equivariant instanton Floer theory \cite{DS19} to compute cobordism maps. 

\subsection{Review of equivariant instanton theory} 

We first briefly review equivariant instanton theory. 

\subsubsection{Algebra of $\mathcal{S}$-complexes and special cycles}
Firstly, we review the algebraic background described in \cite{DS19, DS20, daemi2022instantons}.

An $\mathcal{S}$-complex over a ring $R$ consists of a triple $(\wt{C}_* , \wt{d}, \chi)$ where $\wt{C}_{*}$ is a finitely generated free graded $R$-module, $\wt{d}$ and $\chi $ are endomorphisms on $\wt{C}_*$ such that
\begin{itemize}
    \item The degree of $\wt{d}$ and $\chi $ is $-1$ and $+1$ respectively,

    \item $\wt{d}^2 =0$, $\chi^2 =0$, and $\chi \wt{d}+\wt{d} \chi=0$,

    \item $\mathrm{Ker}(\chi)/\mathrm{Im}(\chi)\cong R_{(0)}$, where $R_{(0)}$ is a copy of $R$ equipped with degree $0$.
\end{itemize}

An $\mathcal{S}$-complex has a natural decomposition:
\begin{eqnarray}\label{decomp of S-cpx}
\wt{C}=C_{*}\oplus C_{*-1}\oplus R_{(0)}
\end{eqnarray}
where $C=\mathrm{Im}\chi$. 
The differential $\wt{d}$ has the form 

\begin{eqnarray*}\label{diff of S-cpx}
\wt{d}=\begin{pmatrix}
    d&0&0\\
    v&-d&\delta_2 \\
    \delta_{1}&0&0
\end{pmatrix}
\end{eqnarray*}
with respect to the decomposition (\ref{decomp of S-cpx}).

\textit{A morphism of $\mathcal{S}$-complexes (or $\mathcal{S}$-morphism)} is a graded $R$-module map $\wt{\lambda}:\wt{C}\rightarrow\wt{C}'$ which commutes with differentials and the action of $\chi$ on $\wt{C}$ and $\wt{C'}$.
\textit{An $\mathcal{S}$-chain homotopy} $\wt{K}$ between two $\mathcal{S}$-morphism $\wt{\lambda}$ and $\wt{\lambda}'$ is a $R$-module map which satisfies $\wt{\lambda}-\wt{\lambda}'=\wt{K}\wt{d}+\wt{d}'\wt{K}$ and which anti-commutes with the $\chi$-actions.
A morphism of $\mathcal{S}$-complexes has the form 
\begin{eqnarray}\label{S-morphism}
    \wt{\lambda}=
\begin{pmatrix}
\lambda&0&0\\\mu&\lambda&\Delta_2 \\ \Delta_1&0&c_0
\end{pmatrix}
\end{eqnarray}
with respect to the decomposition (\ref{decomp of S-cpx}).
For each integer $j>0$, we define an element
\begin{eqnarray}\label{const_i}
    c_{j}:=\delta_{1}'(v')^{j-1}\Delta_{2}(1)+\Delta_1 v^{j-1} \delta_2 (1) +\sum_{l=0}^{j-2}\delta'_{2}(v')^{l}\mu v^{j-2-l}\delta_{2} (1).
\end{eqnarray}
in the coefficient ring $R$.
For a given integer $i\geq 0$,
 an $\mathcal{S}$-morphism $\wt{\lambda}$ is called a
\textit{ height $i$ morphism} if it has homological degree $2i$ and satisfies $c_j=0$ for $j<i$.
Moreover, we call a height $i$ morphism $\wt{\lambda}$ \textit{strong}
if $c_i$ is invertible in $R$.

Given an $\mathcal{S}$-complex $\wt{C}$, we can form equivariant chain complexes, which admit $R[x]$-module structures.
Equivariant chain complexes realize three flavors of $R[x]$-modules as their homology groups.
Each of the flavors admits two different models of underground chain complexes, called \textit{large} and \textit{small} equivariant complexes.
Moreover, each flavor of large and small equivariant complexes is chain homotopy equivariant through chain homotopy maps compatible with the $R[x]$-actions.

The large equivariant complexes $(\widecheck{\bf{C}}, \widecheck{\bf{d}})$, $(\widehat{\bf{C}}, \widehat{\bf{d}})$, and $(\overline{\bf{C}}, \overline{\bf{d}})$
are defined as follows:
\begin{eqnarray*}
\widecheck{\mathbf{C}}_* &:=& \widetilde{C}_* \otimes_R \left( R[\![x^{-1}, x] / R[x] \right), \quad 
\widecheck{d} = \widetilde{d} \otimes 1 - \chi \otimes x, \\
\widehat{\mathbf{C}}_* &:=& \widetilde{C}_* \otimes_R R[x], \quad 
\widehat{d} = -\widetilde{d} \otimes 1 + \chi \otimes x, \\
\overline{\mathbf{C}}_* &:=& \widetilde{C}_* \otimes_R R[\![x^{-1}, x], \quad 
\overline{d} = -\widetilde{d} \otimes 1 + \chi \otimes x.
\end{eqnarray*}
Large equivariant complexes admit the following exact triangle sequence on the homology level:

\[
\begin{tikzcd}
H_{*}(\widecheck{\mathbf{C}} , \widecheck{\bf{d}}) \arrow[rr, "\bf{j}_{*}"] 
& &\arrow[dl, "\bf{i}_{*}"] H_{*}(\widehat{\bf{C}}, \widehat{\bf{d}}) \\
& H_{*}(\overline{\bf{C}}, \overline{\bf{d}}) \arrow[ul, "\bf{p}_{*}"']
\end{tikzcd}
\]

\noindent
where the maps $\bf{i}_{*}$, $\bf{j}_{*}$, and $\bf{p}_{*}$ are defined as the natural way on the chain level.
If an $\mathcal{S}$-morphism $\wt{\lambda}:\wt{C}\rightarrow \wt{C}'$
 is given, we have induced $R[x]$-module homomorphisms 
 \[\widecheck{{\bf{\lambda}}}:\widecheck{\bf{C}}\rightarrow \widecheck{\bf{C}}' ,\quad
 \widehat{\bf{\lambda}}:\widehat{\bf{C}}\rightarrow \widehat{\bf{C}}', \quad \overline{{\bf{\lambda}}}:\overline{\bf{C}}\rightarrow \overline{\bf{C}}'\]
which are defined by taking tensor product of the original map $\wt{\lambda}$ and the identities.

Small equivariant complexes $(  \widecheck{\mathfrak{C}}_*, \widecheck{\mathfrak{d}})$, $(\widehat{\mathfrak{C}}_*,  \widehat{\mathfrak{d}})$, and $(\overline{\mathfrak{C}}_{*}, \overline{\mathfrak{d}})$ are defined as follows:



\begin{eqnarray*}
    \widecheck{\mathfrak{C}}_* := {C}_* \oplus R[\![x^{-1}, x]/R[x],&\widecheck{\mathfrak{d}}(\zeta, \sum_{i=-\infty}^{-1} a_i x^i ) := ( {d}\zeta - \sum_{i=0}^{N} v^i \delta_2 (a_i), 0 ),\\
     \widehat{\mathfrak{C}}_*:=C_{*}\oplus R[x],&\widehat{\mathfrak{d}} (\zeta, \sum_{i=0}^{N} a_i x^i ) := ({d}\zeta, \sum_{i= -\infty}^{-1} \delta_1 v^{-i-1} (\zeta) x^i ),\\
     \overline{\mathfrak{C}}_{*}:=R[\![x^{-1}, x],& \overline{\mathfrak{d}} := 0
\end{eqnarray*}


\noindent
The chain complexes $\widecheck{\mathfrak{C}}$ and $\widehat{\mathfrak{C}}_*$
has the $R[x]$-module structures which are given by

\[
x \cdot ( \alpha, \sum_{i=-\infty}^{-1} a_i x^i ) := 
( v(\alpha) + \delta_2(a_{-1}), \sum_{i=-\infty}^{-2} a_i x^{i+1} ), \quad
x \cdot ( \alpha, \sum_{i=0}^{N} a_i x^i ) := 
( v(\alpha), \delta_1(\alpha) + \sum_{i=0}^{N} a_i x^{i+1} ).
\]

\noindent
The chain complex $\overline{\mathfrak{C}}_{*}=R[\![x^{-1}, x]$ has the obvious $R[x]$-module structure.


Small equivariant complexes admit the following exact triangle at the level of homology.

\[
\begin{tikzcd}
H_{*}(\widecheck{\mathfrak{C}}, \widecheck{\mathfrak{d}}) \arrow[rr, "\mathfrak{j}_{*}"] 
& &\arrow[dl, "\mathfrak{i}_{*}"] H_{*}(\widehat{\mathfrak{C}},     \widehat{\mathfrak{d}}) \\
& H_{*}(\overline{\mathfrak{C}}, \overline{\mathfrak{d}}) \arrow[ul, "\mathfrak{p}_{*}"']
\end{tikzcd}
\]
Note that the homology group $H_{*}(\overline{\mathfrak{C}}, \overline{\mathfrak{d}})$ is canonically isomorphic to $R[\![x^{-1}, x]$.
On the chain level, the maps $\mathfrak{i}_{*}$, $\mathfrak{j}_{*}$, and $\mathfrak{p}_{*}$ are defined by

\[\mathfrak{i}(\alpha, \sum_{i=0}^{N}a_{i}x^{i}):=\sum_{i=-\infty}^{-1}\delta_{1}v^{-i-1}(\alpha)x^{i}+\sum_{i=0}^{N}a_{i}x^{i},\ \ \  \mathfrak{j}(\alpha, \sum_{i=-\infty}^{-1}a_{i}x^{i}):=(-\alpha, 0)\]

\[\mathfrak{p}(\sum_{i=-\infty}^{N}a_{i}x^{i}):=(\sum_{i=0}^{N}v^{i}\delta_{2}(a_{i}), \sum_{i=-\infty}^{-1}a_{i}x^{i}).\]

The large and small equivariant complexes give isomorphic homology groups of each flavor.
In particular, there is an $R[x]$-equivariant chain homotopy equivalence in each flavor of equivariant complexes (\cite[Lemma 4.11]{DS19}).
The $R$-module homomorphisms
\[
\widehat{\Psi} : \widehat{\mathfrak{C}} \rightarrow \widehat{\mathbf{C}} 
\quad\quad
\widehat{\Phi} : \widehat{\mathbf{C}} \rightarrow \widehat{\mathfrak{C}}
\]
that give chain homotopy equivalences between the hat flavor complexes are given by

\[\widehat{\Psi}(\alpha, \sum_{i=0}^{N}a_{i}x^{i}):=(\sum_{i=1}^{N}\sum_{j=0}^{i-1}v^{i}\delta_{2}(a_{i})x^{i-j-1}, \alpha, \sum_{i=0}^{N}a_{i}x^{i}).\]
\[\widehat{\Phi}(\sum_{i=0}^N \alpha_i x^i, \sum_{i=0}^N \beta_i x^i, \sum_{i=0}^N a_i x^i):=(\sum_{i=0}^N v^i(\beta_{i})x^i, \sum_{i=0}^N a_{i}x^i+\sum_{i=1}^{N } \sum_{j=0}^{-N}\delta_{1}v^{j}(\beta_{i})x^{i-j-1} ).\]

The reader can find the definition of other homomorphisms in \cite[Section 4.3]{DS19}.
Based on the above setup, we recall the definition of special cycles.
\begin{defn}(\cite[Definition 3.1]{daemi2022instantons})
Let $k\in \mathbb{Z}$ and $f \in R$.
    A chain $z\in \widehat{\bf{C}}$ is called a special $(k, f)$-cycle if there exists $\mathfrak{z}\in \widehat{\mathfrak{C}}$ such that $\widehat{\Psi}(\mathfrak{z})=z$ and $\mathfrak{i}(\mathfrak{z})=fx^{-k}+\sum_{i=-\infty}^{-k-1}b_{i}x^i$.
\end{defn}

The behavior of special cycles under the induced map from an $\mathcal{S}$ -morphism is summarized as follows:
\begin{prop}(\cite[Lemma 3.3]{daemi2022instantons})
    Let $\wt{\lambda}:\wt{C}\rightarrow \wt{C}'$ be a height $i$ morphism. 
    Then, for a special $(k, f)$-cycle $z$ in $\widehat{\bf{C}}$, the chain $\widehat{\Psi'}\circ \widehat{\Phi}'\circ \widehat{\bf{\lambda}}(z)$
 is a special $(k+i, c_{i}f)$-cycle in $\widehat{\bf{C}}'$.
 Here, $c_{i}$ is a constant defined as (\ref{const_i}).
\end{prop}

\subsubsection{$SU(2)$-singular instanton homology groups for based knots}


We briefly review how to construct an $\mathcal{S}$-complex from $SU(2)$-instanton gauge theory.
For a knot $K\subset Y$, we fix a $\mathbb{Z}/2$-orbifold structure on $Y$ that is singular along $K$. 
Then we have the space $\mathcal{A}$ of singular $SU(2)$-connections associated to the $\mathbb{Z}/2$-orbifold $\widecheck{Y}$ along with the Chern-Simons functional $\mathrm{CS}: \mathcal{A}\rightarrow\mathbb{R}$.
We write $\mathcal{B}$ for the quotient space of $\mathcal{A}$ by the action of the group of gauge transformations.
The Chern-Simons functional descends to the quotient space $\mathcal{B}$ as an $S^1=\mathbb{R}/\mathbb{Z}$-valued functional. 
The set $\mathfrak{C}$ of the critical points of $\mathrm{CS} :\mathcal{B}\rightarrow S^1$ is identified with the traceless $SU(2)$-character variety:
\[\chi(Y, K):=\{\rho \in \mathrm{Hom}(\pi_{1}(Y\!\setminus\! K ), SU(2))\vert \mathrm{tr}\rho(\mu)=0\}/SU(2).\]
The space contains a unique reducible representation. 
We write $\theta\in\mathcal{B}$ as the corresponding flat reducible element.
After choosing a suitable perturbation $\pi$ of the Chern-Simons functional, the set of critical points of the perturbed Chern-Simons functional $\mathrm{CS}_{\pi}$
can be written as $\mathfrak{C}_{\pi}=\mathfrak{C}^{*}_{\pi}\sqcup \{\theta\}$, where $\mathfrak{C}^{*}_{\pi}$ is the irreducible part.
Let $R$ be an algebra over the ring $\Z[T^{\pm1}]$.
We define an irreducible $SU(2)$-singular instanton Floer chain complex
\[(C_{*}(K;\Delta_{R}), d)\]
over the local coefficient system $\Delta_{R}$.
The underlying chain complex $C_{*}(K;\Delta_{R})$ is
a $R$-module finitely generated by the elements in $\mathfrak{C}^{*}_{\pi}$, equipped with an $\Z/4$-grading.
Roughly speaking, the differential $d$ is an $R$-module endomorphism of degree $-1$, which is defined by counting instantons over the cylinder $\mathbb{R}\times \widecheck{Y}$.

Given a based knot $(K, p)$ in an integral homology $3$-sphere $Y$, and an algebra $R$ over the ring $\mathbb{Z}[T^{\pm 1}]$,  we associate a $\mathbb{Z}/4$-graded $\mathcal{S}$-complex with a local coefficient system $\Delta_{R}$.
The underlying chain group is defined as
\[\wt{C}_{*}(Y, K;\Delta_{R}):=C_{*}\oplus C_{*-1}\oplus R_{(0)}\]
by putting $C_{*}=C_{*}(Y, K;\Delta_{R})$.
The differential $\wt{d}$ is essentially defined by counting instantons over the cylinder $\mathbb{R}\times \widecheck{Y}$ equipped with a path $\mathbb{R}\times \{p\}$.
Note that the $\mathcal{S}$-chain homotopy type of an $\mathcal{S}$-complex does not depend on the choice of the base point $p\in K$.

Next, we review the functorial property of $\mathcal{S}$-complexes.
For simplicity, we will only consider the case $Y=S^3$, and cobordisms in a cylinder $[0, 1]\times S^3$.
\begin{defn}
    An immersed surface cobordism $S: K\rightarrow K'$ with $s_{+}$ positive double points is called height $i$ if item (i) holds.
    If both items (i) and (ii) hold, we call $S$ strong height $i$.
    \begin{itemize}
        \item[(i)] $i=-g(S)+\frac{1}{2}\sigma(K)-\frac{1}{2}\sigma(K')$,
        \item[(ii)] $(T^2-T^{-2})^{s_{+}}\in R$ is invertible.
    \end{itemize}
\end{defn}

\begin{prop}
Let $(K, p)$ and $(K', p')$ be based knots and $S:K\rightarrow K'$ be a (strong) height $i$ cobordism.
We fix a smooth path $\gamma$ connecting $p$ and $p'$ on $S$ away from double points of $S$.
   If $i\geq 0$, then the pair $(S, \gamma)$ induces a (strong) height $i$
   morphism 
   $\wt{\lambda}_{(S, \gamma)}:\wt{C}_{*}(K; \Delta_{R})\rightarrow \wt{C}_{*+2i}(K':\Delta_{R})$
   of $\mathcal{S}$-complexes.
  Moreover, for a different choice of the paths $\gamma$ and $\gamma'$, there is an $\mathcal{S}$-chain homotopy between two induced maps $\wt{\lambda}_{(S, \gamma)}$ and $\wt{\lambda}_{(S, \gamma')}$. 
   
\end{prop}
\begin{proof}
    For a given height $i$ immersed cobordism $S:K\rightarrow K'$, we consider the embedded cobordism $\wt{S}$ in $([0, 1]\times S^3)\#_{s}\overline{\mathbb{CP}}^2$, which is obtained by blowing up the double points.
    Condition (i) implies that the cobordism $\wt{S}$ is a negative definite cobordism of height $i$ in the sense of \cite[Definition 4.16]{DS20} (see also \cite[Definition 2.7]{daemi2022instantons}), hence \cite[Proposition 4.17]{DS20} implies that it induces a height $i$ morphism of $\mathcal{S}$-complexes.
     As described in \cite[Section 2]{DS20}, an $\mathcal{S}$-morphism associated to $S$ is defined by
     $\wt{\lambda}_{(S, \gamma)}:=(-T^{2})^{s_{+}}\wt{\lambda}_{(\wt{S}, \wt{\gamma})}$.
     On the other hand, the $c_{0}$ component of $\wt{\lambda}_{\wt{S}}$ is $(1-T^{4})^{s_{+}}$ (see the proof of \cite[Proposition 2.30]{DS20}).
     Hence, condition (ii) implies that the induced map $\wt{\lambda}_{S}$ has the invertible component $c_{0}=(T^2-T^{-2})^{s_{+}}$, which means that the height $i$ $\mathcal{S}$-morphism $\wt{\lambda}_{(S, \gamma)}$ is strong.
     The last statement follows from the fact that 
     there is an $\mathcal{S}$-chain homotopy between $\wt{\lambda}_{(\wt{S}, \wt{\gamma})}$ and $\wt{\lambda}_{(\wt{S}, \wt{\gamma}')}$.
\end{proof}
For any choice of the path, we simply write $\wt{\lambda}_{S}$ for $\wt{\lambda}_{(S, \gamma)}$.

Let us see an $\mathcal{S}$-morphism in a simplified situation.
Suppose that $U_1$ is the unknot in $S^3$, $K$ is a knot in 
$S^3$, and $S\colon U_1 \to K$ is an oriented immersed cobordism with $s_{\pm}$ many $\pm$-double points and genus $g(S)$.
Then we have the induced cobordism map
\[
\widetilde{\lambda}_S
=
\begin{pmatrix}
0 & 0 & 0\\
0 & 0 & \Delta_2\\
0 & 0 & c_0
\end{pmatrix}
\colon
\Z[T^{\pm 1}] = 
\widetilde{C}_*(U_1;\Delta_{\Z[T^{\pm 1}]}) \to 
\widetilde{C}_{*+2i}(K;\Delta_{\Z[T^{\pm 1}]}),
\]
where $i = -g(S) -\sigma(K)/2$.
For any $j>0$, set 
\[
c_j := \delta_1 v^{j-1} \Delta_2(1).
\]
Then we have

\begin{align}
\label{eq:morphism}
c_j &= 
\begin{cases}
(T^2-T^{-2})^{s_+} & (j=i)\\[1mm]
0 & (0 \leq j < i)    .
\end{cases}
\end{align}

\subsubsection{Relation to $I^\sharp(K;\Z)$}
We review the relation between $\wt{C}$-complexes and the $\Z/4$-graded knot homology group $I^{\natural}_{*}(K)$.
A triple $(Y, K, \omega)$ of an oriented $3$-manifold $Y$, a knot $K$, and an embedded $1$-manifold $\omega$ in $Y\!\setminus\! K$ such that $\partial \omega=\omega\cap K$ and $\omega$ meets $K$ transversely is called \textit{admissible} if there exists a closed oriented surface $\Sigma$ in $Y$ satisfying either one of the following:
\begin{itemize}
    \item $\omega \cap K=\emptyset $ and $\sharp (\omega\cap \Sigma)$ is an odd number,
    \item $\Sigma$ transversely intersects to $K$  an odd number of times.
\end{itemize}

For a given admissible triple $(Y, K, \omega)$, we have an mapping cone complex:
\[\wt{C}^{\omega}(Y, K):=\mathrm{Cone}(C^{\omega}\xrightarrow{v}C^{\omega}[1])\]
where $C^{\omega}=C^{\omega}(Y, K)$ is a $\Z/4$-graded singular instanton knot Floer complex defined in \cite{KM11}.

Let $(W, S): (Y, K)\rightarrow (Y', K')$ be a pair of cobordism where $W$ is oriented and $S$ is possibly immersed.
Moreover, we assumed that there is an embedding of cornered 2-manifold $\bf{\omega}$ such that
\begin{itemize}
    \item $\partial \mathbf{\omega}=\omega_{0}\cup \omega_{1} \cup \text{arcs}\ \text{on}\ S$, and $\omega$ is normal to $S$ near the boundary components.
    \item The intersection $\omega\cap S$ allows finitely many points, where the intersection is transverse.
\end{itemize}
Assume that $(Y, K, \omega_{0})$ and $(Y', K', \omega_{1})$ are admissible.
Then we have a chain map
\[\wt{\lambda}^{\omega}_{S}:\wt{C}^{\omega}(Y, K)\rightarrow \wt{C}^{\omega}(Y', K')\]
whose components are given by 
\[
\wt{\lambda}^{\omega}_{S}:=
\begin{bmatrix}
\lambda^{\omega} & 0\\
\mu^{\omega} & \lambda^{\omega}
\end{bmatrix}.
\]
Here $\lambda^\omega_{S}$ is the ordinary cobordism map defined by counting the $0$-dimensional part of the instanton moduli space over the cobordism.
The component $\mu^{\omega}$ is defined by a similar way as the component of $\mu$ in (\ref{S-morphism}), that is given by counting the $2$-dimensional part of the moduli space after transversely cut out by the codimension-2 divisor.

For a Hopf link $H$ in $S^3$ decorated by an arc $\omega$ connecting two components of $H$, the pair $(H, \omega)$ is admissible.
Note that for any pair $(Y, K)$ of an integer homology $3$-sphere $Y$ and a knot $K$, any triple $(Y, K\sqcup H, \omega)$ is admissible.
For a knot $K$ in $Y$, we write 
\[
C^{\sharp}(Y, K):=C^{\omega}(Y, K\sqcup H).
\]

In \cite{KM11u}, Kronheimer and Mrowka defined $I^{\sharp}$-homology group:
\[I^{\sharp}(Y, K):=H_{*}({C}^{\sharp}(Y, K))\]
for any pair $(Y, K)$, which is supposed to be an instanton counterpart of  Khovanov homology.

On the other hand, the $\mathcal{S}$-complex for $(Y, K)$ recovers the $I^{\sharp}$-homology. 
Note that $\widetilde{C}_*(K;\Delta_{\Z}) = \widetilde{C}_*(K;\Delta_{\Z[T^{\pm 1}]}) \otimes_{\Z[T^{\pm 1}]} \Z$, where the action of $\Z[T^{\pm}]$ on $\Z$ is defined by $f(T) \cdot n := f(1) n$ ($f(T) \in \Z[T^{\pm 1}], n \in \Z$).
The reduced framed instanton Floer homology is recovered from equivariant instanton Floer homology in the following way: 
We define a chain complex $\wt{C}^\sharp_*(K;\Delta_{\Z})$ as the quotient complex $\widehat{\mathbf{C}}(K;\Delta_{\Z})/(x^2)$.
In other words,
\[
\wt{C}^\sharp_*(K;\Delta_{\Z}) :=
\wt{C}_*(K;\Delta_{\Z}) \oplus \wt{C}_{*+2}(K;\Delta_{\Z}),
\quad \wt{d}^\sharp :=
\begin{bmatrix}
\wt{d} & 0\\
2\chi & \wt{d}
\end{bmatrix}
\]
after the identification $x\cdot \wt{C}_{*}(K;\Delta_{\Z})=\wt{C}_{*+2}(K;\Delta_{\Z})$.

\begin{thm}\cite[Theorem 8.13]{DS19}\label{recov I sharp}
There is a canonical isomorphism    
\[I^{\sharp}(K;\Z)\cong H_{*}(\wt{C}^{\sharp}(K;\Delta_{\Z}))\]
as $\Z/4$-graded abelian groups.
\end{thm}


Let $S$ be a possibly immersed surface cobordism from $K$ to $K'$.
We choose a smooth embedding of the pair $[0, 1]\times (H, \omega)$ on the surface complement, which extends the Hopf link decoration on each boundary.
After choosing auxiliary data, the cobordism $S$ induces a map
\[\lambda^{\sharp}_{S}:C^{\sharp}(K)\rightarrow C^{\sharp}(K').\]

 See \cite[Section 4.3]{KM11u} for the detail. We write the map induced on the homology level as \[I^{\sharp}_{S}: I^{\sharp}(K)\rightarrow I^{\sharp}(K').\]
On the other hand, if $S$ is an immersed cobordism of height $i\geq 0$, then there is a cobordism map
\[
\wt{\lambda}^\sharp_S \colon
\wt{C}^\sharp(K;\Delta_\Z) \to 
\wt{C}^\sharp(K';\Delta_\Z), \quad
\wt{\lambda}^\sharp_S :=
\begin{bmatrix}
\wt{\lambda} & 0 \\
0& \wt{\lambda}
\end{bmatrix}
\]
where the component $\lambda$ is the same map as in the corresponding component of (\ref{S-morphism}).
\begin{prop}
    For any oriented immersed cobordism $S\colon K \to K'$, the following diagram is commutative:
\begin{align}
\label{eq:equiv-diag-sharp}
\begin{split}
\xymatrix@C=80pt{
H_*(\widetilde{C}^\sharp_*(K;\Delta_{\Z}))
\ar[r]^-{(\widetilde{\lambda}^\sharp_S)_*}
\ar[d]_-{\cong}
& H_*(\widetilde{C}^\sharp_*(K';\Delta_{\Z}))
\ar[d]^-{\cong}\\
I^\sharp(K;\Z)
\ar[r]^-{I^\sharp_S}
& I^\sharp(K';\Z)
}
\end{split}
\end{align}
\end{prop}

\begin{proof}
    Firstly, we consider the natural cobordism map
    \[\wt{\lambda}^{\omega}: \wt{C}_{*}^{\omega}(K\sqcup H)\rightarrow \wt{C}_{*}^{\omega}(K'\sqcup H)\]
    induced from the cobordism associated to the pair  $(S\sqcup [0, 1]\times H, [0, 1]\times \omega)$.
    Since the cobordism decomposes into the boundary sum
    \[([0, 1]\times S^3, S)\natural ([0, 1]\times S^3, [0, 1]\times (U_{1}\sqcup H_{\omega})),\]
     the naturality of the connected sum for $\wt{C}$-complexes \cite[Section 6.3.4]{DS19} implies that we have a commutative diagram
     
\begin{align}\label{commdig}
     \begin{split}
\xymatrix@C=80pt{
\widetilde{C}_*(K;\Delta_{\Z})\otimes\wt{C}^{\omega}(U_{1}\sqcup H;\Delta_{\Z})
\ar[r]^-{\wt{\lambda}_{S}\otimes \wt{\lambda}^{\omega}_{[0, 1]\times H}}
\ar[d]_-{\simeq}
& \widetilde{C}_*(K';\Delta_{\Z})\otimes \wt{C}^{\omega}(U_{1}\sqcup H;\Delta_{\Z})
\ar[d]^-{\simeq}\\
\wt{C}_{*}^{\omega}(K\sqcup H;\Delta_{\Z})
\ar[r]^-{\wt{\lambda}^{\omega}_{S\natural ([0, 1]\times H)}}
& \wt{C}_{*}^{\omega}(K'\sqcup H;\Delta_{\Z})
}
\end{split}
\end{align}

Note that $C^{\omega}_{*}(U_{1}\sqcup H)$ has rank $2$ of two generators $v_{+}$ and $v_{-}$.
Since the isomorphism in Theorem \ref{recov I sharp} is given by the identification with $C^{\sharp}(K)=C^{\omega}(K\sqcup H)\subset \wt{C}^{\omega}(K\sqcup H)$ and \[\wt{C}(K)\otimes \Z v_{+}\oplus \wt{C}(K)\otimes \Z v_{-} \subset \widetilde{C}_*(K;\Delta_{\Z})\otimes\wt{C}^{\sharp}(U_{1})\]
through the vertical maps, the commutative diagram \ref{commdig} restricts to the corresponding subcomplexes.
This induces the commutative diagram in the statement at the level of homology groups.

\end{proof}

\subsection{$\mathcal{S}$-complexes of the torus knots $T_{2,2k+1}$}\label{subsec:S-cpx-T2q}

Here we focus on the $(2,2k+1)$ torus knots $T_{2,2k+1}$ where $k$ is a non-negative integer. (Note that $T_{2,1} = U_1$.) 
It is proved in \cite[Proposition 5]{DS20} that 
$\widetilde{C}_* =
\widetilde{C}_{*} (T_{2,2k+1};\Delta_{\Z[T^{\pm 1}]})$
is given by 
\[
\widetilde{C}_* = C_*(T_{2,2k+1}) \oplus C_{*-1}(T_{2,2k+1}) \oplus \Z[T^{\pm 1}],
\quad C_*(T_{2,2k+1}) = \bigoplus_{i=1}^k \Z[T^{\pm 1}] \cdot \zeta^i
\]
where the $\Z/4$-grading of $\zeta^i$ is $2i-1$,
and the differential $\widetilde{d}$ has the components
$d= \delta_2 = 0$ and
\begin{align*}
\delta_1(\zeta^i) &=
\begin{cases}
T^2 - T^{-2} & (i=1)\\[1mm]
0 & (2 \leq i \leq k),
\end{cases}\\[2mm]
v(\zeta^i) &=
\begin{cases}
(T^2 - T^{-2})\zeta^{i-1} & (2\leq i \leq k)\\[1mm]
0 & (i=1).
\end{cases}
\end{align*}
In particular, for any $1 \leq i \leq k$ and $j \in \Z_{> 0}$, we have
\[
\delta_1 v^{j-1} (\zeta^i)
=
\begin{cases}
(T^2 - T^{-2})^i & (j=i)\\[1mm]
0 & (j \neq i).
\end{cases}
\]

By just putting $T=1$ combined with \cref{recov I sharp}, we can recover the computations of the framed instanton homology. 
Here we consider the $\Z/2\Z$-reduction of the $\Z/4\Z$-grading on $I^\sharp(T_{2,2k+1};\Z)$, and denote the $\Z/2\Z$-grading 0 part (resp. grading 1 part) of $I^\sharp(T_{2,2k+1};\Z)$ by $I^\sharp_{[0]}(T_{2,2k+1};\Z)$ 
(resp.\ $I^\sharp_{[1]}(T_{2,2k+1};\Z)$).

\begin{lem}
\label{lem:T2q-sharp}
For any $k \geq 0$, we have 
$I^\sharp_{[0]}(T_{2,2k+1};\Z) \cong \Z^{k+2} \oplus (\Z/2\Z)^k$ and $I^\sharp_{[1]}(T_{2,2k+1};\Z) \cong \Z^{k}$.
\end{lem}

From the descriptions of the differentials above, we also have the following. 
\begin{lem}
\label{lem:evaluate}
For any element $\zeta = \sum_{i=1}^k a_i \zeta^i \in C_*(T_{2,2k+1}; \Delta_{\Z[T^{\pm 1}]})$,
we have
\[
\delta_1 v^{i-1}(\zeta) = a_i (T^2 - T^{-2})^i.
\]
\end{lem}

\subsection{Cobordism maps from the unknot}

From the structure of $\wt{C}(T_{2,2k+1}; \Delta_{\Z[T^{\pm 1}]})$, we can determine some parts of cobordism maps from the unknot to two-bridge torus knots. 
\begin{lem}
\label{lem:Delta2}
Let $\Delta_2$ be the $(2,3)$-entry of $\widetilde{\lambda}_S$ for an immersed smooth surface cobordism $S \colon U_1 \to T_{2,2k+1}$ with \[
s_+ + g(S)= k.
\]
Then, for some $a_i \in \Z[T^{\pm}]$ $(i=s_+ +1, \ldots, k)$, we have the equality
\[
\Delta_2(1) = \zeta^{s_+} + \sum_{i=s_+ +1}^k a_i \zeta^i
\]
as elements of $C_*(T_{2,2k+1};\Delta_{\Z[T^{\pm 1}]})$, where $\zeta^0 = 0$.
\end{lem}

\begin{proof}
Note that $\{\zeta^i\}_{i=1}^k$ is a free basis for $C_*(T_{2,2k+1};\Delta_{\Z[T^{\pm 1}]})$ over $\Z[T^{\pm 1}]$, and hence
we have $\Delta_2(1) = \sum_{i=1}^k a_i \zeta^i$ for some $a_i \in \Z [T^{\pm 1}]$. Moreover, it follows from the equality \eqref{eq:morphism} and \Cref{lem:evaluate} that for any $1 \leq i \leq s_+$,
\[
a_i(T^2-T^{-2})^i = \delta_2 v^{i-1}\Delta_2(1) = c_i 
=
\begin{cases}
(T^{2}-T^{-2})^{s_+} & (i=s_+)\\[1mm]
0 & (1 \leq i < s_+).
\end{cases}
\]
These give $a_{s_+} = 1$ and $a_i = 0$ for each $1 \leq i < s_+$.
\end{proof}

Next, we consider the $\Z/2\Z$-reduction of the $\Z/4\Z$-grading on $\widetilde{C}_* = \widetilde{C}_*(T_{2,2k+1};\Delta_{\Z[T^{\pm 1}]})$. Denote the $\Z/2\Z$-grading 0 part (resp. grading 1 part) of $\widetilde{C}_*$ by $\widetilde{C}_{[0]}$ 
(resp.\ $\widetilde{C}_{[1]}$ ). Then we see that
\[
\widetilde{C}_{[0]} = 0 \oplus C_{*-1}(T_{2,2k+1};\Delta_{\Z[T^{\pm 1}]}) \oplus \Z[T^{\pm 1}]
\]
and
\[
\widetilde{C}_{[1]} = C_{*}(T_{2,2k+1};\Delta_{\Z[T^{\pm 1}]}) \oplus 0 \oplus 0.
\]
From \Cref{lem:Delta2}, we can give a free basis for $\wt{C}_{[0]}(T_{2,2k+1};\Delta_{\Z[T^{\pm 1}]})$ as follows.
(Here we use the identification
$\wt{C}_*(U_1;\Delta_{\Z[T^{\pm 1}]}) = 0 \oplus 0 \oplus \Z[T^{\pm 1}] = \Z[T^{\pm 1}]$.)
\begin{lem}
\label{lem:family-of-cob}
An arbitrary set of immersed cobordisms 
$\{S_i \colon U_1 \to T_{2,2k+1}\}_{i=0}^k$ with the properties
\[
s_+(S_i)= i  \text{ and } \ g(S_i)= k-i
\]
gives a free basis $\{\widetilde{\lambda}_{S_i}(1)\}_{i=0}^k$ for $\widetilde{C}_{[0]}(T_{2,2k+1};\Delta_{\Z[T^{\pm 1}]})$ over $\Z[T^{\pm 1}]$. Moreover, if a given immersed cobordism $S\colon U_1 \to T_{2,2k+1}$ satisfies
\[
s_+(S)+g(S)=k,
\]
then we have
\[
\widetilde{\lambda}_S(1)-\widetilde{\lambda}_{S_{s_+(S)}}(1) \in 
\Z[T^{\pm 1}] \cdot
\left\langle \widetilde{\lambda}_{S_{s_+(S)+1}}(1), \ldots, \widetilde{\lambda}_{S_{k}}(1) \right\rangle.
\]
\end{lem}
\begin{proof}
By the equality (\ref{eq:morphism})  and \Cref{lem:Delta2},
the element
$\widetilde{\lambda}_{S_i}(1)$ is in the form of
\begin{align}
\label{eq:image-of-cob-map}
\widetilde{\lambda}_{S_i}(1) = 
\begin{cases}
(0,\sum_{j= 1}^k a_j \zeta^j,1) & (i=0) \\[2mm]
\left(0,\zeta^{i} + \sum_{j=i +1}^k a_j \zeta^j, 0 \right) & (1 \leq j \leq k).
\end{cases}
\end{align}
Hence, the first half assertion immediately follows from the fact that 
\[
\{(0,0,1), (0,\zeta^1,0), \ldots (0,\zeta^k,0) \}
\]
is a free basis for 
$\widetilde{C}_{[0]}(T_{2,2k+1};\Delta_{\Z[T^{\pm 1}]})$ over $\Z[T^{\pm 1}]$.
Moreover, the form (\ref{eq:image-of-cob-map}) also shows
the equalities
\[
\widetilde{\lambda}_S(1)-\widetilde{\lambda}_{S_{s_+(S)}}(1)
= 
\begin{cases}
(0,\sum_{i=1}^k a_i \zeta^i,0) & (s_+(S)=0) \\[2mm]
\left(0,(\sum_{i=s_+(S) +1}^k a_i \zeta^i, 0 \right) & (1 \leq s_+(S) \leq k)
\end{cases}
\]
and
\[
\Z[T^{\pm 1}] \cdot
\left\langle 
(0,\zeta^{s_+(S)+1},0), \ldots, (0, \zeta^{k}, 0)
\right\rangle
=
\Z[T^{\pm 1}] \cdot
\left\langle \widetilde{\lambda}_{S_{s_+(S)+1}}(1), \ldots, \widetilde{\lambda}_{S_{k}}(1) \right\rangle.
\]
These imply the second half assertion.
\end{proof}
Now we translate the above computations into the $I^\sharp$-theory. Recall that $I^\sharp(U_1;\Z) = I^\sharp_{[0]}(U_1;\Z) \cong \Z \oplus \Z$ is generated by $u_+,u_-$, which are coresponding to $(1,0), (0,1)$ in $H_*(\wt{C}^\sharp(U_1;\Z))=\wt{C}^\sharp(U_1;\Z)=\Z \oplus \Z$ respectively, via the canonical isomorphism given by \Cref{recov I sharp}.
\begin{cor}
\label{cor:family-of-cob}
An arbitrary set of immersed cobordisms 
$\{S_i \colon U_1 \to T_{2,2k+1}\}_{i=0}^k$ with the property
\[
s_+(S_i)= i  \text{ and } \ g(S_i)= k-i
\]
gives a generating set $\{I^\sharp_{S_i}(u_+), I^\sharp_{S_i}(u_-)\}_{i=0}^k$ of $I^\sharp_{[0]}(T_{2,2k+1};\Z)$. Here,
 the elements $\{I^\sharp_{S_i}(u_-)\}_{i=1}^k$ generate 
$\Tor I^\sharp_{[0]}(T_{2,2k+1};\Z) \cong (\Z/2\Z)^k$, and
the elements $\{I^\sharp_{S_i}(u_+), I^\sharp_{S_0}(u_-)\}_{i=0}^k$ are linearly independent in 
$I^\sharp_{[0]}(T_{2,2k+1};\Z)$ over $\Z$. Moreover, if a given immersed cobordism $S\colon U_1 \to T_{2,2k+1}$ satisfies 
\[
s_+(S)+g(S)=k,
\]
then we have
\[
I^\sharp_S(u_+)-I^\sharp_{S_{s_+(S)}}(u_+) \in 
\Z \cdot
\left\langle I^\sharp_{S_{s_+(S)+1}}(u_+), \ldots, I^\sharp_{S_{k}}(u_+) \right\rangle
\]
and
\[
I^\sharp_S(u_-)-I^\sharp_{S_{s_+(S)}}(u_-) \in 
\Z \cdot
\left\langle I^\sharp_{S_{s_+(S)+1}}(u_-), \ldots, I^\sharp_{S_{k}}(u_-) \right\rangle.
\]
\end{cor}

\begin{proof}
Note that
\begin{align*}
\widetilde{C}^\sharp_{[0]}(T_{2,2k+1};\Delta_{\Z}) &= \left(\widetilde{C}_{[0]}(T_{2,2k+1};\Delta_{\Z[T^{\pm 1}]}) \oplus \wt{C}_{[0]}(T_{2,2k+1};\Delta_{\Z[T^{\pm 1}]})\right)
\otimes_{\Z[T^{\pm 1}]} \Z \\[2mm]
&= (0 \oplus C_{*-1}(T_{2,2k+1};\Delta_{\Z})\oplus \Z)
\oplus (0 \oplus C_{*-1}(T_{2,2k+1};\Delta_{\Z})\oplus \Z)\\[2mm]
&\subset \Ker \wt{d}^\sharp,
\end{align*}
and
\[
\wt{d}^\sharp \left(\wt{C}^\sharp_{[1]} (T_{2,2k+1};\Delta_{\Z}) 
\right) =
(0 \oplus 0 \oplus 0)
\oplus (0 \oplus 2 \cdot C_{*-1}(T_{2,2k+1};\Delta_{\Z})\oplus 0).
\]
Moreover, we have $\wt{\lambda}^\sharp_S(\alpha, \beta) = (\wt{\lambda}_S(\alpha), \wt{\lambda}_S(\beta))$.
Therefore, the assertion directly follows from \Cref{lem:family-of-cob} and 
the commutativity of (\ref{eq:equiv-diag-sharp}).
\end{proof}

\subsection{Cobordism maps from the trefoil}

Next, we consider cobordisms $S \colon T_{2,3} \to T_{2,2k+1}$ for $k>0$ and their induced maps, which will be used for the study of $I^\sharp_{[1]}(T_{2,2k+1};\Z)$.
\begin{lem}
\label{lem:lambda-T23}
Let $\lambda_S$ be the $(2,2)$-entry of $\widetilde{\lambda}_S$ for an immersed smooth surface cobordism $S \colon T_{2,3} \to T_{2,2k+1}$ with \[
s_+ + g(S)= k-1.
\]
Then, for some $a_i \in \Z[T^{\pm}]$ $(i=s_+ +2, \ldots, k)$,
we have the equality
\[
\lambda_S(\zeta^1) = \zeta^{s_++1} + \sum_{i=s_+ + 2}^k a_i \zeta^i
\]
as elements of $C_*(T_{2,2k+1};\Delta_{\Z[T^{\pm 1}]})$.
\end{lem}

\begin{proof}
We can take an immersed smooth surface cobordism 
$S' \colon U_1 \to T_{2,3}$ with $s_+(S')=1$ and $g(S')=0$. Then, the composition 
$S \circ S' \colon U_1 \to T_{2,2k+1}$ satisfies 
\[
s_+(S \circ S') + g(S \circ S')
= s_+ + 1 + g(S) = k.
\]
Furthermore, there exists a $\Z[T^{\pm 1}]$-module map
\[
\wt{h} \colon \wt{C}_*(U_1; \Delta_{\Z[T^{\pm 1}]}) \to
\wt{C}_*(T_{2,2k+1}; \Delta_{\Z[T^{\pm 1}]}),
\quad \wt{h} =
\begin{bmatrix}
h_{11} & 0 & 0\\
h_{21} & -h_{11} & h_{23}\\
h_{31} & 0 & 0
\end{bmatrix}
\]
which increases the $\Z/4$-grading by 1 and satisfies
$\wt{\lambda}_S \circ \wt{\lambda}_{S'} - \wt{\lambda}_{S \circ S'} = \wt{d} \circ \wt{h} + \wt{h} \circ \wt{d}$.
Here we note that 
\[
\wt{d}(0,0,1) = 0 \in  \wt{C}_*(U_1;\Delta_{\Z[T^{\pm 1}]})
\]
and
\[
\wt{h}(0,0,1) \in \wt{C}_{[1]}(T_{2,2k+1};\Delta_{\Z[T^{\pm 1}]})  \cap 
\left(0 \oplus C_{*-1} \oplus 0 \right) 
= (C_* \oplus 0 \oplus 0) \cap (0 \oplus C_{*-1} \oplus 0) = \{0\}.
\]
These show that $(\wt{d} \circ \wt{h} + \wt{h} \circ \wt{d})(0,0,1) = 0$, and hence
$\wt{\lambda}_S \circ \wt{\lambda}_{S'}(0,0,1) = \wt{\lambda}_{S \circ S'}(0,0,1)$.
Thus, by \Cref{lem:Delta2}, we see that
\begin{align*}
(0, \lambda_S(\zeta^1),0) 
&= \wt{\lambda}_{S}(0,\zeta^1,0) 
= \wt{\lambda}_S(0,\Delta_2(1),0)\\[2mm]
&= \wt{\lambda}_S \circ \wt{\lambda}_{S'} (0,0,1)
= \wt{\lambda}_{S \circ S'} (0,0,1)\\[2mm]
&= (0, \Delta'_2(1), 0) = 
\left(0, \zeta^{s_++1} + \sum_{i=s_+ + 2}^k a_i \zeta^i, 0\right),
\end{align*}
where $\Delta_2$ (resp.\ $\Delta'_2$) denotes the $(2,3)$-entry of $\wt{\lambda}_S'$ (resp.\ $\wt{\lambda}_{S \circ S'}$).
\end{proof}
\Cref{lem:lambda-T23} provides an analogy of \Cref{lem:family-of-cob} as follows:
\begin{lem}
\label{lem:family-of-cob2}
An arbitrary set of immersed cobordisms 
$\{S_i \colon T_{2,3} \to T_{2,2k+1}\}_{i=1}^{k}$ with the properties
\[
s_+(S_i)= i-1  \text{ and } \ g(S_i)= k-i
\]
gives a free basis $\{\widetilde{\lambda}_{S_i}(\zeta^1,0,0)\}_{i=1}^{k}$ for $\widetilde{C}_{[1]}(T_{2,2k+1};\Delta_{\Z[T^{\pm 1}]})$ over $\Z[T^{\pm 1}]$. Moreover, if a given immersed cobordism $S\colon T_{2,3} \to T_{2,2k+1}$ satisfies
\[
s_+(S)+g(S)=k-1,
\]
then we have
\[
\widetilde{\lambda}_S(\zeta^1,0,0)
-\widetilde{\lambda}_{S_{s_+(S)+1}}(\zeta^1,0,0) \in 
\Z[T^{\pm 1}] \cdot
\left\langle \widetilde{\lambda}_{S_{s_+(S)+2}}(\zeta^1,0,0), \ldots, \widetilde{\lambda}_{S_{k}}(\zeta^1,0,0)\right\rangle.
\]
\end{lem}
\begin{proof}
By \Cref{lem:lambda-T23},
the element
$\widetilde{\lambda}_{S_i}(\zeta^1,0,0)$ is in the form of
\begin{align}
\label{eq:image-of-cob-map2}
\widetilde{\lambda}_{S_i}(\zeta^1,0,0) = 
\left(\zeta^{i} + \sum_{j=i +1}^k a_j \zeta^j, 0 , 0 \right). \end{align}
(Here, it follows from the $\Z/4$-grading reason that $\mu$ and $\Delta_1$ in \eqref{S-morphism} maps $\zeta^1$ to zero.)
Hence, the first half assertion immediately follows from the fact that 
\[
\{(\zeta^1,0,0), \ldots (\zeta^k,0,0) \}
\]
is a free basis for $\widetilde{C}_{[1](T_{2,2k+1};\Delta_{\Z[T^{\pm 1}]})}$ over $\Z[T^{\pm 1}]$.
Moreover, the form \eqref{eq:image-of-cob-map2} also shows
the equalities
\[
\widetilde{\lambda}_S(\zeta^1,0,0)-\widetilde{\lambda}_{S_{s_+(S)+1}}(\zeta^1,0,0)
= 
\left(0,\sum_{i=s_+(S) +2}^k a_i \zeta^i, 0 \right) 
\]
and
\[
\Z[T^{\pm 1}] \cdot
\left\langle 
(\zeta^{s_+(S)+2},0,0), \ldots, (\zeta^{k}, 0,0)
\right\rangle
=
\Z[T^{\pm 1}] \cdot
\left\langle \widetilde{\lambda}_{S_{s_+(S)+2}}(\zeta^1,0,0), \ldots, \widetilde{\lambda}_{S_{k}}(\zeta^1,0,0) \right\rangle.
\]
These imply the second half assertion.
\end{proof}
Now we show an analogy of \Cref{cor:family-of-cob},
which gives information of $I^\sharp_{[1]}(T_{2,2k+1};\Z)$.
Let $\zeta^\sharp$ be a generator of $I^\sharp_{[1]}(T_{2,3};\Z) \cong \Z$, which is coresponding to 
\[
[\big((0,0,0),(\zeta, 0,0) \big)] \in H_*(\wt{C}^{\sharp}_*(T_{2,3};\Delta_{\Z}))
\]
via the canonical isomorphism given by \Cref{recov I sharp}.
\begin{cor}
\label{cor:family-of-cob2}
An arbitrary set of immersed cobordisms 
$\{S_i \colon T_{2,3} \to T_{2,2k+1}\}_{i=1}^k$ with the property
\[
s_+(S_i)= i -1  \text{ and } \ g(S_i)= k-i
\]
gives a free basis $\{I^\sharp_{S_i}(\zeta^\sharp)\}_{i=1}^k$ of $I^\sharp_{[1]}(T_{2,2k+1};\Z)$.
Moreover, if a given immersed cobordism $S\colon T_{2,3} \to T_{2,2k+1}$ satisfies 
\[
s_+(S)+g(S)=k-1,
\]
then we have
\[
I^\sharp_S(\zeta^\sharp)-I^\sharp_{S_{s_+(S)}+1}(\zeta^\sharp) \in 
\Z \cdot
\left\langle I^\sharp_{S_{s_+(S)+2}}(\zeta^\sharp), \ldots, I^\sharp_{S_{k}}(\zeta^\sharp) \right\rangle.
\]
\end{cor}

\begin{proof}
Note that
\[
\Ker \wt{d}^\sharp \cap \widetilde{C}^\sharp_{[1]}(T_{2,2k+1};\Delta_{\Z})  
= 0
\oplus \left(\wt{C}_{[1]}(T_{2,2k+1};\Delta_{\Z[T^{\pm 1}]} \otimes_{\Z[T^\pm 1]} \Z \right)
\]
and
\[
\wt{d}^\sharp (\wt{C}^\sharp_{[1]}) = 0.
\]
Moreover, we have $\wt{\lambda}^\sharp_S(\alpha, \beta) = (\wt{\lambda}_S(\alpha), \wt{\lambda}_S(\beta))$.
Therefore, the assertion directly follows from \Cref{lem:family-of-cob2} and 
the commutativity of (\ref{eq:equiv-diag-sharp}).
\end{proof}

\subsection{Constraints from h-filtration}
Now we apply \Cref{mainKF} to the study of $I^\sharp(T_{2,2k+1};\Z)$, and then we can conclude that cobordism maps $I^\sharp_S$ and $\Kh^{\low}_S$ from $U_1$ (resp.\ $T_{2,3}$) to $T_{2,2k+1}$ with $s_+ +g(S)=k$ (resp.\ $s_+ +g(S)=k-1$) are uniquely determined by the values $(s_+,g(S))$. 

We first recall the relation between homological gradings and absolute Floer $\Z/4$-gradings. 
In \cite[Section 8.1]{KM11u}, a $\Z/4$-grading on $\Kh(K)$ is defined as
\[
    q - h -b_0(K).
\]
Note that the $\Z/2$-reduction of the above relation is more useful as follows:
\begin{prop}
The $\Z/2$-reduction of the $\Z/4$-grading on $\Kh$ coincides with that of the mod 2 $h$-grading.
\end{prop}

\begin{proof}
Let us prove $q - b_0(K) \overset{(2)}{\equiv} 0$.
Since the $q$-grading mod 2 is constant on each $V(D_v)$ and unchanged by the differential, it suffices to consider the case where $D_v$ is the orientation state. Then, $q$ mod 2 is equal to the Euler characteristic of the Seifert surface derived from $D$, which coincides with $b_0(K)$.
\end{proof}

When one wants to compare the $I^\sharp$-functor 
with the $\Kh$-functor,
the following degeneration is very effective.

\begin{lem}
\label{lem:collapse}
For any $k \in \Z_{\geq 0}$,
the spectral sequence $E^r(\CKh^\sharp(T_{2,2k+1}))$ degenerates at the $E^2$-stage. In particular, we have 
\[
F_{p}I^\sharp_{[0]}(T_{2,2k+1};\Z)
/F_{p+1}I^\sharp_{[0]}(T_{2,2k+1};\Z)
\cong 
\bigoplus_{q \in \Z} \Kh_{[0]}(T^*_{2,2k+1};\Z)^{p,q},
\]
where the isomorphism is derived from \Cref{lem:convergence}.
\end{lem}

\begin{proof}
This immediately follows from \Cref{cor:CKh-T_2q}, \Cref{lem:T2q-sharp} and \Cref{thm:degeneration}.
\end{proof}

Now we show how the $h$-filtration affects the study of $I^\sharp_S$. (In particular, remark that the last assertion in the following lemma is much stronger than the corresponding assertion in \Cref{cor:family-of-cob}.)
\begin{thm}
\label{thm:h-filt-basis}
Let
$\{S_i \colon U_1 \to T_{2,2k+1}\}_{i=0}^k$
be a set of immersed oriented cobordisms, such that for each $i$,
\[
s_+(S_i)= i  \text{ and } \ g(S_i)= k-i.
\]
Then for each $0 \leq p \leq k$, the set $\{I^\sharp_{S_i}(u_+), I^\sharp_{S_i}(u_-)\}_{i=0}^p$ is a generating set of \[F_{-2p}I^\sharp_{[0]}(T_{2,2k+1};\Z)=F_{-2p-1}I^\sharp_{[0]}(T_{2,2k+1};\Z) 
\big(\cong \Z^{p+2} \oplus (\Z/2\Z)^{p} \big).\]
Moreover, 
if a given immersed cobordism $S\colon U_1 \to T_{2,2k+1}$ satisfies $s_+(S)+g(S)=k$,
then we have
\[
I^\sharp_S(u_+) = I^\sharp_{S_{s_+(S)}} (u_+)
\quad \text{and} \quad 
I^\sharp_S(u_-) = I^\sharp_{S_{s_+(S)}} (u_-).
\]
\end{thm}

\begin{proof}
Let us denote $F_{j}I^\sharp_{[0]} := F_{j}I^\sharp_{[0]}(T_{2,2k+1};\Z)$.
The equality $F_{-2p}I^\sharp_{[0]}=F_{-2p-1}I^\sharp_{[0]}$
follows from $F_{-2p-1}I^\sharp_{[0]}/F_{-2p}I^\sharp_{[0]} 
\cong \bigoplus_{q \in \Z} \Kh_{[0]}(T^*_{2,2k+1};\Z)^{-2p-1,q} = \{0\}$ (given by \Cref{lem:collapse}).
Moreover,
the inequality
\[
\Z \cdot 
\left\langle 
I^\sharp_{S_i}(u_+), I^\sharp_{S_i}(u_-) \ \middle| \ 
0 \leq i \leq p
\right\rangle
\subset F_{-2p}I^\sharp_{[0]}
\]
immediately follows from grading arguments with respect to the h-filtration.
Here we prove the opposite inequality by induction of $p$. 
For proving the base case, suppose that $\zeta \in F_0I^\sharp_{[0]}$.
Then, by \Cref{cor:family-of-cob}, the equality
\[
\zeta = \sum_{i=0}^k a_i I^\sharp_{S_i}(u_+) + \sum_{i=0}^k b_i I^\sharp_{S_i}(u_-)
\]
holds (as elements of $I^\sharp_{[0]}(T_{2,2k+1};\Z)$) for some $a_i,b_i \in \Z$. Moreover, from 
the isomorphisms
\[F_0I^\sharp_{[0]} \cong \bigoplus_{q \in \Z}\Kh_{[0]}(T_{2,2k+1};\Z)^{0,q} \cong \Z^2
\]
and the linear independence of $\{I^\sharp_{S_0}(u_+), I^\sharp_{S_0}(u_-)\}$, we also have
\[
n\zeta = c I^\sharp_{S_0}(u_+) + d I^\sharp_{S_0}(u_-)
\]
for some $n, c, d \in \Z$ with $n \neq 0$. Comparing them and using the facts that $\{I^\sharp_{S_i}(u_+), I^\sharp_{S_0}(u_-)\}_{i=0}^p$ are
linearly independent and
$\{I^\sharp_{S_i}(u_-)\}^k_{i=1}$ are 2-tortions,
we have $2n\zeta = 2n a_0 I^\sharp_{S_0}(u_+) + 2n b_0 I^\sharp_{S_0}(u_-)$.
Now, since $F_0I^\sharp_{[0]}$ is free, 
we can conclude that 
$\zeta =
a_0 I^\sharp_{S_0}(u_+) + b_0 I^\sharp_{S_0}(u_-) \in
\left\langle I^\sharp_{S_0}(u_+), I^\sharp_{S_0}(u_-) \right\rangle$.

As the induction step, for given $p>0$,  assume that the equalities
\[
\Z \cdot 
\left\langle 
I^\sharp_{S_i}(u_+), I^\sharp_{S_i}(u_-) \ \middle| \ 
0 \leq i \leq p-1
\right\rangle
= F_{-2p+2}I^\sharp_{[0]} = F_{-2p+1}I^\sharp_{[0]}
\]
holds. 
We first prove that the elements $m\cdot I^\sharp_{S_p}(u_+)$ and $ I^\sharp_{S_p}(u_-)$ ($m \neq 0$) are non-zero in the quotient group
\[
F_{-2p}I^\sharp_{[0]}/F_{-2p+1}I^\sharp_{[0]} \cong \bigoplus_{q \in \Z}\Kh_{[0]}(T_{2,2k+1};\Z)^{-2p,q} \cong \Z \oplus (\Z/2\Z).
\]
If not, then the assumption of induction implies
\[
m \cdot I^\sharp_{S_p}(u_+) = 
\sum_{i=0}^{p-1} a^+_i I^\sharp_{S_i}(u_+) + \sum_{i=0}^{p-1} b^+_i I^\sharp_{S_i}(u_-)
\quad \text{and} \quad
I^\sharp_{S_p}(u_-) = 
\sum_{i=0}^{p-1} a^-_i I^\sharp_{S_i}(u_+) + \sum_{i=0}^{p-1} b^-_i I^\sharp_{S_i}(u_-).
\]
Moreover, observations of
$2m \cdot I^\sharp_{S_p}(u_+)$ and 
$2 I^\sharp_{S_p}(u_-)$
reduce these to
\[
2m \cdot I^\sharp_{S_p}(u_+) = 
\sum_{i=0}^{p-1} 2a^+_i I^\sharp_{S_i}(u_+)
+ 2b^+_i I^\sharp_{S_0}(u_-)
\quad \text{and} \quad
I^\sharp_{S_p}(u_-) = 
\sum_{i=1}^{p-1} b^-_i I^\sharp_{S_i}(u_-),
\]
where the first equality contradicts to the linear independence of 
$\{I^\sharp_{S_i}(u_+), I^\sharp_{S_0}(u_-)\}_{i=0}^p$,
and the second equality contradicts to the fact that 
$\{I^\sharp_{S_i}(u_-)\}_{i=1}^k$ generates
$\Tor I^\sharp_{[0]}(T_{2,2k+1};\Z) \cong (\Z/2\Z)^k$
(both facts are proved in \Cref{cor:family-of-cob}).
As a consequence, we see that
\[\Tor (F_{-2p}I^\sharp_{[0]}/F_{-2p+1}I^\sharp_{[0]}) 
=  \Z \cdot \left\langle \left[I^\sharp_{S_p}(u_-)\right] \right\rangle\]
and 
\[n \cdot \left( (F_{-2p}I^\sharp_{[0]}/F_{-2p+1}I^\sharp_{[0]})/\Tor \right) 
=  \Z \cdot \left\langle \left[I^\sharp_{S_p}(u_+)\right] \right\rangle\]
for some $n \neq 0$.

Now, suppose that $\zeta \in F_{-2p}I^\sharp_{[0]}$.
If $\zeta$ is torsion, then the above arguments imply
\[
\zeta \in \left( \Z \cdot \left\langle I^\sharp_{S_p}(u_-) \right\rangle + F_{-2p+1}I^\sharp_{[0]} \right)
\subset 
\Z \cdot 
\left\langle 
I^\sharp_{S_i}(u_+), I^\sharp_{S_i}(u_-) \ \middle| \ 
0 \leq i \leq p
\right\rangle
.\]
If $\zeta$ is non-torsion,
then the above arguments imply
\[
n\zeta =
\sum_{i=0}^{p} a_i I^\sharp_{S_i}(u_+) + \sum_{i=0}^{p} b_i I^\sharp_{S_i}(u_-)
\]
for some $a_i, b_i \in \Z$ and $n \neq 0$.
On the other hand,
\Cref{cor:family-of-cob} also gives a linear combination
\[
\zeta = \sum_{i=0}^k c_i I^\sharp_{S_i}(u_+) + \sum_{i=0}^k d_i I^\sharp_{S_i}(u_-).
\]
Comparing these, we have
\[
2n\zeta = \sum_{i=0}^p 2n c_i I^\sharp_{S_i}(u_+)
+ 2n d_0 I^\sharp_{S_0}(u_-),
\]
which implies
$\zeta - \sum_{i=0}^p c_i I^\sharp_{S_i}(u_+)
- d_0 I^\sharp_{S_0}(u_-)\in 
\Tor F_p I^\sharp_{[0]} = 
\Z \cdot 
\left\langle 
I^\sharp_{S_i}(u_-) \ \middle| \ 
1 \leq i \leq p
\right\rangle.
$
This completes the proof of the induction step.

Finally, we prove the equality
$I^\sharp_S(u_+) = I^\sharp_{S_{s_+(S)}} (u_+)$.
(The negative sign version can also be proved similarly.)
Let $s_+ := s_+(S)$. Since 
$I^\sharp_S(u_+) \in F_{s_+}I^\sharp_{[0]}$, the above arguments and \Cref{cor:family-of-cob}
provide two linear combinations
\[
I^\sharp_S(u_+) - I ^\sharp_{S_{s_+}}(u_+) = \sum_{i=0}^{s_+} a_i I^\sharp_{S_i}(u_+) + \sum_{i=0}^{s_+} b_i I^\sharp_{S_i}(u_-)
\]
and
\[
I^\sharp_S(u_+) - I ^\sharp_{S_{s_+}}(u_+) = \sum_{i= s_+ + 1}^{k} c_i I^\sharp_{S_i}(u_+).
\]
Then, by 
the linear independence of 
$\{I^\sharp_{S_i}(u_+), I^\sharp_{S_0}(u_-)\}_{i=0}^p$,
and the fact that 
$\{I^\sharp_{S_i}(u_-)\}_{i=1}^k$ generates
$\Tor I^\sharp_{[0]}(T_{2,2k+1};\Z) \cong (\Z/2\Z)^k$,
we have $a_i=b_i = c_i = 0$. This completes the proof.
\end{proof}

Since $I^\sharp_*(U_1) = \Z \cdot \langle u_+, u_- \rangle$, the last assertion of \Cref{thm:h-filt-basis}
can be restated as the following rigidity theorem. 
\begin{cor}
\label{cor:uniqueness}
Suppose that immersed cobordisms
$S, S' \colon U_1 \to T_{2,2k+1}$
satisfy
\[
s_+(S)+g(S)=s_+(S')+g(S')=k, \quad 
s_+(S)= s_+(S') \quad  \text{and} \quad g(S)= g(S').
\]
then the two maps
$I^\sharp_S, I^\sharp_{S'} \colon 
I^\sharp(U_1;\Z) \to I^\sharp(T_{2,2k+1};\Z)$ are equal.
\end{cor}

Next, we give an analogous result of \Cref{thm:h-filt-basis}, where the unknot $U_1$ (resp.\ $u_+, u_-$) is replaced with the trefoil $T_{2,3}$ (resp.\ $\zeta^\sharp$). (Recall that $\zeta^\sharp$ is a generator of $I^\sharp_{[1]}(T_{2,3};\Z) \cong \Z$ appearing in \Cref{cor:family-of-cob2}.)
This is useful for characterizing each element whose $\Z/4$-grading is odd.
\begin{thm}
\label{thm:h-filt-basis-odd}
Let
$\{S_i \colon T_{2,3} \to T_{2,2k+1}\}_{i=1}^k$
be a set of immersed oriented cobordisms, such that for each $i$,
\[
s_+(S_i)= i -1  \text{ and } \ g(S_i)= k-i.
\]
Then for each $1 \leq p \leq k$, the set $\{I^\sharp_{S_i}(\zeta^\sharp)\}_{i=1}^p$ is a free basis for \[F_{-2p-1}I^\sharp_{[1]}(T_{2,2k+1};\Z)=F_{-2p-2}I^\sharp_{[1]}(T_{2,2k+1};\Z)\]
over $\Z$.
Moreover,
if a given immersed cobordism $S\colon T_{2,3} \to T_{2,2k+1}$ satisfies $s_+(S)+g(S)=k-1$,
then we have
\[
I^\sharp_S(\zeta^\sharp) = I^\sharp_{S_{s_+ + 1}}(\zeta^\sharp).
\]
\end{thm}
The proof of \Cref{thm:h-filt-basis-odd}
is regarded as an easier version of \Cref{thm:h-filt-basis}, while we write it for the ease of checking.
\begin{proof}
Let us denote $F_{j}I^\sharp_{[1]} := F_{j}I^\sharp_{[1]}(T_{2,2k+1};\Z)$.
The equality $F_{-2p-1}I^\sharp_{[1]}
=F_{-2p-2}I^\natural_{[1]}$
follows from 
$F_{-2p-2}I^\sharp_{[1]}/F_{-2p-1}I^\sharp_{[1]} \cong \bigoplus_{q \in \Z} \Kh_{[1]}(T^*_{2,2k+1};\Z)^{-2p-2,q} = \{0\}$ (given by \Cref{lem:collapse}).
Moreover, the inequality
\[
\Z \cdot 
\left\langle 
I^\sharp_{S_{0}}(\zeta^\sharp), \ldots, I^\sharp_{S_p}(\zeta^\sharp)
\right\rangle
\subset F_{-2p-1}I^\sharp_{[1]}
\]
immediately follows from grading arguments with respect to the h-filtration.
Here we prove the opposite inequality by induction of $p$. 
In addition, we prove only the induction step. (Indeed, we can regard the equality $F_{-1}I^\sharp_{[1]}=\{0\}$ given by \Cref{lem:collapse} as the base case $p=0$.) Assume that the equalities
\[
\Z \cdot 
\left\langle 
I^\sharp_{S_{1}}(\zeta^\sharp), \ldots, 
I^\sharp_{S_{p-1}}(\zeta^\sharp)
\right\rangle
= F_{-2p+1}I^\sharp_{[1]} = F_{-2p}I^\sharp_{[1]}
\]
hold. Then, by \Cref{cor:family-of-cob2}, $I^\sharp_{S_p}(\zeta^\sharp)$ is non-trivial in 
\[
F_{-2p-1}I^\sharp_{[1]}/F_{-2p}I^\sharp_{[1]} \cong
\bigoplus_{q \in \Z} \Kh_{[1]}(T^*_{2,2k+1};\Z)^{-2p-1,q}
\cong \Z.
\]
This implies that any non-zero element $\zeta \in F_{-2p-1}I^\sharp_{[1]}$ can be written as
\[
n \zeta = \sum_{i=1}^p a_i I^\sharp_{S_i}(\zeta^\sharp) 
\]
for some integer $n \neq 0$. On the other hand,
\Cref{cor:family-of-cob2} also gives a linear combination
\[
\zeta = \sum_{i=1}^k b_i I^\sharp_{S_i}(\zeta^\sharp).
\]
Comparing these and use the fact that $F_pI^\sharp_{[1]}$, 
is free, we have 
$\zeta \in \Z \cdot 
\left\langle 
I^\sharp_{S_{1}}(\zeta^\sharp), \ldots, 
I^\sharp_{S_{p-1}}(\zeta^\sharp)
\right\rangle$, which proves the induction step.

Finally, we prove the last assertion.
Let $s_+ := s_+(S)$. Since $I^\sharp_{S}(\zeta^\sharp) \in F_{-2s_+ -3}I^\sharp_{[1]}$, the above arguments and \Cref{cor:family-of-cob2} provide two linear combinations
\[
I^\sharp_{S}(\zeta^\sharp) - I^\sharp_{S_{s_+ +1}}(\zeta^\sharp) = \sum_{i=1}^{s_+ + 1} a_i I^\sharp_{S_i}(\zeta^\sharp)
\quad \text{and} \quad
I^\sharp_{S}(\zeta^\sharp) - I^\sharp_{S_{s_+ +1}}(\zeta^\sharp) = \sum_{s_+ + 2}^{k} a_i I^\sharp_{S_i}(\zeta^\sharp).
\]
By \Cref{cor:family-of-cob2},
the elements $\{I^\sharp_{S_i}(\zeta^\sharp)\}_{i=1}^k$ are linearly independent, 
and hence $a_i = b_i = 0$.
\end{proof}
We also give an analogous result of \Cref{cor:uniqueness}.
(Note that $I^\sharp(T_{2,3};\Z) \cong \Z^4 \oplus (\Z/2\Z)$ is not generated only by $\zeta^\sharp$,
we need a little more discussion than the case of \Cref{cor:uniqueness-trefoil}.)

\begin{cor}
\label{cor:uniqueness-trefoil}
Suppose that immersed cobordisms
$S, S' \colon T_{2,3} \to T_{2,2k+1}$
satisfy
\[
s_+(S)+g(S)=s_+(S')+g(S')=k-1, \quad 
s_+(S)= s_+(S') \quad  \text{and} \quad g(S)= g(S').
\]
then the two maps
$I^\sharp_S, I^\sharp_{S'} \colon 
I^\sharp(T_{2,3};\Z) \to I^\sharp(T_{2,2k+1};\Z)$ are equal.
\end{cor}

\begin{proof}
It follows from \Cref{thm:h-filt-basis} that
$I^\sharp(T_{2,3};\Z)$ is generated by 
$\{I^\sharp_{S_0}(u_\pm), I^\sharp_{S_1}(u_{\pm}), \zeta^\sharp \}$,
where $S_i$ ($i=0,1$) is an immersed oriented cobordism $U_1 \to T_{2,3}$ with $s_+(S_i)=i$ and $g(S_i)=1-i$.
Since $S \circ S_i$ and $S' \circ S_i$ 
satisfy $s_+ + g = k$ and
have the same values of $(s_+,g)$,
it is shown by \Cref{cor:uniqueness}, \Cref{thm:h-filt-basis-odd} and the functoriality of $I^\sharp$
that
\[
I^\sharp_{S}(I^\sharp_{S_i}(u_\pm))
= I^\sharp_{S}\circ I^\sharp_{S_i}(u_\pm)
= I^\sharp_{S \circ S_i}(u_\pm)
= I^\sharp_{S' \circ S_i}(u_\pm)
= I^\sharp_{S'}\circ I^\sharp_{S_i}(u_\pm)
= I^\sharp_{S'}(I^\sharp_{S_i}(u_\pm))
\]
for each sign and $i=0,1$, and
\[
I^\sharp_S(\zeta^\sharp) = I^\sharp_{S'}(\zeta^\sharp). 
\]
These complete the proof.
\end{proof}

Now, let us give a proof of \Cref{thm:T2q}. 
\begin{proof}[Proof of \Cref{thm:T2q}]
Here we show the proof of the assertion (1) in \Cref{thm:T2q}.
Let 
\[
G_pI^\sharp_{[0]}(T_{2,2k+1};\Z):=
F_{p}I^\sharp_{[0]}(T_{2,2k+1};\Z)
/F_{p+1}I^\sharp_{[0]}(T_{2,2k+1};\Z).
\]
Then, 
it follows from \Cref{mainKF}, \Cref{lem:collapse} and \Cref{lem:degeneration-map} that there exist isomorphisms $\gamma_1, \gamma_2$ such that the diagram
\[
\xymatrix@C=80pt{
G_{0}I^\sharp_{[0]}(U_1;\Z)
\ar[r]^-{G_0I^\sharp_{S^*}}
\ar[d]_-{\gamma_1}
& G_{-2s_-}I^\sharp_{[0]}(T_{2,2k+1};\Z)
\ar[d]^-{\gamma_2}\\
\Kh^0(U_1;\Z)
\ar[r]^-{\Kh^{\low}_S}
& 
\Kh^{-2s_-}(T_{2,2k+1}^*;\Z)
}
\]
is commutative up to over all sign. Moreover,
\Cref{thm:h-filt-basis} shows that 
\[
G_{-2s_-}I^\sharp_{[0]}(T_{2,2k+1};\Z) = \left\langle G_0I^\sharp_{S^*}(u_+), G_0I^\sharp_{S^*}(u_-) \right\rangle.
\]
Now, the commutativity of the above diagram completes the proof.

The assertion (2) in \Cref{thm:T2q} is proved in a similar way, where we use
\Cref{thm:h-filt-basis-odd}
in stead of 
\Cref{thm:h-filt-basis}.
\end{proof}

\subsection{Induced maps from concordances}
Here, we show several applications of  our observation
to the induced maps on $I^\sharp$ and $\Kh$ from concordances whose domain is $T_{2,2k+1}$. 
We start with the instanton side.
\begin{thm}
\label{thm:self-concordance-sharp}
For any $k \geq 0$,
any oriented self-concordance $C \colon T_{2,2k+1} \to T_{2,2k+1}$ induces the identity map 
\[
I^\sharp_{C} = \id_{I^\sharp(T_{2,2k+1})} \colon I^\sharp(T_{2,2k+1};\Z) 
\to I^\sharp(T_{2,2k+1};\Z).
\]
\end{thm}

\begin{proof}
The case $k=0$ is corresponding to the assertion (3) in
\Cref{prop:unlink-properties}, and hence
we may assume that $k>0$.
Choose sets 
\[
\{S_i \colon U_1 \to T_{2,2k+1}\}_{i=0}^k
\quad \text{and} \quad
\{S'_i \colon T_{2,3} \to T_{2,2k+1}\}_{i=1}^{k}
\]
of immersed cobordisms
with the properties 
\[
\big(s_+(S_i), g(S_i)\big)= (i, k-i)  \text{ and } 
\big(s_+(S'_i), g(S'_i)\big)= (i-1, k-i).
\]
By \Cref{thm:h-filt-basis} and \Cref{thm:h-filt-basis-odd}, the set 
$\{I^\sharp_{S_i}(u_+),I^\sharp_{S_i}(u_-)\}_{i=0}^k \cup \{I^\sharp_{S'_i}(\zeta^\sharp)\}_{i=1}^k$ 
generates $I^\sharp(T_{2,2k+1})$. Moreover,
\Cref{cor:uniqueness} and \Cref{cor:uniqueness-trefoil} give the equalities
\[
I^\sharp_{C}(I^\sharp_{S_i}(u_{\pm})) = 
I^\sharp_{C \circ S_i}(u_\pm) = I^\sharp_{S_i}(u_\pm)
\quad \text{and} \quad
I^\sharp_{C}(I^\sharp_{S'_i}(\zeta^\sharp)) = 
I^\sharp_{C \circ S'_i}(\zeta^\sharp) = I^\sharp_{S'_i}(\zeta^\sharp).
\]
for any $0 \leq i \leq k$ and each sign. These imply that the map $I^\sharp_C$ coincides with $\id_{I^\sharp(T_{2,q})}$.
\end{proof}

Then, an analogy of \Cref{thm:left-inverse} on the instanton side is proved as follows.

\begin{thm}\label{thm:left-inverse-sharp}
For any smooth knot concordance $C$ from a positive two-bridge torus knot, the induced map on $I^\sharp$ is injective, with a left inverse given by the reversal of \(C\).
\end{thm}

\begin{proof}
For given concordance $C \colon T_{2,2k+1} \to K$, let $\overline{C} \colon K \to T_{2,2k+1}$ be the reversal of $C$.
Applying \Cref{thm:self-concordance-sharp}
to $\overline{C} \circ C \colon T_{2,2k+1} \to T_{2,2k+1}$, we have
\[
I^\sharp_{\overline{C}} \circ I^\sharp_C
= I^\sharp_{\overline{C} \circ C} = \id_{I^\sharp(T_{2,2k+1};\Z)}.
\]
This completes the proof.
\end{proof}

Next, we discuss the Khovanov side. We first show
an analogy of \Cref{thm:self-concordance-sharp}.
\begin{thm}
\label{thm:self-concordance-Kh}
For any $k \geq 0$,
any oriented self-concordance $C \colon T^*_{2,2k+1} \to T^*_{2,2k+1}$ induces the identity map 
\[
\Kh_{C} = \id_{\Kh(T^*_{2,2k+1})} \colon \Kh(T^*_{2,2k+1};\Z) 
\to \Kh(T^*_{2,2k+1};\Z).
\]
\end{thm}

\begin{proof}
By \cite[Theorem 1.1]{ISST1}, \Cref{lem:collapse} and 
\Cref{lem:degeneration-map}, there exists
an isomorphism $\gamma$ such that the diagram
\[
\xymatrix@C=80pt{
\bigoplus_{p\in\Z}G_{p}I^\sharp(T_{2,2k+1};\Z)
\ar[r]^-{\bigoplus G_{p}I^\sharp_{C^*}}
\ar[d]_-{\gamma}
& 
\bigoplus_{p \in \Z} G_{p}I^\sharp(T_{2,2k+1};\Z)
\ar[d]^-{\gamma}\\
 \Kh(T_{2,2k+1}^*;\Z)
\ar[r]^-{\varepsilon \Kh_C}
& 
\Kh(T_{2,2k+1}^*;\Z)
}
\]
is commutative for some $\varepsilon \in \{\pm 1\}$.
Moreover, \Cref{thm:self-concordance-sharp} shows
$\bigoplus G_p I^\sharp_{C^*} = \id$, and hence we have 
\[
 \Kh_C = \varepsilon (\gamma \circ \bigoplus G_p I^\sharp_{C^*} \circ \gamma^{-1})
 = \varepsilon (\gamma \circ \gamma^{-1})
= \varepsilon \id_{\Kh(T^*_{2,2k+1};\Z)}.
 \]
 Here, we are adjusting the signs of the oriented cobordism maps according to \cite{Sano:2020-b}, so that for any concordance, the induced map on $\Kh$ maps some non-torsion element to itself.
 This implies $\varepsilon = 1$.
\end{proof}

Now we prove \Cref{thm:left-inverse}. 
\begin{proof}[Proof of \Cref{thm:left-inverse}]
For given concordance $C \colon T^*_{2,2k+1} \to K$, 
let $\overline{C} \colon K \to T^*_{2,2k+1}$ be the reversal of $C$.
Then, the composition $\overline{C} \circ C \colon T^*_{2,2k+1}  \to T^*_{2,2k+1} $ is a self-concordance, and hence 
\Cref{thm:self-concordance-Kh} implies
\[
\Kh_{\overline{C}} \circ \Kh_{C} 
= \Kh_{\overline{C} \circ C} = \id_{\Kh(T^*_{2,2k+1};\Z)}.  
\]
This completes the proof. 
\end{proof}

\subsection{Analogous results on generalized 
Khovanov homology}

Finally, we show that \Cref{thm:self-concordance-Kh} 
and \Cref{thm:left-inverse} can be extended to any generalized Khovanov homology $\Kh_{h,t}$.

Hereafter, we assume that any ground ring $R$ is non-graded. In addition, since we only consider the immersed cobordism map of lowest homological degree $\phi^{\low}_S$, we simply denote it by $\phi_S$.

For a link diagram $D$, the homological and quantum gradings give a decomposition of the complex as $R$-modules,
\[
\CKh_{h,t}(D;R) = \bigoplus_{p,q \in \Z} C^{p,q} = \bigoplus_{p,q \in \Z} \CKh_{h,t}(D;R)^{p,q}
\]
where each summand $C^{p,q}$ is the free $R$-module generated by the enhanced states of bigrading $(p, q)$. Note that $C^{p,q}=0$ if $q \equiv |D| + 1$ (mod 2) and $d(C^{p,q}) \subset C^{p+1,q} \oplus C^{p+1, q+2} \oplus C^{p+1, q+4}$, where $|D|$ is the number of components of $D$.
Set
\[
F_i\CKh_{h,t}(D;R) := \bigoplus_{p \in \Z,  q \geq i } C^{p,q}.
\]
Then $\{ F_i\CKh_{h,t}(D;R) \}_{i \in \Z}$ gives a filtration on $\rCKh_{h,t}(D;R)$.
We call it {\it the quantum filtration}.
The following is easy to check.
\begin{lem}
For any based generic cobordism $S \colon D \to D'$
between link diagrams, the induced chain map
\[
\phi_S \colon \CKh_{h,t}(D;R)
\to \CKh_{h,t}(D';R)
\]
is filtered of order $\geq \chi(S)$ w.r.t the quantum filtration. 
\end{lem}

Let $\{E^r_i(\CKh_{h,0}(D;R))\}$ denote the spectral sequence obtained from the quantum filtration.
    \footnote{
    Here, the Khovanov complex arise in the $E^0$-term of the spectral sequence obtained from the filtered chain complex, and hence the Khovanov homology arise in the $E^1$-term. Some authors use different conventions, for example in \cite{Ras10}, where Khovanov homology arise in the $E^2$-term of the spectral sequence obtained from the filtered Lee complex. Details are given in \Cref{app:homological-algebra}. 
}
(For the details of our notation, see \Cref{app:homological-algebra}.)
The following lemma immediately follows from the construction of $\CKh_{h,t}$.
\begin{lem}
\label{lem:BN-spectral}
The $R$-isomorphisms
\[
A_{h,t} \cong R^2 \cong A_{0,0}, \quad 1 \mapsto (1,0) \mapsto 1, \quad X \mapsto (0,1) \mapsto X.
\]
induce the chain isomorphism
\[
\eta \colon E^0_i(\CKh_{h,0}(D;R)) 
\overset{\cong}{\longrightarrow} 
\CKh_{0,0}(D;R)^{*, i}.
\]
Moreover, for any cobordism $C \colon D \to D'$
between link diagrams, we have the following commutative diagram:
\[
\xymatrix@C=80pt{
E_i^0(\CKh_{h,t}(D;R))
\ar[r]^-{E^0_i(\phi_S)}
\ar[d]_-{\eta}
& 
E_{i+\chi(S)}^0(\CKh_{h,t}(D';R))
\ar[d]^-{\eta}\\
\CKh_{0,0}(D;R)^{*, i}
\ar[r]^-{\phi_{S}}
& 
\CKh_{0,0}(D';R)^{*, i+\chi(S)}
}
\]
\end{lem}

Hereafter, we focus on the special case $D = T^*_{2,2k+1}$. 
Denote by
$\Kh_{h,t}(T^*_{2,2k+1};R)^p$ the homological grading $p$ part of $\Kh_{h,t}(T^*_{2,2k+1};R)$, and
set
\[
F_i\Kh_{h,t}(T^*_{2,2k+1};R)^p
:= \left\{ [x]\in \Kh_{h,t}(T^*_{2,2k+1};R))^p \  \middle| \  dx=0, \ x \in F_i\CKh_{h,t}(T^*_{2,2k+1};R) \right\}.
\]
Then, there is a direct decomposition 
$\Kh_{h,t}(T^*_{2,2k+1};R)^p = \bigoplus_{q \in \Z} \Kh_{h,t}(T^*_{2,2k+1};R)^{p,q}$ compatible with
$\{F_i\Kh_{h,t}(T^*_{2,2k+1};R)^p\}_{i \in \Z}$
as follows.
(Note that such a decomposition does not naturally follows from the definition of $\Kh_{h,t}$.)
\begin{lem}
\label{lem:BN-T2q-filt}
There exists a direct decomposition
\[
\Kh_{h,t}(T^*_{2,2k+1};R)^p
= \bigoplus_{q \in \Z} 
\Kh_{h,t}(T^*_{2,2k+1};R)^{p,q}
\]
such that for any $i \in \Z$, the equality
\[
F_i\Kh_{h,t}(T^*_{2,2k+1};R)^p
= \bigoplus_{q \geq i}\Kh_{h,t}(T^*_{2,2k+1};R)^{p,q}
\]
holds as $R$-submodules of $\Kh_{h,t}(T^*_{2,2k+1};R)^p$. 
\end{lem}

\begin{proof} 
Let $(C,d)$ be the chain complex in the left hand side of \Cref{cor:CKh-T_2q} and 
\[
\Phi \colon C \to \CKh_{h,t}(T^*_{2,2k+1})
\]
be a chain homotopy equivalence map given in \Cref{cor:CKh-T_2q}.
Then, for proving the lemma, it suffices to give a direct decomposition $H_p(C) = \bigoplus_{q \in \Z} H^{p,q}$ satisfying $F_iH_p(C) = \bigoplus_{q \geq i}H^{p,q}$. Indeed, 
since $\Phi$ is graded and filtered with respect to  homological grading and quantum filtration respectively,
the induced map $\Phi_* \colon H_*(C) \to \Kh_{h,t}(T^*_{2,2k+1})$ satisfies
\[
F_i\Kh_{h,t}(T^*_{2,2k+1})^p
= \Phi_*(F_iH_p(C))
= \bigoplus_{q \geq i} \Phi_*(H^{p,q}),
\]
and hence $\Kh_{h,t}(T^*_{2,2k+1})^{p,q} := \Phi_*(H^{p,q})$ gives the desired decomposition of 
$\Kh(T^*_{2,2k+1})^p$.

Here we note that for the chain complex $(C,d)$
satisfies the relation $d(C^{p,q}) \subset C^{p+1,q+2}$ for any $p,q \in \Z$. This implies that
\[
F_iH_p(C) = \frac{\ker d \cap (\bigoplus_{q \geq i }C^{p,q})}{d(\bigoplus_{q \in \Z}C^{p-1, q})}
= \bigoplus_{q \geq i} \frac{\ker d \cap C^{p,q}}{d(C^{p-1,q-2})}
\]
for any $i \in \Z$. Therefore, $H^{p,q} :=  \frac{\ker d \cap C^{p,q}}{d(C^{p-1,q-2})}$ gives the desired decomposition of $C$.
\end{proof}
Let $G_i\Kh_{h,t}(T^*_{2,2k+1};R)^p :=
F_i\Kh_{h,t}(T^*_{2,2k+1};R)^p/ F_{i-1} \Kh_{h,t}(T^*_{2,2k+1};R)^p$.
\Cref{lem:BN-T2q-filt} shows that the map
\[
\Kh_{h,t}(T^*_{2,2k+1};R)^{p,q} \to
G_q \Kh_{h,t}(T^*_{2,2k+1};R)^p,
\quad \xi \to [\xi]
\]
is an isomorphism for any $p,q \in \Z$, which is combined into an isomorphism
\[
\Kh_{h,t}(T^*_{2,2k+1};R) \overset{\cong}{\longrightarrow}
E^\infty(\CKh_{h,t}(T^*_{2,2k+1};R)).
\]
This isomorphism is canonical in the following sense:
\begin{cor}
\label{cor:BN-limit}
For any self-concordance $C \colon T^*_{2,2k+1} \to T^*_{2,2k+1}$,
the diagram
\[
\xymatrix@C=80pt{
\Kh_{h,t}(T^*_{2,2k+1};R)
\ar[r]^-{(\phi_{C})_*}
\ar[d]_-{\cong}
& 
\Kh_{h,t}(T^*_{2,2k+1};R)
\ar[d]^-{\cong}\\
E^\infty(\CKh_{h,t}(T^*_{2,2k+1};R))
\ar[r]^-{E^\infty(\phi_C)}
& 
E^\infty(\CKh_{h,t}(T^*_{2,2k+1};R))
}
\]
where the vertical maps are the isomorphism induced from \Cref{lem:BN-T2q-filt}.
\end{cor}
    
\begin{proof}
It follows from the definition of $E^\infty(\phi_C)$ that for any $\xi \in \Kh_{h,t}(T^*_{2,2k+1};R)^{p,q}$,
the equalities
\[
E^\infty(\phi_C)([\xi]) = G_q (\phi_C)_*([\xi])
= [(\phi_C)_*(\xi)]
\]
hold as elements of $G_q\Kh_{h,t}(T^*_{2,2k+1};R)^p$.
This completes the proof.
\end{proof}
Now, we have extensions of \Cref{thm:self-concordance-Kh} and \Cref{thm:left-inverse}
to all generalized Khovanov homologies 
(with PID coefficient).
\begin{thm}
\label{thm:self-concordance-general}
For any $k \geq 0$, $R$ and $h,t \in R$,
any oriented self-concordance $C \colon T^*_{2,2k+1} \to T^*_{2,2k+1}$ induces the identity map 
\[
\Kh_{C} = \id_{\Kh_{h,t}(T^*_{2,2k+1};R)} \colon \Kh_{h,t}(T^*_{2,2k+1};R) 
\to \Kh_{h,t}(T^*_{2,2k+1};R).
\]
\end{thm}

\begin{proof}
We first consider the following commutative diagram between the short exact sequences derived from the universal coefficient theorem;
\[
\xymatrix@C=60pt{
\Kh(T^*_{2,2k+1};\Z)^{p,q} \otimes R
\ar[r]^-{\iota}
\ar[d]_-{(\phi_{C})_* \otimes 1 = \id}
& 
\Kh(T^*_{2,2k+1}; R)^{p,q}
\ar[r]^-{\pi}
\ar[d]^-{(\phi_{C})_* }
& 
\Tor_1 (\Kh(T^*_{2,2k+1}; \Z)^{p+1,q}, R)
\ar[d]^-{\Tor_1(\phi_{C})_* = \id}
\\
\Kh(T^*_{2,2k+1};\Z)^{p,q} \otimes R
\ar[r]^-{\iota}
& 
\Kh(T^*_{2,2k+1}; R)^{p,q}
\ar[r]^-{\pi}
& 
\Tor_1 (\Kh(T^*_{2,2k+1}; \Z)^{p+1,q}, R)
}
\]
where the coincidence of two vertical maps with the identities follows from \Cref{thm:self-concordance-Kh}.
Then, it can be seen from \Cref{cor:CKh-T_2q}
that
\[
\Kh(T^*_{2,2k+1};\Z)^{p,q} \oplus
\Kh(T^*_{2,2k+1};\Z)^{p+1,q}
\ \cong \  \Z, \  \Z/2\Z \  \text{or} \  \Z \oplus \Z
\]
for any $p,q \in \Z$. This implies that either $\iota$ or $\pi$ is the zero-map, and consequently
the middle vertical map $(\phi_C)_*$ in the above diagram is also equal to the identity.

Next, by \Cref{lem:BN-spectral}, we have the commutative diagram:
\[
\label{eq:BN-diag}
\begin{split}
\xymatrix@C=80pt{
E^1(\CKh_{h,t}(T^*_{2,2k+1};R))
\ar[r]^-{E^1(\phi_C)}
\ar[d]_-{\eta_*}
& 
E^1(\CKh_{h,t}(T^*_{2,2k+1};R))
\ar[d]^-{\eta_*}\\
\Kh(T^*_{2,2k+1};R)
\ar[r]^-{(\phi_{C})_*= \id}
& 
\Kh(T^*_{2,2k+1};R)
}
\end{split}
\]
This shows $E^1(\phi_C) = \id$.  
Moreover, it follows from \Cref{lem:spectral-morphism} that 
\[
E^r(\phi_C) =\id 
\quad \Rightarrow \quad
E^{r+1}(\phi_C) = \id
\]
for any $r \geq 0$.
As a consequence, we have $E^\infty(\phi_C) = \id$, and now 
the commutative diagram in \Cref{cor:BN-limit}
completes the proof.
\end{proof}

\begin{thm}
\label{thm:left-inverse-general}
For any smooth knot concordance $C \colon T^*_{2,2k+1} \to K$, the induced map
\[
\Kh_C \colon \Kh_{h,t}(T^*_{2,2k+1};R) \to 
\Kh_{h,t}(K;R)
\]
is injective, with a left inverse given by the reversal of \(C\).
\end{thm}

\begin{proof}
The proof is similar to the proof of \Cref{thm:left-inverse}, where we use \Cref{thm:self-concordance-general} instead of \Cref{thm:self-concordance-Kh}.
\end{proof}

\appendix
\section{Homological algebra}
\label{app:homological-algebra}

We shall describe several propositions that we need to compare instanton cobordism maps with cobordism maps in Khovanov theory. Some of them are well-known in the general theory of spectral sequences of filtered complexes. Since our setting is a little bit unusual, i.e. we are having periodic gradings and doubly filtered on the cube complexes, we basically give the proofs of each proposition. 
For reference, see \cite{mccleary2001user} for example.

\subsection{Filtered complexes}

For a given principal ideal domain algebra $R$ over $\Z$,
the {\it filtered complexes} considered in this paper are $\Z/4$-graded chain complexes $(C,d)$ over $R$ such that $C$ is freely and finitely generated and admits a direct sum decomposition of the form
\[
C=\bigoplus_{i \geq i_0}C_i,
\]
where:
\begin{itemize}
    \item $d(C_i) \subset \bigoplus_{j \geq i} C_j$, and
    \item $C_i=\{0\}$ for all $i$ greater than some $i_1$.
\end{itemize}
For each $p \geq i_0$, set $F_pC:= \bigoplus_{i \geq p}C_i$.
Then we have the following sequence of subcomplexes:
\[
C=F_{i_0} \supset F_{i_0+1}C \supset \cdots 
\supset F_{i_1}C \supset F_{i_1+1}C =\{0\}.
\]
We also have a filtration on the homology $H_*(C)$, defined by
\[
F_pH_*(C) : = \{ [x]\in H_*(C) \mid dx=0, \ x \in F_pC \}.
\]
We set $G_pH_*(C) := F_pH_*(C)/F_{p+1}H_*(C)$, called the {\it associated graded pieces}.  

\subsection{Spectral sequence}
For a filtered complex $(C,d)$, we define {\it the $E^r$-complex} 
by
\[
E^r_p(C)
:= \frac{\{x \in F_pC \mid dx \in F_{p+r}C \}}{F_{p+1}C+d(F_{p-r+1}C)}
\]
and
\[
d^r_p\colon E^r_p(C) \to E^r_{p+r}(C), \quad [x]^r_p \mapsto [dx]^r_{p+r}.
\]
We call the sequence $\{(E^r(C),d^r):=(\bigoplus E^r_p(C),d^r_p)\}_{r \geq 0}$ {\it the spectral sequence} of $(C,d)$.

For the homology $H_*(E^r_p) := \Ker d^r_p / \im d^r_{p-r}$ also, we define the differential 
\[
\bd^r_p \colon H_*(E^r_p(C)) \to H_*(E^r_{p+r+1}(C)),
\quad
\bd^r_p([[x]^r_p]):= [[y]^r_{p+r+1}],
\]
where $y \in F_{p+r+1}C$ is a cycle satisfying $dx=y+dz$
for some $z \in F_{p+1}C$.
For any $r \geq 0$ and $p \geq i_0$, the map
\begin{align}
\label{eq:spectral-homology}
E^{r+1}_p(C) \to H_*(E^r_p(C)), \quad [x]^{r+1}_p \mapsto [[x]^r_p]
\end{align}
is a chain isomorphism.
In particular, we have the inequality $\rank_R E^r_p(C) \geq \rank_R E^{r+1}_p(C)$ for any $r \geq 0$ and $p \geq i_0$.
If the equality holds, then the following lemmas also hold.
\begin{lem}
\label{lem:rank-equality}
If
the equality $\rank_R E^r_p(C) = \rank_R E^{r+1}_p(C)$ holds,
then we have 
\[
\im d^r_{p-r} \subset \Tor E^r_{p}(C)
\quad \text{and} \quad
\Tor H_*(E^r_p(C)) = \Tor (\Ker d^r_p)/ \im d^r_{p-r}.
\]
In particular, if $\Tor E^r_p(C)$ is finite, then the inequality 
$|\Tor E^r_p(C)| \geq |\Tor E^{r+1}_p(C)|$ holds.
\end{lem}

\begin{proof}
Let $Q:=\Frac(R)$ be the field of fractions. Then, for any quotient $R$-module $C = B/A$ of finitely generated $R$-modules $A, B$, we have
the isomorphism
\begin{align*}
C\otimes_R Q \cong (B \otimes_R Q)/ (A \otimes_R Q)
\end{align*}
as $Q$-vector spaces. This implies
\begin{align}
\label{eq:quotient}
\rank_R C = \rank_R B - \rank_R A,
\end{align} 
and we have
\[
\rank_R E^{r+1}_p(C) = \rank_R H_*(E^r_p(C))
= \rank_R (\ker d^r_p) -\rank_R (\im d^r_{p-r}) \leq \rank_R E^r_p(C).
\]
Thus, the assumption $\rank_R E^r_p(C) = \rank_R E^{r+1}_p(C)$ shows 
\[
\rank_R (\Ker d^r_p) = \rank_R E^r_p(C)
\quad \text{and} \quad
\rank_R (\im d^r_{p-r}) = 0.
\]
In particular, The second equality dirctly shows 
$\im d^r_{p-r} \subset \Tor E^r_p(C)$. 

We next show that 
$\Tor H_*(E^r_p (C) ) 
= \Tor (\Ker d^r_p) / \im d^r_{p-r}$.
If an element $x \in \Ker d^r_p$ gives a torsion homology class $[x] \in \Tor H_*(E^r_p(C))$, then there exists an element $a \in R$ with $a x \in \im d^r_{p-r}$.
Here it is already shown that $\im d^r_{p-r} \subset \Tor E^r_p$, and hence we have $x \in \Tor (\Ker d^r_p)$.
This implies that the desired equality holds.

Finally, if $\Tor E^r_p(C)$ is finite, then the isomorphism $E^{r+1}_p(C) \cong H_*(E^r_p(C))$ shows that
\[
|\Tor E^{r+1}_p(C)| = 
|\Tor H_*(E^r_p (C) )| = 
\frac{|\Tor (\Ker d^r_p)|}{|\im d^r_{p-r}|} \leq |\Tor E^r_p(C)|. \qedhere
\]
\end{proof}
\begin{lem}
\label{lem:tor-equality}
Suppose that $\Tor E^r_p(C)$ is finite. If the equalities $\rank_R E^r_p(C) = \rank_R E^{r+1}_p(C)$ and $|\Tor E^r_p(C)| = |\Tor E^{r+1}_p(C)|$ hold, 
then we have $d^r_{p-r} = 0$.
\end{lem}

\begin{proof}
By \Cref{lem:rank-equality} and the assumptions, we have
\[
\frac{|\Tor (\Ker d^r_p)|}{|\im d^r_{p-r}|} = |\Tor E^r_p(C)|.
\]
This implies that $|\im d^r_{p-r}| = 1$, or equivalently
$d^r_{p-r} = 0$.
\end{proof}

Next, we discuss the {\it degeneration} of a spectral sequence.
For given $r_0 \geq 0$,
we say that a spectral sequence $E^r(C)$ {\it degenerates at $r_0$} (or {\it at the $E^{r_0}$-stage}) if $d^r_p = 0$ for any $r\geq r_0$ and $p \geq i_0$. 
Then, we can obtain an isomorphism between the $E^{r_0}$-page and the associated graded pieces as follows.
\begin{lem}\label{lem:convergence}
If $E^r(C)$ degenerates at $r_0$,
then the map 
 \[
 G_pH_*(C) \to E^{r_0}_p(C), \quad [[x]] \mapsto [x]^r_p
 \]
 is a well-defined isomorphism.
\end{lem}

\begin{proof}
For proving the well-definedness, suppose that 
$[[x]] = [[x']]$ in $G_pH_*(C)$ for cycles $x,x' \in F_pC$. Then there exist a cycle $y \in F_{p+1}C$
and a chain $z \in C$ such that 
$x-x'=y+ dz$.

Let $k$ be the maximal integer satisfying $F_kC \ni z$.  
If $k \geq p-r_0+1$, then 
$[x]^{r_0}_p-[x']^{r_0}_p=[y+ dz]^{r_0}_p = 0$ in $E^{r_0}_p(C)$. Suppose $k \leq p -r_0$. Since $dz = x - x' -y \in F_pC$, we see that
\[
[dz]^{p-k}_p = d^{p-k}_k([z]^{p-k}_k) \in E^{p-k}_p(C).
\]
Here, since $p-k \geq r_0$, we have $d^{p-k}_k=0$, and hence 
$[dz]^{p-k}_p=0$. This implies an equality
\[
dz = y' + d z',
\]
where $y' \in F_{p+1}C$ and $z' \in F_{k+1}C$.
In particular, we have $x-x' = (y+y') +dz'$.
Applying this argument repeatedly, we have chains $y'' \in F_{p+1}C$ and $z'' \in F_{p-r_0+1}C$ with $x-x' = y'' + dz''$, which proves $[x]^{r_0}_p=[x']^{r_0}_p$.

It is easy to check that the map is an injective homomorphism.
We prove the surjectivity. Let $x \in F_pC$ be a chain with $dx \in F_{p+r_0}C$. Then, since $d^{r_0}_p=0$, we have
\[
[dx]^{r_0}_{p+r_0} = d^{r_0}_p([x]^{r_0}_p) = 0 \in E^{r_0}_{p+r_0}(C).
\]
This implies an equality $dx = y + dz$, where 
$y \in F_{p+r_0+1}C$ and $z \in F_{(p+r_0)-r_0+1}C = F_{p+1}C$.
In particular, $x' := x - z \in F_pC$ is a chain with $dx' \in F_{p+r_0+1}C$ and $[x]^{r_0}_p - [x']^{r_0}_p = [z]^{r_0}_p = 0$ in $E^{r_0}_p(C)$. Applying this argument repeatedly, we have a chain $x'' \in F_pC$ with $dx'' \in F_{i_1+1}C=\{0\}$ and 
$[x'']^r_p =[x]^r_p$.
Now we see that $[[x'']] \in G_pH_*(C)$ is mapped to 
$[x]^r_p \in E^r_p$.
\end{proof}

In our setting, 
any spectral sequence degenerates at $i_1-i_0+1$.
Moreover, if $E^r(C)$ degenerates at $r_0$,
then $E^r_p$ is canonically isomorphic to $G_pH_*(C)$ for any $r \geq r_0$. In light of these facts,  
$G_pH_*(C)$ is often denoted by $E^\infty_p(C)$
and
$E^\infty(C):=\bigoplus_{p \geq i_0} E^\infty_p(C)$
is called {\it the limit of the spectral sequence}.  

If $R = \bbF$ is a field, then we have an (uncanonical) isomorphism 
$E^\infty(C) \cong H_*(C)$.
Indeed, there exist a basis $\{v_i\}_{i=1}^n$ for $H_*(C)$ 
and a sequence $0 \leq n_{i_1} \leq n_{i_1-1} \cdots \leq n_{i_0}=n$
such that $\{v_i\}_{i=1}^{n_p}$ is a basis for $F_pH_*(C)$, and this gives a basis $\{[v_i]\}_{n_{p+1} \leq i \leq n_{p}}$ for
$E^\infty_p(C)=G_pH_*(C)=F_pH_*(C)/F_{p+1}H_*(C)$.
Note that this isomorphism is highly dependent on the choice of the basis $\{v_i\}$.
For a general principal ideal domain $R$, we have the following two lemmas.
\begin{lem}
\label{lem:rank-infty}
The equality $\rank_R E^\infty(C) = \rank_R H_*(C)$ holds.
\end{lem}

\begin{proof}
By the equality (\ref{eq:quotient}), we have
\begin{align*}
\rank_R E^\infty (C) &= \sum_{i_0 \leq p \leq i_1}
\rank_R G_pH_*(C)\\[2mm]
&= \sum_{i_0 \leq p \leq i_1} \rank_R F_pH_*(C)
- \sum_{i_0 \leq p \leq i_1} \rank_R F_{p+1}H_*(C)\\[2mm]
&= \rank_R F_{i_0}H_*(C)- \rank_R F_{i_1 + 1}H_*(C) \\[2mm]
&= \rank_R H_*(C).
\qedhere
\end{align*}
\end{proof}

\begin{lem}
\label{lem:tor-infty}
If $\Tor H_*(C)$ is finite, then
the inequality $|\Tor E^\infty(C)| \geq |\Tor H_*(C)|$ holds.
Moreover, if $|\Tor E^\infty(C)| = |\Tor H_*(C)|$ holds,
then for any $p \geq i_0$, any torsion element of $G_pH(C_*)$ is represented by a torsion element of $H_*(C)$.
\end{lem}

\begin{proof}
To prove the lemma, we construct an injective map $f \colon \Tor H_*(C) \to \Tor E^\infty(C)$.
First, consider the map
\[
\Mdeg \colon \Tor H_*(C) \to \{i_0 \leq i \leq i_1\}, \quad
x \mapsto \max\{i \mid x \in F_iH_*(C) \}.
\]
and set $\calG_i := \Mdeg^{-1}(i)$ for each $i_0 \leq i \leq i_1$, which gives a partition
\[
\Tor H_*(C) = \bigsqcup_{i_0 \leq i \leq i_1} \calG_i.
\]
In addition, we fix an order of the elements $x_{i1}, x_{i2}, \ldots, x_{in_i}$ of $\calG_i$. 
Then, for each $i_0 \leq i, j \leq i_1$, we first define a map 
\[
f_{ij} \colon \calG_i \to \Tor G_jH_*(C)
\]
inductively. For $i=i_1$, since $G_{i_1}H_*(C)=F_{i_1}H_*(C)$,
we can define $f_{i_1j}$ by
\[
f_{i_1 j} (x) := 
\begin{cases}
0 \in G_jH_*(C) &(i_0 \leq j < i_1)\\[2mm]
x \in G_{i_1}H_*(C) &(j = i_1)
\end{cases}.
\]
Next, assume that $f_{s j}$ is defined for any $i <s \leq i_1$ and $i_0 \leq j \leq i_1$, and define $f_{ij}$ for any $i_0 \leq j \leq i_1$. Define maps
\[
\sigma_i \colon \{1, 2, \ldots, n_i \} \to \{1, 2, \ldots, n_i \}, \quad \sigma_i(k) := \min \{l \mid [x_{il}] = [x_{ik}] \text{ in } G_iH_*(C)  \}
\]
and
\[
\tau_i \colon \{1,2, \ldots, n_i \} \to
\{ i_0 \leq i \leq i_1\}, \quad
\tau_i(k) := \Mdeg(x_{ik}-x_{i\sigma_i(k)}).
\]
(Note that $\sigma_i(k) \leq k$ and $\tau_i(k)>i$.)
Then the map $f_{ij}$ is defined by
\[
f_{i j} (x_{ik}) := 
\begin{cases}
0 \in G_jH_*(C) &(i_0 \leq j < i)\\[2mm]
[x_{ik}] \in G_iH_*(C) &(j = i)\\[2mm]
f_{\tau_i(k) j}(x_{ik} - x_{i \sigma_i(k)}) & (i<j\leq i_1)
\end{cases}.
\]
(Note that $[x_{ik}] \neq 0$ in $G_iH_*(C)$.)
Now we define 
\[
f_i := \bigoplus_{i_0 \leq j \leq i_1} f_{ij}
\colon \calG_i \to \bigoplus_{i_0 \leq j \leq i_1} 
\Tor G_iH_*(C) = \Tor E^\infty(C)
\]
and
\[
f := \bigsqcup_{i_0 \leq i \leq i_1} f_i 
\colon \Tor H_*(C) = \bigsqcup_{i_0 \leq i \leq i_1} \calG_i \to \Tor E^\infty(C).
\]
Let us prove that $f$ is injective.
For given $x,x' \in \Tor H_*(C)$, suppose that $f(x)=f(x')$.
Here we note that for any $x \in \Tor H_*(C)$ with $\Mdeg(x)=i$, we have 
\[
f(x)=(f_{i, i_0}(x), f_{i, i_0 +1}(x), \ldots, f_{i, i_1}(x))
\in \bigoplus_{i_0 \leq j \leq i_1} 
\Tor G_iH_*(C) = \Tor E^\infty(C)
\]
and
\[
\min \{j \mid f_{ij}(x) \neq 0 \} = i = \Mdeg(x).
\]
These imply that $\Mdeg(x) = \Mdeg(x')$.
Thus, we may assume that
$x=x_{ik}$ and $x' = x_{il}$ for some $i$ and $1 \leq k, l \leq n_i$. 

Here we prove that $f(x_{ik})=f(x_{il})$ implies $x_{ik}=x_{il}$ by induction of $i$.
If $i=i_1$, then 
$f(x_{i_1k})=f(x_{i_1l})$ implies the equalities
\[
x_{i_1 k} = f_{i_1i_1}(x_{i_1 k}) = f_{i_1 i_1}(x_{i_1 l}) = x_{i_1 l}.
\]
Next, for given $i$, assume that the assertion holds for any $s$ with $i < s \leq i_1$, and prove the assertion for $i$. By the equality $f(x_{ik})= f(x_{il})$, we see that
\[
[x_{ik}] = f_{ii}(x_{ik}) = f_{ii}(x_{il}) = [x_{il}]
\]
as elements of $G_iH_*(C)$. Then, the definition of $\sigma_i$ shows 
$\sigma_i(k) = \sigma_i(l)$, which denoted by $m$. Hence we have
\[
f_{\tau_i(k)j}(x_{ik}-x_{im})
= f_{ij}(x_{ik}) = f_{ij}(x_{il})  = f_{\tau_i(l)j}(x_{il}-x_{im})
\]
for any $j$ with $i < j \leq i_1$.
Here, since $\Mdeg(x_{ik}-x_{im})>i$ and 
$\Mdeg(x_{il}-x_{im}) > i$, we also have
\[
f_{\tau_i(k)j}(x_{ik}-x_{im})
= 0  = f_{\tau_i(l)j}(x_{il}-x_{im})
\]
for any $j$ with $i_0 \leq j \leq i$. Consequently, we have $f(x_{ik}-x_{im}) = f(x_{il}-x_{im})$.
By the assumption of induction, this equality implies
$x_{ik} -x_{im} = x_{il} - x_{im}$, and hence 
$x_{ik} = x_{il}$. This proves 
$|\Tor E^\infty(C)| \geq |\Tor H_*(C)|$.

Finally, we prove the last-half assertion.
If $|\Tor E^\infty(C)| = |\Tor H_*(C)|$, then the above map $f$ must be bijective.
Here, note that for any element $y$ of $f(\Tor H_*(C))$,
each entry of $y$ (w.r.t.\ the direct decomposition $\Tor E^\infty(C) = \bigoplus_{i_0 \leq j \leq i_1} \Tor G_jH_*(C)$)
is represented by a torsion element of $H_*(C)$.
Since any element $y_p$ of $G_pH_*(C)$ is contained in $E^\infty(C)$ as the element $(y_{i_0},y_{i_0 +1}, \ldots, y_{i_1})$ with $y_j = \delta_{jp} y_p$ (where $\delta_{jp}$ denotes the Kronecker delta), this completes the proof.
\end{proof}

Now, we have the following sufficient condition for the degeneration. 
\begin{thm}
\label{thm:degeneration}
For given $r_0 \geq 0$, if $\Tor E^{r_0}(C)$ is finite and the equalities
\[
\rank_R E^{r_0}(C) = \rank_R H_*(C)
\quad \text{and} \quad
|\Tor E^{r_0}(C)| = |\Tor H_*(C)|
\]
hold, then $E^r(C)$ degenerates at $r_0$. 
\end{thm}

\begin{proof}
We prove that $E^{r}_p(C)$ is free and $d^r_p=0$ for each $r \geq r_0$ and $p \geq i_0$ by induction of $r$.
Here, we first note that for any $r \geq r_0$, 
the first assumption and \Cref{lem:rank-infty} show
\[
\rank_R E^{r_0}(C) \geq \rank_R E^r(C) \geq
\rank_R E^{r+1}(C) \geq \rank_R E^\infty(C) = \rank_R H_*(C) = \rank_R E^{r_0}(C),
\]
which imply $\rank_R E^r(C) = \rank_R E^{r+1}(C)$, and hence we have
\begin{align*}
\rank_R E^{r+1}_p(C) \leq \rank_R E^r_p(C)
&= \rank_R E^r(C) - \sum_{i \neq p} \rank_R E^r_i(C)\\[2mm]
&= \rank_R E^{r+1}(C) - \sum_{i \neq p} \rank_R E^r_i(C)\\[2mm]
&\leq \rank_R E^{r+1}(C) - \sum_{i \neq p} \rank_R E^{r+1}_i(C) = \rank_R E^{r+1}_p(C).
\end{align*}
These show $\rank_R E^{r}_p(C) = \rank_R E^{r+1}_p(C)$
for any $r \geq r_0$ and $p \geq i_0$.
Then, by \Cref{lem:rank-equality}, \Cref{lem:tor-infty} and the second assumption,
we have
\[
|\Tor E^{r_0}(C)| \geq |\Tor E^r(C)| \geq |\Tor E^{r+1}(C)|
\geq |\Tor E^\infty(C)| \geq |\Tor H_*(C)| = |\Tor E^{r_0}(C)|,
\]
and hence $|\Tor E^r(C)| = |\Tor E^{r+1}(C)|$.
(Here we note that the equalities 
\[
\left|\Tor E^r(C) \right| = \left|\bigoplus_{p \geq i_0} \Tor E^r_p(C) \right|
= \prod_{p \geq i_0} \left|\Tor E^r_p(C) \right|
\]
hold.)
Now, we see that for any $p \geq i_0$, 
\begin{align*}
|\Tor E^{r+1}_{p+r}(C)| \leq |\Tor E^r_{p+r}(C)|
&= \frac{|\Tor E^r(C)|}{\prod_{i \neq p+r} |\Tor E^r_i(C)|}\\[2mm]
&= \frac{|\Tor E^{r+1}(C)|}{\prod_{i \neq p+r} |\Tor E^r_i(C)|}\\[2mm]
&\leq \frac{|\Tor E^{r+1}(C)|}{\prod_{i \neq p+r} |\Tor E^{r+1}_i(C)|} 
= |\Tor E^{r+1}_{p+r}(C)|,
\end{align*}
and hence $|\Tor E^r_{p+r}(C)| = |\Tor E^{r+1}_{p+r}(C)|$.
By combining the above computations with \Cref{lem:tor-equality}, we have $d^r_p = 0$ for any $r \geq r_0$ and $p \geq i_0$.
\end{proof}

\subsection{Filtered maps}

Let $(C,d), (C',d')$ be filtered complexes.
Then $f \colon C \to C'$ is a {\it degree $k$ filtered map}
if $f(F_pC) \subset F_{p+k}C'$ for any $p \geq i_0$. 
Since the induced map $f_* \colon H_*(C) \to H_*(C')$ satisfies
$f_*(F_pH_*(C)) \subset F_{p+k}H_*(C')$, we have an induced map
\[
G_pf_* \colon G_pH_*(C) \to G_{p+k}H_*(C'),
\quad [[x]] \mapsto [f_*([x])].
\]
If two degree $k$ filtered chain maps $f,g$ are chain homotopic, then obviously $G_pf_*=G_pg_*$ holds for any $p \geq i_0$. 

Next, we discuss induced maps on $E^r$-complexes.
For $l \geq 0$,
two degree $k$ filtered chain maps $f,g \colon C \to C'$
are {\it chain homotopic in degree $\geq -l$} if there exists a degree $k-l$ filtered chain homotopy $\Phi \colon C \to C'$ from $f$ to $g$. 
(Regarding to the $\Z/4$-grading, we require that chain maps preserve the grading and chain homotopies shift the grading by 1.)
Two filtered complexes are {\it filtered chain homotopy equivalent in degree $\geq -l$} if there exists a degree $0$ chain homotopy equivalence map whose chain homotopies have degree $-l$.
(If $l=0$, then we simply say that two filtered complexes are filtered chain homotopy equivalent.)
Here we enumerate basic lemmas about filtered chain maps.

\begin{lem} \label{lem:spectral-morphism}
    For any degree $k$ filtered chain map $f \colon C \to C'$,
    the map 
    \[
    f^r_p \colon E^r_p(C) \to E^r_{p+k}(C'), [x]^r_p \mapsto [f(x)]^r_p
    \]
    is a well-defined chain map, and the diagram
        \begin{align} \label{eq:chain-map-diag}
    \begin{split}
    \xymatrix{
        E^{r+1}_p(C)
        \ar[r]^-{f^{r+1}_p}
        \ar[d]_-{\cong}
        & E^{r+1}_{p+k}(C')
        \ar[d]^-{\cong}\\
        H_*((E^r_p(C))
        \ar[r]^-{(f^r_p)_*}
        & H_*(E^r_{p+k}(C'))
    }
    \end{split}
    \end{align}
    is commutative for any $r \geq 0$ and $p \geq i_0$, where the vertical maps are given by $(\ref{eq:spectral-homology})$.
    Moreover, if degree $k$ filtered chain maps $f,g \colon C \to C'$ are chain homotopic in degree $\geq -l$, then $f^{l}=\bigoplus_{p \geq i_0} f^{l}_p$ and $g^{l}=\bigoplus_{p \geq i_0} g^{l}_p$ are chain homotopic, and $f^r_p=g^r_p$ for any $r>l$ and $p \geq i_0$.
\end{lem}

\begin{proof}
It is easy to check that $f^r_p$ is a well-defined chain map
and the diagram $(\ref{eq:chain-map-diag})$.
Suppose that $f,g$ are chain homotopic in degree $\geq -l$.
Then we have a degree $k-l$ filtered map $\Phi \colon C \to C'$
satisfying $f-g = \Phi \circ d + d' \circ \Phi$.
Here we prove that for any $p \geq i_0$, the map 
\[
\Phi^{l}_p \colon E^{l}_p(C) \to E^{l}_{p+k-l}(C'), \quad
[x]^{l}_{p} \mapsto [\Phi(x)]^{l}_{p+k-l}
\]
is well-defined. To prove this, take $x,x' \in F_pC$ satisfying
$dx,dx' \in F_{p+l}C$ and $x-x' = y + dz$ for some $y \in F_{p+1}C$ and $z \in F_{p-l+1}C$. Then we see that
$\Phi(x),\Phi(x') \in F_{p+k-l}C$, and
\begin{align*}
    \Phi(x) - \Phi(x') 
    &= \Phi(y) + \Phi\circ d(z)\\[2mm]
    &= \Phi(y) +f(z) - g(z) - d' (\Phi(z))
    \in F_{p+k-l+1}C' + d'(F_{p+k-2l+1}C').
\end{align*}
Now we show that $\Phi^{l}_p$ is a chain homotopy $f^{l} \Rightarrow g^{l}$. To see this, 
for any $p \geq i_0$, take a chain $x \in F_pC$ satisfying $dx \in F_{p+l}C$. Then we see
\begin{align*}
    f^{l}_p([x]^l_p) - g^{l}_p([x]^l_p) =  [f(x) -g (x)]^{l}_{p+k}
    &= [\Phi \circ d(x) + d' \circ \Phi (x)]^{l}_{p+k}\\[2mm]
    &= \Phi^{l}_{p+l}([dx]^{l}_{p+l}) + d'^{l}_{p+k-l}([\Phi(x)]^{l}_{p+k-l})\\[2mm]
    &= (\Phi^{l}_{p+l}\circ d^{l}_p+ d'^{l}_{p+k-l}\circ \Phi^{l}_{p})
    ([x]^{l}_p). 
\end{align*}
The last assertion in \Cref{lem:spectral-morphism}
follows from the diagram $(\ref{eq:chain-map-diag})$.
\end{proof}

\begin{lem}
\label{lem:spectral-functor}
The following assertions hold;
\begin{itemize}
    \item[{\rm(i)}]
    For any filtered complex $C$, $r \geq 0$ and $p \geq i_0$, we have $(\id_C)^r_p = \id_{E^r_p(C)}.$
   \item[{\rm(ii)}]
    Let $f \colon C \to C'$ and $g \colon C' \to C''$ be
    degree $k$ and $k'$ filtered chain maps respectively.
    Then, $g \circ f \colon C \to C''$ is a degree $(k+k')$ filtered chain map, and $(g \circ f)^r_p = g^r_{p+k} \circ f^r_{p}$ for any $r \geq 0$ and $p \geq i_0$.
    \end{itemize}
\end{lem}

\begin{proof}
The proof is elementary.
\end{proof}

\begin{lem}
\label{lem:degeneration-map}
If $E^r(C)$ degenerates at $r_0$, then the diagram
    \begin{align} \label{eq:convergence-diag}
    \begin{split}
    \xymatrix{
        G_pH_*(C)
        \ar[r]^-{G_pf_*}
        \ar[d]_-{\cong}
        & G_{p+k}H_*(C')
        \ar[d]^-{\cong}\\
        E^{r_0}_p(C)
        \ar[r]^-{f^{r_0}_p}
        & E^{r_0}_{p+k}(C')
    }
    \end{split}
    \end{align}
    is commutative for $p \geq i_0$, where the vertical maps are given by \Cref{lem:convergence}.
\end{lem}
\begin{proof}
For any cycle $x \in F_pC$, we see that
\begin{align*}    
[[x]] \in G_pH_*(C) 
&\mapsto G_pf_*([x])=[f_*([x])] =[[f(x)]] \in G_{p+k}H_*(C') \\[2mm]
&\mapsto [f(x)]^r_{p+k} \in E^r_{p+k}(C')
\end{align*}
and
\begin{align*}
[[x]] \in G_pH_*(C) 
& \mapsto [x]^r_p \in E^r_{p}(C) \\[2mm]
&\mapsto f^r_p([x]^r_{p}) = [f(x)]^r_{p+k} \in E^r_{p+k}(C').
\qedhere
\end{align*}
\end{proof}

\begin{lem}
\label{lem:degeneration-equiv}
For given $r_0>0$, if two filtered chain complexes $C, C'$ are filtered chain homotopy equivalent, then  $E^r(C)$ degenerates at $r_0$ if and only if
$E^r(C')$ degenerates at $r_0$.
\end{lem}

\begin{proof}
Since $C,C'$ can be exchanged, we only need to prove that if $E^r(C)$ degenerates at $r_0$, then $E^r(C')$ also degenerates at $r_0$. 

Let $f \colon C \to C'$ be a filtered chain homotopy equivalence map and $g$ be the homotopy inverse of $f$.
Then, it follows from \Cref{lem:spectral-morphism} and \Cref{lem:spectral-functor} that for any $r \geq r_0>0$,
we have
\[
g^r_p \circ f^r_p = (g \circ f)^r_p = \id_{E^r_p(C)}
\quad \text{and} \quad
f^r_p \circ g^r_p = (f\circ g)^r_p = \id_{E^r_p(C')}.
\]
This shows that $f^r_p \colon E^r_p(C) \to E^r_p(C')$
is a chain isomorphism. Now, $d^r_p=0$ implies $d'^r_p = 0$, which completes the proof. 
\end{proof}

\subsection{Tensor product}
For given two filtered complex $(C,d)$ and $(C',d')$, we define a filtration on the tensor product $(C\otimes C', d^{\otimes} = d\otimes 1 + (-1)^{\gr} \otimes d')$ by 
\[
(C \otimes C')_i := \bigoplus_{j_1+j_2 = i} C_{j_1} \otimes C_{j_2}
\]
for any $i \geq i_0+i'_0$. (Note that $\gr$ is the $\Z/4$-grading on $C$ and independent of the filtration level $i$.)
In addition, we can also consider the tensor product of their $E^r$-complexes 
$(E^r(C)\otimes E^r(C'), (d^r)^{\otimes}= d^r\otimes 1 + (-1)^{\gr}\otimes d'^r$,
which admits a direct sum decomposition
\[
E^r(C)\otimes E^r(C')
= \bigoplus_{p \geq i_0 + i'_0} (E^r(C)\otimes E^r(C'))_p,
\]
where
\[
(E^r(C)\otimes E^r(C'))_p := \bigoplus_{j_1+j_2=p}E^r_{j_1}(C)\otimes E^r_{j_2}(C')
\]
and $(d^r)^{\otimes}_p := (d^r)^{\otimes}|_{(E^r(C)\otimes E^r(C'))_p}$ satisfies $\im (d^r)^{\otimes}_p \subset (E^r(C)\otimes E^r(C'))_{p+r}$. 

\begin{lem} \label{lem:tensor-spectral}
    The map
    \[
    T^r_p\colon
    (E^r(C) \otimes E^r(C'))_p \to E^r_p(C \otimes C'),
    \quad \sum_{j_1 + j_2 = p}[x_{j_1}]^r_{j_1} \otimes [y_{j_2}]^r_{j_2}
    \mapsto \sum_{j_1 + j_2 = p}
    \left[ x_{j_1}\otimes y_{j_2}\right]^r_p
    \]
    is a well-defined chain map. Moreover, the diagram
    \begin{align} \label{eq:tensor-diag}
    \begin{split}
    \xymatrix{
        (E^{r+1}(C)\otimes E^{r+1}(C') )_p
        \ar[r]^-{T^{r+1}_p}
        \ar[d]
        & E^{r+1}_p(C\otimes C')
        \ar[d]^-{\cong}\\
        H_*((E^r(C)\otimes E^r(C'))_p)
        \ar[r]^-{(T^r_p)_*}
        & H_*(E^r_p(C \otimes C'))
    }
    \end{split}
    \end{align}
    is commutative, where the right-hand vertical map is given by $(\ref{eq:spectral-homology})$, and the left-hand vertical map is given by 
    \[
    \xymatrix@R=1pt{
    (E^{r+1}(C)\otimes E^{r+1}(C') )_p \ar[r]^-{\cong}&
    {\displaystyle \bigoplus_{j_1+j_2=p} H_*(E^r_{j_1}(C)) \otimes H_*(E^r_{j_2}(C'))} \ar[r] &
    H_*((E^r(C)\otimes E^r(C'))_p) \\
    \rotatebox{90}{$\in$} &
    \rotatebox{90}{$\in$} &
    \rotatebox{90}{$\in$} \\
    {\displaystyle \sum_{j_1 + j_2 = p} [x_{j_1}]^{r+1}_{j_1}\otimes[y_{j_2}]^{r+1}_{j_2}}
    \ar@{|->}[r] &
    {\displaystyle \sum_{j_1 + j_2 = p}
    [[x_{j_1}]^{r}_{j_1}]\otimes[[y_{j_2}]^r_{j_2}]}
    \ar@{|->}[r] &
    {\displaystyle \sum_{j_1 + j_2 = p}
    [[x_{j_1}]^{r}_{j_1}\otimes[y_{j_2}]^r_{j_2}]}
    }
    \]
\end{lem}

\begin{rem}
The left-hand vertical map in the diagram $(\ref{eq:tensor-diag})$ 
is the composition of the tensor product of the chain isomorphisms given by $(\ref{eq:spectral-homology})$
and an injective map appearing in the K\"{u}nneth formula.
\end{rem}

\begin{proof}
To prove the well-definedness of $T^r_p$, take $x'_{j_1} \in F_{j_1}C$ with $dx'_{j_1} \in F_{j_1+r}C$ and $[x'_{j_1}]^{r}_{j_1}=[x_{j_1}]^{r}_{j_1}$.
Then we have chains $u \in F_{j_1+1}C$ and $v \in F_{j_1-r+1}C$
satisfying $x'_{j_1} - x_{j_1} = u + dv$.
Now we see
\begin{align*}
    x'_{j_1}\otimes y_{j_2} - x_{j_1}\otimes y_{j_2}
    &= u\otimes y_{j_2} + dv\otimes y_{j_2}\\[2mm]
    &= (u \otimes y_{j_2} - (-1)^{\gr(v)} v \otimes d'y_{j_2})
    + d^{\otimes}(v\otimes y_{j_2})\\[2mm]
    &\in F_{p+1}(C\otimes C') + d^{\otimes}(F_{p-r+1}(C\otimes C')).
\end{align*}
Hence $[x'_{j_1}\otimes y_{j_2}]^r_p = [x_{j_1}\otimes y_{j_2}]^r_p$. similarly, we see that
any choice of a representative of $[y_{j_2}]$ does not affect
the image either. 

We next show $T^r_p \circ (d^r)^\otimes_p = (d^{\otimes})^r_p \circ T^r_p$. Indeed, we see
\begin{align*}    
T^r_p \circ (d^r)^\otimes_p([x_{j_1}]^r_{j_1} \otimes [y_{j_2}]^r_{j_2}) 
&= T^r_p \left(d^r([x_{j_1}]^r_{j_1})\otimes [y_{j_2}]^r_{j_2} 
+ (-1)^{\gr(x_{j_1})} [x_{j_1}]^r_{j_1}\otimes d'^r([y_{j_2}]^r_{j_2})\right)\\[2mm]
&= T^r_p \left([dx_{j_1}]^r_{j_1-r}\otimes [y_{j_2}]^r_{j_2} 
+ (-1)^{\gr(x_{j_1})} [x_{j_1}]^r_{j_1}\otimes [d'y_{j_2}]^r_{j_2-r}\right)\\[2mm]
&= [dx_{j_1}\otimes y_{j_2} 
+ (-1)^{\gr(x_{j_1})} x_{j_1}\otimes d'y_{j_2} ]^r_{p-r} 
\end{align*}
and
\begin{align*}
(d^{\otimes})^r_p \circ T^r_p 
([x_{j_1}]^r_{j_1} \otimes [y_{j_2}]^r_{j_2})
&= (d^{\otimes})^r_p ([x_{j_1} \otimes y_{j_2}]^r_p)\\[2mm]
&= [dx_{j_1}\otimes y_{j_2} 
+ (-1)^{\gr(x_{j_1})} x_{j_1}\otimes d'y_{j_2} ]^r_{p-r}.
\end{align*}
The proof of the commutativity of $(\ref{eq:tensor-diag})$
is also straightforward.
\end{proof}

\begin{lem}
    $T^0 = \bigoplus_{p \geq i_0+i_1} T^0_p$ is a chain isomorphism.
    Moreover, if either $E^1(C)$ or $E^1(C')$ is free over $R$,
    then $T^1 = \bigoplus_{p \geq i_0+i_1} T^1_p$ is also a chain isomorphism.
\end{lem}

\begin{proof}
Note that $E^0_p(D) = D_p$ for any filtered complex $D$,
and hence $T^0_p$ coincides with the identity.
Next, if either $E^1(C)$ or $E^1(C')$ is free over $R$,
then the left-hand vertical map in $(\ref{eq:tensor-diag})$
is an isomorphism since 
$\Tor^R_1(H_*(E^0(C)),H_*(E^0(C'))) =
\Tor^R_1(E^1(C),E^1(C'))=0$, while
it follows from the above argument that
the bottom map $(T^0_p)_*$ is also an isomorhism. 
Now, the commutativity of $(\ref{eq:tensor-diag})$
completes the proof.
\end{proof}

Next, we discuss the behavior of filtered maps under tensor products.
\begin{lem} \label{lem:tensor-map}
Let $f^{(')} \colon C^{(')} \to D^{(')}$ be a degree $k^{(')}$ filtered chain map. Then, $f \otimes f' \colon C \otimes C' \to D \otimes D'$ is a degree $(k+k')$ filtered chain map. Moreover, if two degree $k^{(')}$ filtered chain maps $f^{(')},g^{(')} \colon C^{(')} \to D^{(')}$ are chain homotopic in degree $\geq -l^{(')}$, then $f \otimes f'$ and $g \otimes g'$ are chain homotopic in degree $\geq \min\{-l,-l'\}$. 
\end{lem}

\begin{proof}
This immediately follows from the fact that
$F_{p}(C\otimes C')= \sum_{j_1+j_2=p}F_{j_1}C \otimes F_{j_2}C'$ for any $p \geq i_0 + i'_0$.
(Note that if chain homotopies $\Phi^{(')} \colon f^{(')} \Rightarrow g^{(')}$ are given, then $\Phi\otimes f'+(-1)^{\gr} g \otimes \Phi'$ is a chain homotopy $f\otimes f' \Rightarrow g \otimes g'$, where $\gr$ is the $\Z/4$-grading on $C \otimes C'$.)
\end{proof}

As corollaries of \Cref{lem:tensor-map}, we have the following two lemmas.

\begin{lem} \label{lem:tensor-equivalence}
If $C$ and $C'$ are filtered chain homotopy equivalent to $D$ and $D'$ in degree $\geq -l$ respectively, then 
$C \otimes C'$ is chain homotopy equivalent to $D \otimes D'$ in degree $\geq -l$.
\end{lem}

\begin{lem} \label{lem:tensor-spectral-map}
Let $f^{(')} \colon C^{(')} \to D^{(')}$ be a degree $k^{(')}$ filtered chain map. Then, the diagram
    \begin{align*}
    \begin{split}
    \xymatrix{
        (E^{r}(C)\otimes E^{r}(C') )_p
        \ar[r]^-{f^r\otimes f'^r}
        \ar[d]^-{T^{r}_p}
        &(E^{r}(D)\otimes E^{r}(D') )_{p+k+k'}
        \ar[d]^-{T'^{r}_p}\\
        E^{r}_p(C\otimes C')
        \ar[r]^-{(f\otimes f')^r_p}
        & E^{r}_{p+k+k'}(D\otimes D')
    }
    \end{split}
    \end{align*}
is commutative for any $r \geq 0$ and $p \geq i_0+i'_0$,
where the vertical maps are given in \Cref{lem:tensor-spectral}.
\end{lem}

\printbibliography
\printaddresses

\end{document}